\documentclass[10pt,a4paper]{amsart}

\usepackage{amssymb}
\usepackage[final]{showkeys}
\usepackage{microtype}
\usepackage{color}
\usepackage{graphicx}
\usepackage[hmargin=3cm,vmargin={5cm,4cm}]{geometry}

\newtheorem{te}{Theorem}[section]
\newtheorem{example}{Example}
\newtheorem{definition}[te]{Definition}
\newtheorem{os}[te]{Remark}
\newtheorem{prop}[te]{Proposition}
\newtheorem{lem}[te]{Lemma}
\newtheorem{coro}[te]{Corollary}

\numberwithin{equation}{section}

\allowdisplaybreaks

\begin{document}
	
	\title[]{S-invariant and S-multinvariant functions and some symmetry groups of algebraic sieves }
	\author{Francesco Maltese}
\maketitle

\begin{abstract}
	In this article we introduced algebraic sieves, i.e. selection procedures on a given finite set to extract a particular subset. Such procedures are performed by finite groups acting on the set. They are called sieves because there are certain sets of numbers which, with appropriate groups, can select, for example, a set of primes, think of the famous Eratosthenes sieve. In this article we have given a general definition of algebraic sieves. And we also introduced the notion of invariant and multi-invariant functions, certain permutations on the sieve set which, in the invariant case, commute with the action of a given sieve-selecting group and the automorphism of that group, and multi-invariants which commute with all groups and their respective automorphisms. By means of such functions we have given symmetries on such sieves. In particular, we studied certain groups of symmetries of invariant functions. Then, using such notions, we studied a particular example, the Goldbach sieve, where the selector groups are dihedral groups and the selected set consists of the primes and, in some cases, the numbers $1$ and $N-1=p$, with $N$ being even and p prime, which satisfies the Goldbach conjecture for $N$. We have shown that one of these symmetry groups is isomorphic to a subgroup or affinity group of the ring of integers modulo N with N an integer even $\mathbb{Z}_{N}$.
	
	\bigskip
	
		\textit{Algebraic Sieve, Invariant function, Multinvariant function, Sieve symmetry groups, Dihedral Sieve, Goldbach Sieve.	}
		
	\end{abstract}

\section*{Contents}

\subsection*{1.Introduction ...............................................................................................................................2}
\subsection*{1.1.S-invariant functions.................................................................................................................2}
\subsection*{2.Permutation Induced  ..................................................................................................................6}
\subsection*{3.The algebraic structure of $\widehat{{Aut(G)}^{S}}$ ..............................................................................................9}
\subsection*{3.1.Examples of $\widehat{{Aut(G)}^{S}}$................................................................................................................19}

\subsection*{4.Algebraic Sieves .........................................................................................................................27}

\subsection*{4.1.Example of algebraic sieves
	.....................................................................................................28}
\subsection*{5.Multi-invariant functions
	...........................................................................................................30}

\subsection*{6.Some symmetry groups of an algebraic sieve
	..............................................................................31}

\subsection*{7.The group $\widehat{Aut(D_n)}^{\mathcal{C}(D_{p,q}^{2n}; x_{0,q} )}$ 
	.......................................................................................................31}

\subsection*{8.Methods of calculation for the group $\mathcal{G}_{N}$......................................................................................35}

\subsection*{9.Final conclusions and conjectures................................................................................................42}

\subsection*{References......................................................................................................................................43}

	\section{Introduction}
	The following article will analyse the concept of numerical sieves, which are methods for selecting a specific numerical subset from a given list of numbers. For instance, the sieve constructed by Eratosthenes (see [9] in the bibliography) is a notable and historically significant example. This sieve is able to select all prime numbers less than a specified integer, denoted by $N$. However, we will not employ, in the study of sieves, the analytic number theory approach used in [9] and many other texts on analytic number theory. Instead, we will adopt a perspective from finite group theory, which allows us to take a more algebraic approach. This is the reason for the name $Algebraic \hspace{0.1cm} Sieve$, and it is precisely through the lens of group actions that we will examine the sieve and this will be examined in greater detail in Section 3. We will therefore attempt to abstract the concept of the sieve in a purely insiemistic context. From this perspective, the aim is to study some of the symmetries of algebraic sieves, which will be discussed in Section 4.
	
	In concrete terms, the first step will be to analyse groups of symmetries generated by particular permutations on the set in which the sieve operates. These permutations must commute the action of a group that selects the elements of the sieve, together with its automorphism. Such permutations are called $S-invariants$. For details of this, see an extensive study dealt with in this article in sections 1, 2 and 3. Subsequently, groups of permutations that consider all the actions of groups involving the sieve are referred to as $S$-multinvariant functions, studied in section 5. For further details, refer to Section 6. In particular, Section 6 will demonstrate that it may be more convenient to study $S$-invariant functions than $S$-multinvariant functions, as they are more effective at obtaining information about certain sieves.
	The prerequisites for these latter topics are the basics of group theory, which are outlined in the bibliography from  [1] to [6]. 
	
	\quad
	
	Section 4 will present an example of an algebraic sieve, designated a dihedral sieve due to the action of dihedral groups on its set. For further insight into dihedral groups, see reference [1] in the bibliography. The definition of a dihedral sieve can be found in definition 4.4. In particular, this section will focus on the Goldbach sieve (see Definition 4.5) on the set of all integers modulo N, where N is a certain even integer.
	
	\quad
	
	In Section 7, we will undertake a detailed examination of the structure of the Goldbach dihedral sieve symmetry group,which has been induced by the action of the dihedral group  responsible for selecting the even and odd numbers of $\mathbb{Z}_N$ for a specific even integer $N$  through its action on the sieve.
	
	\quad
	
	Section 8 presents an attempt to provide non-general numerical methods for calculating the group studied in Section 7 and to construct algorithms for these methods. Section 9 presents an attempt to reformulate Goldbach's strong conjecture in the language of algebraic sieves via the symmetry group studied in Sections 7 and 8.

	\subsection{S-invariant functions}

	\begin{definition}
  Let us consider a finite set not empty $X$ and bijective application  $f: X \rightarrow X$. We shall assume that a finite group $G$ acts on $X$ in the following way:
  
  $$	G\times X \longrightarrow X $$ 
  	
  $$	(g,x)\longrightarrow g \cdot x$$

and let be a $\phi $ an automorphism of $G$.

We call $f$ an application $\phi$-invariant if $\forall g \in G $ and $\forall x \in X$, we get $f(g\cdot x)=\phi(g)\cdot f(x)$

	\end{definition}

It can note that if $f$ is $\phi$-invariant then we get $f(Orb_{G} (x_{0})) = Orb_{G} (f(x_{0}))$. In fact , being $f$ bijective, $f(x)=f(y) \iff x=y$ and since that 
$g\cdot x_{0} =g_{1} \cdot x_{1} \iff \phi(g)\cdot f( x_{0}) =\phi( g_{1}) \cdot f( x_{1}) $ then ${g_{1}}^{-1} g\cdot x_{0}=x_{1}  \iff 
{\phi(g_{1})}^{-1} \phi (g ) \cdot f( x_{0} )= f( x_{1} ) $.
 
In this case we can review $f$ as, once fixed $ x_{0}  \in Orb_{G} (x_{0} )$, $f(g\cdot x_{0} )=\phi(g)\cdot f(x_{0} )$. And we can do it for every $Orb_{G}(x)$.
	
So we can redefine a $f$ as   $\phi-invariant$ in this way
	\begin{definition}	
Fixed $x_{0}^{i} \in X$ for $i=1,2,\dots,t$  so that $X=\bigsqcup_{i=1}^{t} Orb_{G}(x_{0}^{i})$. A bijective application  $f: X \rightarrow X$ is said to be a $\phi-invariant$ if,  $\forall g \in G$, we have that $f(g\cdot x_{0}^{i})=\phi(g)g_{ij}\cdot x_{0}^{j}$ if $f(x_{0}^{i})=g_{ij}\cdot x_{0}^{j}$ for $i=1,2,\dots,t$ and for  appropiate 
$j \in {1,2,\dots,t}$ and $g_{ij} \in G$.

\end{definition}

Using the definition (1.2) we can note that $ g_{1}\cdot x_{0}^{i}=g_{2}\cdot x_{0}^{i} \iff f(g_{1}\cdot x_{0}^{i})=f(g_{2}\cdot x_{0}^{i}) \iff
 \phi( g_{1})g_{ij}\cdot x_{0}^{j}=\phi(g_{2})g_{ij}\cdot x_{0}^{j}$ for which $ g_{2}^{-1}g_{1}\cdot x_{0}^{i}=x_{0}^{i} \iff \phi( g_{2}^{-1}g_{1})g_{ij}\cdot x_{0}^{j}=g_{ij}\cdot x_{0}^{j}$ .

Thus $g_{2}^{-1}g_{1} \in Stab_{G}(x_{0}^{i}) \iff \phi(g_{2}^{-1}g_{1}) \in Stab_{G}(g_{ij}\cdot  
 x_{0}^{j})=g_{ij}Stab_{G}(x_{0}^{j})g_{ij}^{-1}=Stab_{G}(x_{0}^{j})^{g_{ij}}$.

Therefore we found the following necessary condition for $\phi$ automorphism of a $\phi$-invariance function. Then we have the following proposition.

\begin{prop}
If f $\in Sym(X)$ (i.e set of bijective applications from X to itself) is a $\phi-invariant$ then $\phi(Stab_{G}(x_{0}^{i}))=Stab_{G}(x_{0}^{j})^{g_{ij}}$ fixed 
$x_{0}^{i}$ for $i=1,2,\dots,t$ 
\end{prop}

Now let’s see how the inverse of a $\phi$-invariant function  $f$ can be represented as an appropriate function  $ \phi^{-1}$-invariant  $f^{-1}$.

\begin{prop}
The inverse function $f^{-1}$ of a  $\phi$-invariant function  $f$ can be represented as an function  $ \phi^{-1}$-invariant, fixed  $x_{0}^{j} \in X$ for $j=1,2,\dots,t$. 

\begin{equation}
	f^{-1}(g\cdot x_{0}^{j}  )= \phi^{-1}(g)\phi^{-1}(g_{ij}^{-1})\cdot x_{0}^{i}
\end{equation}

if $f(x_{0}^{i})=g_{ij}\cdot x_{0}^{j}$ for appropriate $ i \in {1,2,\dots,t}$ and $ g_{ij} \in G$. In addition you have
that  $f^{-1}( x_{0}^{j}  )= \phi^{-1}(g_{ij}^{-1})\cdot x_{0}^{i}$.
\end{prop}	
\begin{proof}
Let's see that function of (1.1) satisfies the following conditions $ f^{-1}f=\mathbb{I}_{X}$ and $ f f^{-1}=\mathbb{I}_{X}$.

$$ f^{-1}f(g\cdot x_{0}^{i}  )=f^{-1}(\phi(g)g_{ij} x_{0}^{j})=\phi^{-1}(\phi(g)g_{ij})\phi^{-1}(g_{ij}^{-1})\cdot x_{0}^{i}$$

$$                                                           = g \phi ^{-1}(g_{ij})\phi^{-1}(g_{ij}^{-1})\cdot x_{0}^{i}$$

$$                                                           =g \cdot x_{0}^{i}$$.

So we showed that $ f^{-1}f=\mathbb{I}_{X}$ now check $ f f^{-1}=\mathbb{I}_{X}$.

$$ f f^{-1}(g\cdot x_{0}^{j})=f(\phi^{-1}(g)\phi^{-1}(g_{ij}^{-1}) x_{0}^{i})=\phi(\phi^{-1}(g)\phi^{-1}(g_{ij}^{-1}))g_{ij}\cdot x_{0}^{j}$$

$$                                                           = g g_{ij}^{-1}g_{ij}\cdot x_{0}^{j}$$

$$                                                           = g\cdot x_{0}^{j}$$.

To complete the demonstration, verify that $\phi^{-1}$ has the necessary condition of proposition (1.3).

$$ \phi^{-1}( Stab_{G}(x_{0}^{j} ))=\phi^{-1}({Stab_{G}(x_{0}^{j} )}^{g_{ij}g_{ij}^{-1}})$$

$$                                 ={\phi^{-1}({Stab_{G}(x_{0}^{j} )}^{g_{ij}})}^{\phi^{-1}(g_{ij}^{-1})}$$

$$                                  ={\phi^{-1}(\phi(Stab_{G}(x_{0}^{i} ))}^{\phi^{-1}(g_{ij}^{-1})}$$

$$                                  ={Stab_{G}(x_{0}^{i} )}^{\phi^{-1}(g_{ij}^{-1})}$$

\end{proof}	
	
 And now let’s investigate how two functions $\phi$-invariant and $\psi$-functions behave through the composition.
 
\begin{prop}
	If two functions, $f$ and $h$, are resptictively $\phi$-invariant and $\psi$-invariant then $fh$ is $\phi\psi$-invariant.
\end{prop}  	
\begin{proof}
We know by hypothesis that it exists $g_{ij}$, $g_{ij}^{'} \in G$ so that $$ f(g\cdot x_{0}^{i})=\phi(g)g_{ij}\cdot x_{0}^{j} \quad h(g\cdot x_{0}^{i})=\psi(g)g_{ij}^{'}\cdot x_{0}^{j}$$
for  appropiate $j \in {1,2,\dots,t}$ and for each $i \in {1,2,\dots,t}$.

Thus the compositions between $f$ and $h$ yields $$ fh(g\cdot x_{0}^{i})=f(\psi(g)g_{ij}^{'}\cdot x_{0}^{j})=\phi\psi(g)\phi(g_{ij}^{'})g_{jk}\cdot x_{0}^{k}$$
for  appropiate $k \in {1,2,\dots,t}$.

To simplify the notation we place $g_{ik}^{''}=\phi(g_{ij}^{'})g_{jk}$ and so that $fh$ can rewrite as $fh(g\cdot x_{0}^{i})=\phi\psi(g)g_{ik}^{''}\cdot x_{0}^{k}$.

Now we have to prove that $fh$ is well defined through stabilizers

$$ \phi\psi(Stab_{G}(x_{0}^{i} ))=\phi({Stab_{G}(x_{0}^{j} )}^{g_{ij}^{'}})=\phi({Stab_{G}(x_{0}^{j} )})^{\phi(g_{ij}^{'})}$$

$$                                                                         =\phi(g_{ij}^{'})g_{jk} {Stab_{G}(x_{0}^{k} )}g_{jk}^{-1}{\phi(g_{ij}^{'})}^{-1}$$

$$                                                                         =\phi(g_{ij}^{'})g_{jk} {Stab_{G}(x_{0}^{k} )}{((\phi(g_{ij}^{'})g_{jk})}^{-1}$$

$$  ={Stab_{G}(x_{0}^{k} )}^{\phi(g_{ij}^{'})g_{jk}}= {Stab_{G}(x_{0}^{k} )}^{g_{ik}^{''}}  $$.

\end{proof}

We have seen the necessary condition of a $\phi$ automorphism for a $\phi$-invariant function, now we will show that the proposition thesis (1.3) is a sufficient condition for a $\phi$-invariant function.

\begin{prop}
There is at least one $\phi$-invariant function $f\neq\mathbb{I}_{X}  \iff \phi(Stab_{G}(x_{0}^{i}))=Stab_{G}(x_{0}^{j})^{g_{ij}}  \hspace{0.25cm} \forall x_{0}^{i}$ representative of the orbits of $X$ under the action of group $G$
\end{prop}

\begin{proof}
By proposition (1.3)we already demonstrated the implication $\Rightarrow$). So now we have to prove the inverse implication( $\Leftarrow$.

Given a $\phi \in Aut(G) $ so that $\phi(Stab_{G}(x_{0}^{i}))=Stab_{G}(x_{0}^{j})^{g_{ij}}$ then we can define an application $f(g\cdot x_{0}^{i})=\phi(g)g_{ij}\cdot x_{0}^{j}$.
And now to finish the demonstration we have to prove that $f$ is bijective.
In fact $f$  is injective because $\phi(g)g_{ij}\cdot x_{0}^{j}=\phi(h)g_{ij}\cdot x_{0}^{j} \Rightarrow \phi(h^{-1}g)\cdot g_{ij}\cdot x_{0}^{j}=g_{ij}\cdot x_{0}^{j} $,  thus $\phi(h^{-1}g) \in Stab_{G}(x_{0}^{j})^{g_{ij}}$, by hypothesis $\phi(h^{-1}g) \in \phi(Stab_{G}(x_{0}^{i})) $ and   since that $ h^{-1}g=\phi^{-1}\phi(h^{-1}g)$ we get that $h^{-1}g \in \phi^{-1}( \phi(Stab_{G}(x_{0}^{i})))=Stab_{G}(x_{0}^{i})$ and finally we get $ h^{-1}g\cdot x_{0}^{i}= x_{0}^{i} \Rightarrow g\cdot x_{0}^{i}=h\cdot x_{0}^{i}$.

On the other hand $f$ is also surjective in that if $g\cdot x_{0}^{j} \in Orb_{G} (x_{0}^{j})$, we can see $g\cdot x_{0}^{j}$ as $ gg_{ij}^{-1}g_{ij}\cdot x_{0}^{j}$
then $g\cdot x_{0}^{j}=\phi (\phi^{-1}(gg_{ij}^{-1}))g_{ij}\cdot x_{0}^{j}$. So we found an element of the counter-image of $g\cdot x_{0}^{j}$ that is $\phi^{-1}(gg_{ij}^{-1})\cdot x_{0}^{i}$, in fact $f(\phi^{-1}(gg_{ij}^{-1})\cdot x_{0}^{i})=g\cdot  x_{0}^{i}$  .

\end{proof}

The set of automorphism $\phi$ which satisfies Proposition (1.6) forms a soubgroup of $Aut(G)$. In fact if $ \phi(Stab_{G}(x_{0}^{i}))=Stab_{G}(x_{0}^{j})^{g_{ij}}$ then $\phi^{-1}(Stab_{G}(x_{0}^{j})^{g_{ij}})=Stab_{G}(x_{0}^{i}) \Rightarrow \phi^{-1}(Stab_{G}(x_{0}^{j}))^{\phi^{-1}(g_{ij})}= Stab_{G}(x_{0}^{i}) $ thus $\phi^{-1}(Stab_{G}(x_{0}^{j}))= Stab_{G}(x_{0}^{i})^{\phi^{-1}(g_{ij}^{-1})}$.
It follows from this that at every $\psi$ of Proposition (1.6) $\psi^{-1}$ also satisfies Proposition (1.6), and it is also easy to see that $ id_{G} \in G$ satisfies Proposition 1.6 since $id_{G}(Stab_{G}(x_{0}^{i}))=Stab_{G}(x_{0}^{i})$.
Then if $\phi$ and $\psi$ satisfy Proposition (1.6) it is verified that $\phi\psi$ satisfies (1.6).

In fact $\phi\psi(Stab_{G}(x_{0}^{i}))=\phi(Stab_{G}(x_{0}^{j})^{g_{ij}})=\phi(Stab_{G}(x_{0}^{j}))^{\phi(g_{ij})}=Stab_{G}(x_{0}^{k})^{g_{ik}}$ where $g_{ik}=\phi(g_{ij})g_{jk}^{'}$ for an opportune $g_{jk}^{'}$ so that $\phi(Stab_{G}(x_{0}^{j}))=Stab_{G}(x_{0}^{k})^{g'_{jk}}$.

We then determined the subgroup
\begin{equation}
	 \{\phi \in Aut(G)| \phi \hspace{0.25cm}  satisfies \hspace{0.25cm} (1.6)  \}
\end{equation}
that we will denote ${Aut(G)}^{S}$ i.e the group of the automorphism that send stabilizers to other stabilizers with $S:=(Stab_{G}(x_{0}^{i}))_{i=1}^{t}$ that is, the ordered t-uple that by components has the stabilizers of $G$ which we can name as the sobgroup of automorphism $S$-invariants.

From (1.2) we can rephrase Proposition (1.6) as follows:

\begin{coro}
	There is at least one $\phi$-invariant function  $f$  $\iff \phi \in {Aut(G)}^{S} \leq Aut(G).$
\end{coro}

Furthermore, for Propositions (1.3) , (1.4), (1.5) $ \{f | f \hspace{0.25cm} \phi-invariant   \} \leq Sym(X)$ which we can name as $\widehat{{Aut(G)}^{S}}$ i.e the sobgroup of functions $S$-invariants.

\begin{os}
As ${Aut(G)}^{S}$   is defined, we have that the subgroup of $Aut(G)$ of internal automorphisms $I(G)$ ( to see a specific definition see definition 2.28 of[1]) is a subgroup of ${Aut(G)}^{S}$ . In fact	given a $\gamma_{g} \in I(G)$ then $\gamma_{g}(Stab_{G}(x_{0}^{i}))=Stab_{G}(x_{0}^{i})^{g}$ with $g_{ii}=g$ for each $i$.

Moreover, the corresponding $\gamma_{g}$-invariants are the representations of the elements of $G$ on the set $X$ i.e., the maps $f_{g}:x \rightarrow g\cdot x$. 

In fact if $f$ is a  $\gamma_{g}$-invariant then by definition 
$$  f(h\cdot x_{0}^{i} )=\gamma_{g}(h)g\cdot x_{0}^{i} $$
$$   f(h\cdot x_{0}^{i} )=ghg^{-1}g\cdot x_{0}^{i}$$
$$ f(h\cdot x_{0}^{i} )=gh\cdot x_{0}^{i}=g\cdot (h\cdot x_{0}^{i}).$$

By posing $h\cdot x_{0}^{i}=x$ we obtain $f=f_{g}$.

Conversely, gaven $f_{g}:x \rightarrow g\cdot x$ there exist appropiate $h \in G$ and $x_{0}^{i} \in X$ such that $ x= h\cdot x_{0}^{i}$ .

 Thus
 $$ f_{g}(h\cdot x_{0}^{i}) =g\cdot (h\cdot x_{0}^{i})=gh\cdot x_{0}^{i}=ghg^{-1}g\cdot x_{0}^{i}=\gamma_{g}(h)g\cdot x_{0}^{i}.$$

\end{os}

So we conclude that 

\begin{prop}
 A function $f$ is $\gamma_{g}$-invariant  with $\gamma_{g} \in I(G)  \iff  f=f_{g}.$  
	
\end{prop}

To conclude this section, we can note that in some cases a $\phi \in Aut(G)^{S}$ can define more than one $f \hspace{0.1cm} \phi$-invariant function.
In fact, $ \phi(Stab_{G}(x_{0}^{i}))={Stab_{G}(x_{0}^{j})}^{g_{ij}}={Stab_{G}(x_{0}^{j})}^{g_{ij}n_{j}}$ where $n_{j} \in N_{G}(Stab_{G}(x_{0}^{j}))$.

Then if $n_{j} \notin Stab_{G}(x_{0}^{j})$ then exsists another function $f\hspace{0.1cm} \phi$-invariant with ${g_{ij}}^{'}=g_{ij}n_{j}$. Otherwise if $n_{j} \in Stab_{G}(x_{0}^{j})$ we can still consider the  element $g_{ij}$ that maps $f:x_{0}^{i}\rightarrow g_{ij}\cdot x_{0}^{j}$ since $g_{ij}n_{j}\cdot x_{0}^{j}=g_{ij}\cdot x_{0}^{j}$.

Therefore we can assert that

\begin{prop}
Given $f$ is a $\phi$-invariant, it is unique if $N_{G}(Stab_{G}(x_{0}^{j}))=Stab_{G}(x_{0}^{j}) \hspace{0.25cm} \forall x_{0}^{j}.$
	
\end{prop}

\quad

\section{Permutations induced}

Once the representatives of the orbits have been fixed, one can revise a $\phi$-invariant f in terms of permutations on the indices of the orbits. This implies that an ordering on the representatives must be provided. In concrete terms given a $\phi \in Aut(G)^{S}$ to it is associated an appropriate permutation $\gamma^{\phi}$ on the sets of indices $I= \{1,2,3,4,...,i,...,t\}$ ordering the representatives $x_{0}^{i}$ such that

$$ f(g\cdot x_{0}^{i} )= \phi(g)g_{i\gamma^{\phi}(i) }\cdot x_{0}^{\gamma^{\phi}(i)}. $$

We can call such permutation $\gamma^{\phi}$ as permutation $\phi$-induced .

It might happen that there exists another $\phi$-induced permutation $\sigma^{\phi}$ where $\sigma^{\phi}(i) \neq \gamma^{\phi}(i)$ for some $ i\in I$.In this case we'll have that $\phi((Stab_{G}(x_{0}^{i}))={Stab_{G}(x_{0}^{\sigma^{\phi}(i)})}^{g_{i\sigma^{\phi}(i)}}$ and $\phi((Stab_{G}(x_{0}^{i}))={Stab_{G}(x_{0}^{\gamma^{\phi}(i)})}^{g_{i\gamma^{\phi}(i)}} \Rightarrow {Stab_{G}(x_{0}^{\sigma^{\phi}(i)})}^{g_{i\sigma^{\phi}(i)}}={Stab_{G}(x_{0}^{\gamma^{\phi}(i)})}^{g_{i\gamma^{\phi}(i)}} \Rightarrow Stab_{G}(x_{0}^{\sigma^{\phi}(i)})= {Stab_{G}(x_{0}^{\gamma^{\phi}(i)})}^{{g^{-1}}_{i\sigma^{\phi}(i)}g_{i\gamma^{\phi}(i)}} \Rightarrow |Cl(Stab_{G}(x_{0}^{\gamma^{\phi}(i)}))|> 1 $ for some $i \in I$. In this context, such a set refers to the conjugate class of a group. For a detailed explanation, refer to Definition 2.47 in [1].From this we can deduce the following property

\begin{prop}
	
There is more than one permutation $\phi$-induced $\gamma^{\phi}$ then $|Cl(Stab_{G}(x_{0}^{\gamma^{\phi}(i)}))|>1$ for some $i \in I.$
	
\end{prop}

However, as the function $Stab_{G}(x_{0}^{\gamma^{\phi}(i)})$ is defined, the converse is also true. 

\begin{prop}
	
	There is more than one permutation $\phi$-induced $\gamma^{\phi} \iff |Cl(Stab_{G}(x_{0}^{\gamma^{\phi}(i)}))|>1$ for some $i \in I.$
	
\end{prop}

From proposition (2.2) with short logical steps, the following corollary can be deduced.

\begin{coro}
	There is only one permutation $\phi$-induced $\gamma^{\phi} \iff |Cl(Stab_{G}(x_{0}^{\gamma^{\phi}(i)}))|=1     \hspace{0.25 cm}              \forall i \in I$ namely $Stab_{G}(x_{0}^{\gamma^{\phi}(i)}) \unlhd G          \hspace{0.25cm} \forall i \in I$ 
\end{coro}

It can be observed that to a permutation $\gamma^{\phi}$ induced there exists its $\psi$-induced inverse for some $\psi \in Aut(G)^{S}$. In fact,
$$ \phi(Stab_{G}(x_{0}^{i}))={Stab_{G}(x_{0}^{\gamma^{\phi}(i)})}^{g_{i\gamma^{\phi}(i)}} \Rightarrow {\phi}^{-1}({Stab_{G}(x_{0}^{\gamma^{\phi}(i)})}^{g_{i\gamma^{\phi}(i)}})=Stab_{G}(x_{0}^{i}), $$

therefore

$$\phi^{-1}(Stab_{G}(x_{0}^{\gamma^{\phi}(i)}))=Stab_{G}(x_{0}^{i})^{{(\phi^{-1}(g_{i\gamma^{\phi}(i)})})^{-1}}.$$

Assuming $\gamma^{\phi}(i)=j$ and $g_{j{\gamma}^{\phi^{-1}}(j)}={{{\phi}^{-1}(g_{i\gamma^{\phi}(i)}^{-1})}} $, the previous equality becomes

$$  \phi^{-1}(Stab_{G}(x_{0}^{j}))={Stab_{G}(x_{0}^{\gamma^{\phi^{-1}(j)}})^{g_{j{\gamma^{\phi}}^{-1}(j)}}} .$$

We have then shown the following proposition.

\begin{prop}
 Given a $\phi$-induced permutation $\gamma^{\phi}$ the inverse is a $\phi^{-1}$- induced permutation i.e. exists a permutation  $\overline{\gamma}^{\phi^{-1}}$ such that
 
  \begin{equation}
  (\gamma^{\phi})^{-1}= \overline{\gamma}^{\phi^{-1}}.
  \end{equation}
  
\end{prop}

Now we prove the following result.

\begin{prop}
	
Given $\gamma^{\phi}$ and $\sigma^{\psi}$ respectively $\phi$-induced and $\psi$-induced permutations their product is a $\phi\psi$-induced pemutation i.e. exists a permutation  ${\gamma\sigma}^{\phi\psi}$ such that
\begin{equation}	
\gamma^{\phi}\sigma^{\psi}=\gamma\sigma^{\phi\psi}. 
	\end{equation}	

\end{prop}

\begin{proof}
	Given two functions $f$ and $h$ respectively $\phi$-invariant and $\psi$-invariant associated  to $\gamma^{\phi}$ and $\sigma^{\psi}$ respectively. By proposition (1.5) we have that $fh$ is $\phi \psi$-invariant so we must look for $\gamma^{\phi}\sigma^{\psi}$ in that function
	$$fh(g\cdot x_{0}^{i}) =f(h(g\cdot x_{0}^{i}) ))= f( \psi(g)g_{i\sigma^{\psi}(i)}\cdot x_{0}^{\sigma^{\psi}(i)} )  $$
	
	$$                  fh(g\cdot x_{0}^{i})      = \phi\psi(g)\phi(g_{i\sigma^{\psi}(i)})g_{\sigma^{\psi}(i) \gamma^{\phi}\sigma^{\psi}(i)}\cdot x_{0}^{\gamma^{\phi}\sigma^{\psi}(i)}.$$

	Assuming $g_{i \gamma\sigma^{\phi\psi}(i)}=\phi(g_{i\sigma^{\psi}(i)})g_{\sigma^{\psi}(i) \gamma^{\phi}\sigma^{\psi}(i)}$ we get the thesis.
\end{proof}

	Combining propositions 2.4 and 2.5 we have that the set

$$   \{ \gamma \in Sym(I)| \gamma=\gamma^{\phi} \hspace{0.25cm} \phi \in {Aut(G)}^{S}   \}    $$

is a soubgrup of $Sym(I)$ that we will denote in this way $ {Sym(I)}^{Aut(G)^S}$ i.e., the group of permutations induced on the representatives of the orbits.

\quad

We now turn to the following subset of $ {Sym(I)}^{Aut(G)^S}$ namely $ {Sym(I)}^{\{\mathbb{I}_{G}\}}$,  if $ {Sym(I)}^{\{ \mathbb{I}_{G}\} }=\{\mathbb{I}_{I}\}$ is obviously a soubgroup of  $ {Sym(I)}^{Aut(G)^S}$. We want to see if this also happens in the case that 
$ {Sym(I)}^{\{\mathbb{I}_{G} \}} \supset \{\mathbb{I}_{I}\} $ for the purpose we prove the following lemma.

\begin{lem}
	
Given a $\mathbb{I}_{G}$-induced permutation $\gamma^{\mathbb{I}_{G}}$, the inverse is a $\mathbb{I}_{G}$- induced permutation  that we can denote according to (2.2) $\overline{\gamma}^{\mathbb{I}_{G}}$
	
\end{lem}

\begin{proof}
	
Existence by hypothesis of the $\gamma^{\mathbb{I}_{G}}$ permutation implies existence of an associated nontrivial function $\mathbb{I}_{G}$-invariant $ f$ so defined 	$f(g\cdot x_{0}^{i})=\mathbb{I}_{G}(g) g_{i \gamma^{\mathbb{I}_{G}}(i)}\cdot x_{0}^{\gamma^{\mathbb{I}_{G}}(i)}=gg_{i \gamma^{\mathbb{I}_{G}}(i)}\cdot x_{0}^{\gamma^{\mathbb{I}_{G}}(i)}$.

By proposition (1.4) 

$$ f^{-1}(g\cdot x_{0}^{\gamma^{\mathbb{I}_{G}}(i)} )=\mathbb{I}_{G}^{-1}(g)\mathbb{I}_{G}^{-1}(g^{-1}_{i \gamma^{\mathbb{I}_{G}}(i)})\cdot x_{0}^{i}        . $$

From here we get

$$ f^{-1}(g\cdot x_{0}^{\gamma^{\mathbb{I}_{G}}(i)} )= \mathbb{I}_{G}(g)g^{-1}_{i \gamma^{\mathbb{I}_{G}}(i)}\cdot  x_{0}^{(\gamma^{\mathbb{I}_{G}})^{-1}\gamma^{\mathbb{I}_{G}}(i)} . $$

By positing $j=\gamma^{\mathbb{I}_{G}}(i)$  and $g_{i {(\gamma^{\mathbb{I}_{G}})}^{-1}(i)}=g^{-1}_{i \gamma^{\mathbb{I}_{G}}(i)}$,  we have

$$f^{-1}(g\cdot x_{0}^{j} )=gg_{i {(\gamma^{\mathbb{I}_{G}})}^{-1}(i)}\cdot x_{0}^{{(\gamma^{\mathbb{I}_{G}})}^{-1}(j)}  $$
i.e, the thesis.
	
\end{proof}

\quad

Using proposition (2.4) in the case $\phi=\mathbb{I}_{G}$ and $\psi=\mathbb{I}_{G}$, by lemma 2.6 we get that 

    \begin{equation}
	{Sym(I)}^{\{\mathbb{I}_{G} \}} \leq  {Sym(I)}^{Aut(G)^S}	
    \end{equation}

 The ${Sym(I)}^{\{\mathbb{I}_{G} \}}$  group is not merely a subgroup of  ${Sym(I)}^{Aut(G)^S}$.  Indeed, we observe that for each $\gamma^{\phi}$ permutations $\phi$-induced, considering a $\sigma^{\mathbb{I}_{G}}$, we have that

$$  \gamma^{\phi}\sigma^{\mathbb{I}_{G}}\overline{\gamma}^{\phi^{-1}}=\gamma\sigma^{\phi\mathbb{I}_{G} }\overline{\gamma}^{\phi^{-1}}=\gamma\sigma^{\phi }\overline{\gamma}^{\phi^{-1}} =       $$

$$                                                            = \gamma\sigma\overline{\gamma}^{\phi\phi^{-1}}=\gamma\sigma\overline{\gamma}^{\mathbb{I}_{G}}  .$$

From this we conclude that

\begin{equation}		{Sym(I)}^{\{\mathbb{I}_{G} \}} \unlhd  {Sym(I)}^{Aut(G)^S}.
\end{equation}

\begin{os}

It is important to highlight that despite the fact that ${Sym(I)}^{\{\mathbb{I}_{G} \}} $ we have defined it from the identity automorphism $\mathbb{I}_{G}$,  there may be $\phi$-induced permutations with $\phi \neq \mathbb{I}_{G}$  that belong to that group. In fact, all this we can see from the fact that it may happen that, given two induced permutations $ \gamma^{\phi}$ and $\sigma^{\psi}$, where $\phi$ and $\psi$ are two distinct automorphisms, we have that $\gamma^{\phi}=\sigma^{\psi}$. This implies that $\gamma^{\phi}{(\sigma^{\psi})}^{-1}=\mathbb{I}_{I}$ and therefore by (2.2) we get that $\gamma\overline{\sigma}^{\phi{\psi}^{-1}}=\mathbb{I}_{I} $.	On the other hand as an example we have that the permutations $\sigma^{\gamma_{g}}$ induced by the internal automorphisms $\gamma_{g}$ are equal to the identity $\mathbb{I}_{I} $ ( for that see remark 1.8). Also from that example, we can obtain two distinct automorphisms $\gamma_{g}\phi$ and $\phi$, where two induced permutations $\tau^{\gamma_{g}\phi}$ and $\sigma^{\phi}$ are equal.
		
\end{os}

Now starting from the quotient group $ {Sym(I)}^{Aut(G)^S}/{Sym(I)}^{\{\mathbb{I}_{G} \}}$ we find that there is a relationship between  $ {Sym(I)}^{Aut(G)^S}/{Sym(I)}^{\{\mathbb{I}_{G} \}}$   and the $Aut(G)^S$  group through the following map
		
		$$\hspace{0.8cm} Aut(G)^S \stackrel{F}{\longrightarrow}  {Sym(I)}^{Aut(G)^S}/{Sym(I)}^{\{\mathbb{I}_{G} \}}$$
			\quad
		$$   \phi  \mapsto \gamma^{\phi}{Sym(I)}^{\{\mathbb{I}_{G} \}}                 . $$
		
\quad
		
First of all, it is good to point out that the map $F$ is well-defined ; that is such the $\phi$ automorphism, although it might generate more than one induced	$\gamma^{\phi}$ permutation, nevertheless it will be mapped into the one class $\gamma^{\phi}{Sym(I)}^{\{\mathbb{I}_{G} \}}$ and this assured to us by properties (2.1) and (2.2).
	The map $F$ is also a homorphism of groups in fact $F(\phi\psi)=\gamma^{\phi\psi}{Sym(I)}^{\{\mathbb{I}_{G} \}}$ and it is known that there are $\sigma^{\phi}$ and $\tau^{\psi}$ such that
		
		$$ \gamma^{\phi\psi}{Sym(I)}^{\{\mathbb{I}_{G} \}}=\sigma^{\phi}\tau^{\psi}{Sym(I)}^{\{\mathbb{I}_{G} \}}=\sigma^{\phi}{Sym(I)}^{\{\mathbb{I}_{G} \}}\tau^{\psi}{Sym(I)}^{\{\mathbb{I}_{G} \}}   $$
		 and then
		 
		 $$F(\phi\psi)=F(\phi)F(\psi).$$
		 
Furthermore, the map $F$ is obviously suriective homomorphism so we obtain that by the first isomorphism theorem for groups

\begin{equation}
Aut(G)^S/kerF \cong {Sym(I)}^{Aut(G)^S}/{Sym(I)}^{\{\mathbb{I}_{G} \}}.
\end{equation}		 
		 
\begin{os}
To conclude section 2, let us analyze the $kerF$ subgroup. Given that $$kerF=\{ \phi \in Aut(G)^S |  \gamma^{\phi} \in {Sym(I)}^{\{\mathbb{I}_{G} \}} \hspace{0.2cm} \forall \gamma^{\phi}   \}.$$ We can deduce based on observation (2.7) that the set $I(G) \leq kerF$. In particular, any permutation $\gamma_{g}$-induced coincides with the identity $\{\mathbb{I}_{I} \} $, from which it follows that  $I(G)\unlhd Aut(G)^S$, which can be easily proved from the definition of internal automorphism.
\quad
Another interesting group that contains $I(G)$ and is contained $kerF$, is the subgroup fixing orbits i.e. such a subgroup contains all automorphisms that send a stabilizer to its conjugate . Such a subgroup we can denote as $Aut(G)^O$ with $O:=(Orb_{G}(x_{0}^{i}))_{i=1}^{t}$. Moreover, by the same argument made earlier for $I(G)$, it can be easily shown that $Aut(G)^O \unlhd Aut(G)^S$ . In conclusion we can summarize it thus

\begin{equation}
\quad
\hspace{0.5cm} I(G)\unlhd Aut(G)^O \unlhd kerF, \hspace{0.2cm} Aut(G)^O \unlhd Aut(G)^S .		
\end{equation}

\end{os}

\quad\quad

\section{The algebraic structure of $\widehat{{Aut(G)}^{S}}$}	 

\quad

We proceed to introduce a new topic, which is to rewrite S-invariant functions in a more compact manner in order to streamline the notation. 

\quad

First we denote by $S_i:=Stab_G(x_{0}^{i})$ and by $N_i:=N_G(S_i) \hspace{0.2cm} \forall i \in I$ then by $N:=(N_i)_{i=1}^{t}$; whereas by denoting $n \in N$ we say that $n=(n_i)_{i=1}^{t}$ with $n_i \in N_i$. Next we arbitrarily fix the elements $\gamma^\phi\neq\mathbb{I}_{I}  \in {Sym(I)}^{Aut(G)^S}$, associated with $\phi \in Aut(G)^{S}$.  For each element $\gamma^\phi \sigma^{\mathbb{I}_{G}} \in {Sym(I)}^{Aut(G)^S} $ we again arbitrarily fix the representatives $g_{i \gamma^\phi\sigma^{\mathbb{I}_{G}}(i)}\neq 1 $ of the classes $ G/N_{\gamma^\phi\sigma^{\mathbb{I}_{G}}(i)} $ and again fix the element $1$ as a representative of the class $N_i$. We will then denote by $g_{ \gamma^\phi\sigma^{\mathbb{I}_{G}}}n S=(g_{j_{i} \gamma^\phi\sigma^{\mathbb{I}_{G}}(j_{i})} n_i S_i)_{i=1}^{t}$ where $\gamma^\phi\sigma^{\mathbb{I}_{G}}(j_{i})=i$.

\quad

Note also that it can happen that $\phi\neq \mathbb{I}_{G}$ is such that $\phi(S_i)=S_i$ for every $i$ or that $\phi(S_i)=S_{\gamma(i)}$ for some $i$. In the first case we choose $g_{i,\gamma^\phi(i)}$, which we denote by $g_{\phi,i,\mathbb{I}_{I}(i)}$ and set $g_{\phi,i,\mathbb{I}_{I}(i)}=1$, in the second case we simply set $g_{i,\gamma^\phi(i)}=1$. On the other hand, if for example $\phi(S_i)=S_{i}^{g}$, for some $g\neq 1$, we will denote this element as $g_{\phi,i,\mathbb{I}_{I}(i)}\neq 1$ for some $i$ (e.g. $\gamma_{g}$-invariant functions or, more generally, orbit-fixing functions), and the element $g_{\phi,i,\mathbb{I}_{I}(i)}$ is arbitrarily fixed among the classes $G/N_i$.

\quad\quad

Now we can start compacting the notation of a $\phi$-invariant function, but we need to further compact the following scripts: from now on in this section, unless you specify otherwise, we will use $x_0$ and $g$ to denote the following ordered $t$-uple $x_0: =(x_{0}^{i})_{i=1}^{t}$ and $g:=(g_i)_{i=1}^{t}$ with $g_i \in G$, in particular we will denote $\bar{g}$ a constant t-uple i. e.g. $\bar{g}=(g)_{i=1}^{t}$, for example with $\bar{1}$ we will denote $\bar{1}=(1)_{i=1}^{t}$. Therefore $\phi(g)=( \phi(g_i))_{i=1}^{t}$, also for $g\cdot x_0$ we will understand $g\cdot x_0:=(g_i\cdot\ x_{0}^{i})_{i=1}^{t}$. 
Finally, we move on to the notation change on a $\phi$-invariant $f$ function such that

 $\phi(S)={S_{ \gamma^\phi\sigma^{\mathbb{I}_{G}}}}^{g_{ \gamma^\phi\sigma^{\mathbb{I}_{G}}}n_{ \gamma^\phi\sigma^{\mathbb{I}_{G}}}}$(i.e, with notation we adopted $\phi(S_{i})={S_{ \gamma^\phi\sigma^{\mathbb{I}_{G}}(i)}}^{g_{i \gamma^\phi\sigma^{\mathbb{I}_{G}}(i)}n_{ \gamma^\phi\sigma^{\mathbb{I}_{G}}(i)}}$)

             $$f=f_{g_{ \gamma^\phi\sigma^{\mathbb{I}_{G}}}n S}$$
             
             $$f_{g_{ \gamma^\phi\sigma^{\mathbb{I}_{G}}}n S}(g\cdot x_0)=\phi(g)g_{ \gamma^\phi\sigma^{\mathbb{I}_{G}}}n_{ \gamma^\phi\sigma^{\mathbb{I}_{G}}}\cdot x_{0}^{\gamma^\phi\sigma^{\mathbb{I}_{G}}}$$
 
 where the last equality means the following equality between ordered t-uples
 
  $$f_{g_{ \gamma^\phi\sigma^{\mathbb{I}_{G}}}n S}(g_i\cdot x_{0}^{i})=\phi(g_i)g_{i \gamma^\phi\sigma^{\mathbb{I}_{G}}(i)}n_{ \gamma^\phi\sigma^{\mathbb{I}_{G}}(i)}\cdot x_{0}^{\gamma^\phi\sigma^{\mathbb{I}_{G}}(i)} \hspace{0.2cm} \forall i \in I .$$           		

\quad

In the new writing of this $\phi$-invariant function, we have put a subscript with the elements $g_{ \gamma^\phi\sigma^{\mathbb{I}_{G}}}n S$ to emphasize the invariance from the stabilizers $S$ elements and the dependence instead on $g_{ \gamma^\phi\sigma^{\mathbb{I}_{G}}}$   and $n$. In fact, as $n$ varies and as $\gamma^\phi\sigma^{\mathbb{I}_{G}}$  varies, we will have different $\phi$-invariant functions (for this see last part of section 1  and Proposition 1.10).
In the case where the functions are $\phi$-invariant with $\phi\neq\mathbb{I}_{G}$  but belong to $ker F$, so as not to confuse them with $\mathbb{I}_{G}$-invariant functions, we will denote them to emphasize the dependence on $\phi$; thus $f_{\phi, g_{\sigma^{\mathbb{I}_{G}}}nS}$ will be written specifically in particularly for the fixing stabilizers we will write $f_{\phi,nS}$ , while more generally for these fixing orbits we will denote them by $f_{\phi,g_{\mathbb{I}_{I}}nS}$.
\quad

We will now show that the set of functions $f_{nS}$ forms a normal subgroup of $\widehat{{Aut(G)}^{S}}$ .
\quad
 First we show that the functions  $f_{nS}$ form a subgroup of $\widehat{{Aut(G)}^{S}}$ ; in fact
 
 $$(f_{nS})^{-1}(g\cdot x_0)=\mathbb{I}_{G}^{-1}(g)\mathbb{I}_{G}^{-1}(n^{-1})\cdot x_0=gn^{-1}\cdot x_0$$
 
 and therefore
 
 \begin{equation}
 	(f_{nS})^{-1}=f_{n^{-1}S}
 \end{equation}

and then

$$f_{nS}f_{\tilde{n}S}(g\cdot x_0)=f_{nS}(g\tilde{n}\cdot x_0)=\mathbb{I}_{G}(g\tilde{n})n\cdot x_0= $$

$$        =g\tilde{n}n \cdot x_0= f_{\tilde{n}nS}   $$

thus

\begin{equation}
	f_{nS}f_{\tilde{n}S}=f_{\tilde{n}nS}
\end{equation}              
 
and since $\mathbb{I}_X$ can be rewritten as $f_{S}$ we have shown that

\begin{equation}	
	\{  f_{nS} | n \in N \}=<f_{nS}> \leq \widehat{{Aut(G)}^{S}}
\end{equation}

\quad

Now we prove a normality of such a soubgroup.

\begin{prop}
\quad	
$<f_{nS}> \unlhd \widehat{{Aut(G)}^{S}}$	
\end{prop}
\begin{proof}
	
Given an $f_{g_{ \gamma^\phi\sigma^{\mathbb{I}_{G}}}n S} \in \widehat{{Aut(G)}^{S}}$  and an $f_{\tilde{n}S} \in <f_{nS}>$ we get

$$ f_{g_{ \gamma^\phi\sigma^{\mathbb{I}_{G}}}n S}f_{\tilde{n}S}(f_{g_{ \gamma^\phi\sigma^{\mathbb{I}_{G}}}n S})^{-1}(g\cdot x_0)= f_{g_{ \gamma^\phi\sigma^{\mathbb{I}_{G}}}n S}f_{\tilde{n}S}(\phi^{-1}(g)\phi^{-1}((g_{ \gamma^\phi\sigma^{\mathbb{I}_{G}}}n_{\gamma^\phi\sigma^{\mathbb{I}_{G}}})^{-1}) \cdot x_{0}^{(\gamma^\phi\sigma^{\mathbb{I}_{G}})^{-1}})=$$

$$ =f_{g_{ \gamma^\phi\sigma^{\mathbb{I}_{G}}}n S}(\phi^{-1}(g)\phi^{-1}((g_{ \gamma^\phi\sigma^{\mathbb{I}_{G}}}n_{\gamma^\phi\sigma^{\mathbb{I}_{G}}})^{-1}) \tilde{n}_{( \gamma^\phi\sigma^{\mathbb{I}_{G}})^{-1}}\cdot x_{0}^{(\gamma^\phi\sigma^{\mathbb{I}_{G}})^{-1}}) =   $$

$$           = g(g_{ \gamma^\phi\sigma^{\mathbb{I}_{G}}}n_{ \gamma^\phi\sigma^{\mathbb{I}_{G}}} )^{-1}\phi(\tilde{n}_{( \gamma^\phi\sigma^{\mathbb{I}_{G}})^{-1}})g_{ \gamma^\phi\sigma^{\mathbb{I}_{G}}}n_{ \gamma^\phi\sigma^{\mathbb{I}_{G}}}\cdot x_{0}^{\gamma^\phi\sigma^{\mathbb{I}_{G}}(\gamma^\phi\sigma^{\mathbb{I}_{G}})^{-1}}=g(g_{ \gamma^\phi\sigma^{\mathbb{I}_{G}}}n_{ \gamma^\phi\sigma^{\mathbb{I}_{G}}} )^{-1}\phi(\tilde{n}_{( \gamma^\phi\sigma^{\mathbb{I}_{G}})^{-1}})g_{ \gamma^\phi\sigma^{\mathbb{I}_{G}}}n_{ \gamma^\phi\sigma^{\mathbb{I}_{G}}}\cdot x_0 .$$

It remains for us to prove that $(g_{ \gamma^\phi\sigma^{\mathbb{I}_{G}}}n_{ \gamma^\phi\sigma^{\mathbb{I}_{G}}} )^{-1}\phi(\tilde{n}_{( \gamma^\phi\sigma^{\mathbb{I}_{G}})^{-1}})g_{ \gamma^\phi\sigma^{\mathbb{I}_{G}}}n_{ \gamma^\phi\sigma^{\mathbb{I}_{G}}} \in N$; since $f_{\tilde{n}S}$ is associated with the $\mathbb{I}_{G}$ automorphism, the conjugate of $f_{\tilde{n}S}$ via $f_{g_{ \gamma^\phi\sigma^{\mathbb{I}_{G}}}n S} $  is also associated with the $\mathbb{I}_{G}$ automorphism so that 

$$S=S^{(g_{ \gamma^\phi\sigma^{\mathbb{I}_{G}}}n_{ \gamma^\phi\sigma^{\mathbb{I}_{G}}} )^{-1}\phi(\tilde{n}_{( \gamma^\phi\sigma^{\mathbb{I}_{G}})^{-1}})g_{ \gamma^\phi\sigma^{\mathbb{I}_{G}}}n_{ \gamma^\phi\sigma^{\mathbb{I}_{G}}}}.$$

 Therefore there exists an $n' \in N$ such that 
 \begin{equation}
 	 n'=(g_{ \gamma^\phi\sigma^{\mathbb{I}_{G}}}n_{ \gamma^\phi\sigma^{\mathbb{I}_{G}}} )^{-1}\phi(\tilde{n}_{( \gamma^\phi\sigma^{\mathbb{I}_{G}})^{-1}})g_{ \gamma^\phi\sigma^{\mathbb{I}_{G}}}n_{ \gamma^\phi\sigma^{\mathbb{I}_{G}}}
 \end{equation} 
 so 
 \begin{equation}
 	 f_{g_{ \gamma^\phi\sigma^{\mathbb{I}_{G}}}n S}f_{\tilde{n}S}(f_{g_{ \gamma^\phi\sigma^{\mathbb{I}_{G}}}n S})^{-1}=f_{n'S}
 \end{equation}

 hence the thesis.

\end{proof} 

We conclude the study of the subgroup $<f_{nS}>$ by demonstrating the following property.

\begin{prop}
	\quad	
	$	<f_{nS}>\cong \prod_{i=1}^{t} N_{i}/S_{i}=\prod_{i=1}^{t} N_{G}(Stab_{G}(x_{0}^{i}))/Stab_{G}(x_{0}^{i})$. 
\end{prop}
\begin{proof}
Given the following map

$$	\hspace{0.8cm}<f_{nS}>\stackrel{\mathcal{I}}{\longrightarrow}\prod_{i=1}^{t} N_{i}/S_{i}$$
\quad
$$\hspace{3cm} f_{nS}\mapsto n^{-1}S=(n_{i}^{-1}S_{i})_{i=1}^{t} ,$$

the strategy is to prove that $\mathcal{I}$ is a isomorphism of group: first it is a well-defined and injective	map.

$$ f_{nS}= f_{\tilde{n}S} \Rightarrow (f_{\tilde{n}S})^{-1}f_{nS} =\mathbb{I}_{X}=f_{S}$$

by (3.1) and (3.2)
$$f_{{\tilde{n}}^{-1}S}f_{nS}=f_{n{\tilde{n}}^{-1}S}=f_{S} .$$

This implies $n{\tilde{n}}^{-1}=s$ for a $ s \in S$ then $ n=s\tilde{n}$.
\quad
From this you get

$$ \mathcal{I}(f_{nS})=n^{-1}S=(s\tilde{n})^{-1}S={\tilde{n}}^{-1}s^{-1}S=$$
$$={\tilde{n}}^{-1}S=\mathcal{I}(f_{\tilde{n}S}).$$

So the map $\mathcal{I}$ is well defined; for injectivity is enough , in that case, logically proceed back the steps made so far.

\quad

The suriettivity of $\mathcal{I}$, as the map is defined, is immediate. So to conclude the demonstration we show that $\mathcal{I}$ is a homomorphism of groups.

$$\mathcal{I}(f_{nS}f_{\tilde{n}S})=\mathcal{I}(f_{\tilde{n}nS})=$$

 $$=(\tilde{n}n)^{-1}S=n^{-1}{\tilde{n}}^{-1}S=n^{-1}S{\tilde{n}}^{-1}S=\mathcal{I}(f_{nS})\mathcal{I}(f_{\tilde{n}S}).$$

\end{proof}

Proposition (3.2) leads us to an immediate consequence.

\begin{coro}
	Unless isomorphisms $$ \prod_{i=1}^{t} N_{G}(Stab_{G}(x_{0}^{i}))/Stab_{G}(x_{0}^{i})\unlhd  \widehat{{Aut(G)}^{S}} .$$
\end{coro}

Now we study the quotient group $\widehat{{Aut(G)}^{S}}/<f_{nS}>$; we try to consider the set of elements $<f_{nS}>f_{g_{\sigma^{\mathbb{I}_{G}}}S}$ as sigma $\sigma^{\mathbb{I}_{G}}$  varies in $Sym(I)^{\{\mathbb{I}_{G}  \}}$ . Such a set is a subgroup of $\widehat{{Aut(G)}^{S}}/<f_{nS}>$  in fact

$$<f_{nS}>f_{g_{\sigma^{\mathbb{I}_{G}}}S}<f_{nS}>f_{g_{\tau^{\mathbb{I}_{G}}}S}=<f_{nS}>f_{g_{\sigma^{\mathbb{I}_{G}}}S}f_{g_{\tau^{\mathbb{I}_{G}}}S}.$$

So let's calculate $f_{g_{\sigma^{\mathbb{I}_{G}}}S}f_{g_{\tau^{\mathbb{I}_{G}}}S}.$

$$f_{g_{\sigma^{\mathbb{I}_{G}}}S}f_{g_{\tau^{\mathbb{I}_{G}}}S}(g\cdot x_0)=f_{g_{\sigma^{\mathbb{I}_{G}}}S}(gg_{\tau^{\mathbb{I}_{G}}}\cdot x_{0}^{\tau^{\mathbb{I}_{G}}} )=$$

$$=gg_{\tau^{\mathbb{I}_{G}}}g_{\sigma^{\mathbb{I}_{G}}}\cdot x_{0}^{\sigma^{\mathbb{I}_{G}}\tau^{\mathbb{I}_{G}}}=gg_{\tau^{\mathbb{I}_{G}}}g_{\sigma^{\mathbb{I}_{G}}}\cdot x_{0}^{\sigma\tau^{\mathbb{I}_{G}}}.$$

Since we had fixed from the beginning of the section ,arbitrarily,  an element $g_{\sigma\tau^{\mathbb{I}_{G}}}$, associated with the map $f_{g_{\sigma\tau^{\mathbb{I}_{G}}}S}$, then we will have that

\begin{equation}
	g_{\tau^{\mathbb{I}_{G}}}g_{\sigma^{\mathbb{I}_{G}}}=g_{\sigma\tau^{\mathbb{I}_{G}}}n_{\sigma\tau^{\mathbb{I}_{G}}}	 
\end{equation}

for an appropriate $n_{\sigma\tau^{\mathbb{I}_{G}}} \in N_{\sigma\tau^{\mathbb{I}_{G}}}$.

\quad

So we have that

$$f_{g_{\sigma^{\mathbb{I}_{G}}}S}f_{g_{\tau^{\mathbb{I}_{G}}}S}(g\cdot x_0)=gg_{\sigma\tau^{\mathbb{I}_{G}}}n_{\sigma\tau^{\mathbb{I}_{G}}}\cdot x_{0}^{\sigma\tau^{\mathbb{I}_{G}}}=$$

	$$=f_{nS}f_{g_{\sigma\tau^{\mathbb{I}_{G}}}S}(g\cdot x_0)$$
	\quad
	and thus
	\begin{equation}
		f_{g_{\sigma^{\mathbb{I}_{G}}}S}f_{g_{\tau^{\mathbb{I}_{G}}}S}=f_{nS}f_{g_{\sigma\tau^{\mathbb{I}_{G}}}S}.
	\end{equation}
Moreover, by (3.7) we have that

\begin{equation}
<f_{nS}>f_{g_{\sigma^{\mathbb{I}_{G}}}S}<f_{nS}>f_{g_{\tau^{\mathbb{I}_{G}}}S}=<f_{nS}>f_{g_{\sigma\tau^{\mathbb{I}_{G}}}S}.
\end{equation} 

It remains to prove the existence of the inverse: since

$$(f_{g_{\sigma^{\mathbb{I}_{G}}}S})^{-1}=f_{(g_{\sigma^{\mathbb{I}_{G}}})^{-1}S}=f_{\tilde{n}S}f_{g_{\bar{\sigma}^{\mathbb{I}_{G}}}S}.$$

 for an opportune $f_{\tilde{n}S} \in <f_{nS}>$.
 
 Hence
 
 \begin{equation}
 	(<f_{nS}>f_{g_{\sigma^{\mathbb{I}_{G}}}S})^{-1}=<f_{nS}>f_{g_{\bar{\sigma}^{\mathbb{I}_{G}}}S}.
 \end{equation}

Finally, we have proved the following result.

\begin{prop}
	\quad
	$\{ <f_{nS}> f_{g_{\sigma^{\mathbb{I}_{G}}}S}	\}=< <f_{nS}> f_{g_{\sigma^{\mathbb{I}_{G}}}S}	>\leq \widehat{{Aut(G)}^{S}}/<f_{nS}>.$
\end{prop} 

\quad

For what has been said about  the $<<f_{nS}>f_{g_{\sigma^{\mathbb{I}_{G}}}S}>$ subgroup, we note another important property which proves the following result.

\begin{prop}
\quad		$< <f_{nS}> f_{g_{\sigma^{\mathbb{I}_{G}}}S}	>\unlhd \widehat{{Aut(G)}^{S}}/<f_{nS}>$
	
\end{prop}
\begin{proof}
Before beginning the demonstration, it is appropriate to make a few preliminary remarks: given a generic element $f_{g_{ \gamma^\phi\sigma^{\mathbb{I}_{G}}}n S} \in \widehat{{Aut(G)}^{S}}$ it can be rewritten in the following way

$$ f_{g_{ \gamma^\phi\sigma^{\mathbb{I}_{G}}}n S}= f_{nS}f_{g_{ \gamma^\phi\sigma^{\mathbb{I}_{G}}}S}.$$

In turn, the element $f_{g_{ \gamma^\phi\sigma^{\mathbb{I}_{G}}}S}$ can be rewritten in such a way, for an opportune $f_{\tilde{n}S} \in <f_{nS}> $.

$$
	f_{g_{ \gamma^\phi\sigma^{\mathbb{I}_{G}}}S}=f_{\tilde{n}S}f_{g_{ \gamma^\phi} S}	f_{g_{\sigma^{\mathbb{I}_{G}}}S}.
$$
 To show this it suffices to use argumentation similar to that made in (3.7).
 In the end, if we place $n'=\tilde{n}n$, we will get
 \begin{equation}
 	f_{g_{ \gamma^\phi\sigma^{\mathbb{I}_{G}}}n S}=f_{n'S}f_{g_{ \gamma^\phi} S}	f_{g_{\sigma^{\mathbb{I}_{G}}}S}.
 \end{equation}
Based on the observations done so far, added the fact of normality of $<f_{nS}>$, we have that 

$$<f_{nS}>f_{g_{ \gamma^\phi\tau^{\mathbb{I}_{G}}}n' S}<f_{nS}>f_{g_{\sigma^{\mathbb{I}_{G}}}S}<f_{nS}>(f_{g_{ \gamma^\phi\tau^{\mathbb{I}_{G}}}n' S})^{-1}=$$ 

$$=<f_{nS}>f_{g_{ \gamma^\phi} S}	f_{g_{\tau^{\mathbb{I}_{G}}}S}<f_{nS}>f_{g_{\sigma^{\mathbb{I}_{G}}}S}<f_{nS}>(f_{g_{ \gamma^\phi} S}	f_{g_{\tau^{\mathbb{I}_{G}}}S})^{-1}=$$

$$=<f_{nS}>f_{g_{ \gamma^\phi} S}f_{g_{\tau^{\mathbb{I}_{G}}}S}f_{g_{\sigma^{\mathbb{I}_{G}}}S}	(f_{g_{ \gamma^\phi} S}f_{g_{\tau^{\mathbb{I}_{G}}}S})^{-1}=<f_{nS}>f_{g_{ \gamma^\phi}} f_{g_{\tau^{\mathbb{I}_{G}}}S}f_{g_{\sigma^{\mathbb{I}_{G}}}S}(f_{g_{\tau^{\mathbb{I}_{G}}}S})^{-1}(f_{g_{ \gamma^\phi} S})^{-1}.$$

And since  $f_{g_{\tau^{\mathbb{I}_{G}}}S}f_{g_{\sigma^{\mathbb{I}_{G}}}S}(f_{g_{\tau^{\mathbb{I}_{G}}}S})^{-1}$ is an element of type $f_{g_{\tilde{\sigma}^{\mathbb{I}_{G}}}S}$ modulo $<f_{nS}>$ ; to conclude the demonstration, just calculate $f_{g_{ \gamma^\phi}S}f_{g_{\tilde{\sigma}^{\mathbb{I}_{G}}}S}(f_{g_{ \gamma^\phi} S})^{-1}$ and see that it is an element $f_{g_{\tilde{\tau}^{\mathbb{I}_{G}}}S}$ modulo $<f_{nS}>$.

$$f_{g_{ \gamma^\phi}S}f_{g_{\tilde{\sigma}^{\mathbb{I}_{G}}}S}(f_{g_{ \gamma^\phi} S})^{-1}(g\cdot x_0)=g(g_{ \gamma^\phi})^{-1}\phi(g_{\tilde{\sigma}^{\mathbb{I}_{G}}})g_{ \gamma^\phi}x_{0}^{\gamma^\phi\tilde{\sigma}^{\mathbb{I}_{G}}(\gamma^\phi)^{-1}}.$$
\quad

We know that $\gamma^\phi\tilde{\sigma}^{\mathbb{I}_{G}}(\gamma^\phi)^{-1}=\tilde{\tau}^{\mathbb{I}_{G}} \in Sim(I)^{\{{\mathbb{I}_{G}}\}}$  so

 $$(g_{ \gamma^\phi})^{-1}\phi(g_{\tilde{\sigma}^{\mathbb{I}_{G}}})g_{ \gamma^\phi}=g_{\tilde{\tau}^{\mathbb{I}_{G}}}\tilde{n}_{\tilde{\tau}^{\mathbb{I}_{G}}},$$
 
 for some $\tilde{n} \in N$.
 
 \quad
 Therefore
 
 $$f_{g_{ \gamma^\phi}S}f_{g_{\tilde{\sigma}^{\mathbb{I}_{G}}}S}(f_{g_{ \gamma^\phi} S})^{-1}=f_{\tilde{n}S}f_{g_{\tilde{\tau}^{\mathbb{I}_{G}}}S }.$$

And we  finished.

\end{proof}

In truth, there is an additional property regarding the subgroup $ < <f_{nS}> f_{g_{\sigma^{\mathbb{I}_{G}}}S}	>$.

\begin{prop}
\quad $	< <f_{nS}> f_{g_{\sigma^{\mathbb{I}_{G}}}S}	> \cong Sim(I)^{\{{\mathbb{I}_{G}}\}}$
\end{prop}
\begin{proof}
The isomorphism we want to prove is as follows.
	$$	\hspace{0.8cm}< <f_{nS}> f_{g_{\sigma^{\mathbb{I}_{G}}}S}	>\stackrel{R}{\longrightarrow}Sim(I)^{\{{\mathbb{I}_{G}}\}}$$
	\quad
	$$\hspace{0.2cm} <f_{nS}> f_{g_{\sigma^{\mathbb{I}_{G}}}S} \mapsto \sigma^{\mathbb{I}_{G}} .$$

The map $R$ is well-defined: in fact if

$$f_{g_{\sigma^{\mathbb{I}_{G}}}S}=f_{g_{\tau^{\mathbb{I}_{G}}}S} \Rightarrow \sigma^{\mathbb{I}_{G}}=\tau^{\mathbb{I}_{G}}.$$

This is due to the fact that $x_{0}^{i}$ representatives are unique, unless there are equivalence relations due to group action on $X$.
\qquad

Let’s move on to the injectivity of R.

$$ R(<f_{nS}>f_{g_{\sigma^{\mathbb{I}_{G}}}S})=R(<f_{nS}>f_{g_{\tau^{\mathbb{I}_{G}}}S})\Rightarrow \sigma^{\mathbb{I}_{G}}=\tau^{\mathbb{I}_{G}} \Rightarrow x_{0}^{\sigma^{\mathbb{I}_{G}}}=x_{0}^{\tau^{\mathbb{I}_{G}}}, g_{\sigma^{\mathbb{I}_{G}}}=g_{\tau^{\mathbb{I}_{G}}} \Rightarrow f_{g_{\sigma^{\mathbb{I}_{G}}}S}=f_{g_{\tau^{\mathbb{I}_{G}}}S}.$$

The fact that $R$ is surjective is immediate so that $R$ is bijective. There remains only one to prove that it is a homomorphism

$$R(<f_{nS}>f_{g_{\sigma^{\mathbb{I}_{G}}}S}<f_{nS}>f_{g_{\tau^{\mathbb{I}_{G}}}S})=R(<f_{nS}>f_{g_{\sigma^{\mathbb{I}_{G}}}S}f_{g_{\tau^{\mathbb{I}_{G}}}S})=$$
$$=R(<f_{nS}>f_{\tilde{n}S}f_{g_{\sigma\tau^{\mathbb{I}_{G}}}S})=\sigma\tau^{\mathbb{I}_{G}}=\sigma^{\mathbb{I}_{G}}\tau^{\mathbb{I}_{G}}=     $$
$$=R(<f_{nS}>f_{g_{\sigma^{\mathbb{I}_{G}}}S})R(<f_{nS}>f_{g_{\tau^{\mathbb{I}_{G}}}S})  $$

\end{proof}

This result leads us to an immediate consequence.

\begin{coro}
	Unless isomorphisms $$Sim(I)^{\{{\mathbb{I}_{G}}\}} \unlhd  \widehat{{Aut(G)}^{S}}/\prod_{i=1}^{t} N_{G}(Stab_{G}(x_{0}^{i}))/Stab_{G}(x_{0}^{i}) $$
\end{coro}

In the analysis of the  group $\widehat{{Aut(G)}^{S}}$, we begin to study a final quotient group $(\widehat{{Aut(G)}^{S}}/<f_{nS}>)/<<f_{nS}>f_{g_{\sigma^{\mathbb{I}_{G}}}S}>$:given the following corrispondence 

\quad

$$	\hspace{0.8cm}Aut(G)^{S}\stackrel{\mathcal{C}}{\longrightarrow}(\widehat{{Aut(G)}^{S}}/<f_{nS}>)/<<f_{nS}>f_{g_{\sigma^{\mathbb{I}_{G}}}S}>$$
\quad
$$\hspace{0.7cm} \phi \mapsto <f_{nS}>f_{g_{\gamma^\phi}S}<<f_{nS}>f_{g_{\sigma^{\mathbb{I}_{G}}}S}> .$$

 It is easy to check, since the corrispondence is defined, that $ \mathcal{C}$ is well-defined, and then it's immediate to prove that $ \mathcal{C}$ is surjective . 
 
 \quad 
 
Let's now prove that $ \mathcal{C}$ is homomorphism:

$$ \mathcal{C}(\phi\psi)=<f_{nS}>f_{g_{\gamma^{\phi\psi}}S}<<f_{nS}>f_{g_{\sigma^{\mathbb{I}_{G}}}S}>=<f_{nS}>f_{g_{\lambda^{\phi}}S} <f_{nS}>f_{g_{\tau^{\psi}}S}<<f_{nS}>f_{g_{\sigma^{\mathbb{I}_{G}}}S}>=$$

$$=<f_{nS}>f_{g_{\lambda^{\phi}}S}<<f_{nS}>f_{g_{\sigma^{\mathbb{I}_{G}}}S}><f_{nS}>f_{g_{\tau^{\psi}}S}<<f_{nS}>f_{g_{\sigma^{\mathbb{I}_{G}}}S}>=\mathcal{C}(\phi)\mathcal{C}(\psi)   . $$ 

Where $\lambda^{\phi}$ and $\tau^{\psi} \in Sim(I)^{{Aut(G)}^{S}}$ are such that $\gamma^{\phi\psi}=\lambda^{\phi}\tau^{\psi}$.

\quad

As a result we have the following isomorphism.

\begin{equation}
Aut(G)^{S}/Ker\mathcal{C}\cong (\widehat{{Aut(G)}^{S}}/<f_{nS}>)/<<f_{nS}>f_{g_{\sigma^{\mathbb{I}_{G}}}S}>
\end{equation}

and	that is

\begin{equation}
	Aut(G)^{S}/Ker\mathcal{C}\cong (\widehat{{Aut(G)}^{S}}/\prod_{i=1}^{t} N_{G}(Stab_{G}(x_{0}^{i}))/Stab_{G}(x_{0}^{i}))/Sim(I)^{\{{\mathbb{I}_{G}}\}}. 
\end{equation}

We now describe the $Ker\mathcal{C}$ subgroup in detail by proving the following result.

\begin{prop}
	\quad $Ker\mathcal{C}=\{ \phi \in Aut(G)^{S} |   \phi(S)=S , \phi(g)=gs \hspace{0.2cm} \forall g \in (G)^{t} \hspace{0.2cm} with \hspace{0.2cm} s \in S     \}$
\end{prop}
\begin{proof}
	If $\phi \in Ker\mathcal{C}$ then necessarily, as $\mathcal{C}$ is defined, we will have that $f_{g_{\gamma^{\phi}}S}=f_{nS}f_{g_{\sigma^{\mathbb{I}_{G}}}S}$ for some $n \in N$ and $\sigma^{\mathbb{I}_{G}} \in Sim(I)^{\{{\mathbb{I}_{G}}\}} $.  Thus we get that $\phi(g)g_{\gamma^{\phi}}\cdot x_{0}^{\gamma^{\phi}}=gg_{\sigma^{\mathbb{I}_{G}}}n_{\sigma^{\mathbb{I}_{G}}}\cdot x_{0}^{\sigma^{\mathbb{I}_{G}}} \hspace{0.2cm}\forall g \in (G)^{t} $, namely, we will have for $g=\bar{1}$ that $g_{\gamma^{\phi}}\cdot x_{0}^{\gamma^{\phi}}=g_{\sigma^{\mathbb{I}_{G}}}n_{\sigma^{\mathbb{I}_{G}}}\cdot x_{0}^{\sigma^{\mathbb{I}_{G}}}$; since the representatives $x_{0}^{i}$ once fixed are unique module of the action on G, 
	this implies that $\gamma^{\phi}=\sigma^{\mathbb{I}_{G}}$(in particular, we note that $Ker\mathcal{C}\leq KerF$ ) and so $g_{\gamma^{\phi}}=g_{\sigma^{\mathbb{I}_{G}}}n_{\sigma^{\mathbb{I}_{G}}}s_{\sigma^{\mathbb{I}_{G}}}$ for some $s_{\sigma^{\mathbb{I}_{G}}} \in S_{\sigma^{\mathbb{I}_{G}}}$; we infer that
	
	$$\phi(S)=S_{\gamma^{\phi}}^{g_{\gamma^{\phi}}}=S_{\sigma^{\mathbb{I}_{G}}}^{g_{\gamma^{\phi}}}=S_{\sigma^{\mathbb{I}_{G}}}^{g_{\sigma^{\mathbb{I}_{G}}}}=S.$$
	
	Now if it is not $ g_{\gamma^{\phi}}$ that coincides with  $\bar{1}$ there will still exist another $\phi$-invariant function of the type $f_{\phi,S}$; since we have just seen that $\phi$ fixes the stabilizers. And since $\phi \in Ker\mathcal{C}$ we have that $f_{\phi,S}=f_{\tilde{n}S} $ for some $\tilde{n} \in N$.
		
	\quad
	Thus
	
	$$ \phi(g)\cdot x_{0}=g\tilde{n}\cdot x_{0}.$$
	
	Especially for $g=\bar{1}$, we get.
	
	$$ x_0= \tilde{n}\cdot x_{0}.$$
	
	This implies $\tilde{n}=s$ for some $s \in S$, and for therefore
	\begin{equation}
	f_ {\phi,S}=f_{S}.
	\end{equation}
	 
	And the only possibility for $\phi$ to satisfy (3.13) is that for every $g \in (G)^t$ there exists an $s \in S$ such that $\phi(g)=gs$.
	
	\quad
	So we have shown 
	\begin{equation}
			 Ker\mathcal{C}\subseteq\{ \phi \in Aut(G)^{S} |   \phi(S)=S , \phi(g)=gs \hspace{0.2cm} \forall g \in (G)^{t} \hspace{0.2cm} with \hspace{0.2cm} s \in S     \}.     
\end{equation}
	Showing the inverse inclusion to (3.14), because of the way the set is defined at the second member, is almost immediate. 
	Hence we obtain the desired equality.
	
\end{proof}

\quad

The automorphisms $\phi \in \ker\mathcal{C}$ can be called $quasi-identical\hspace{0.2cm} atuomorphisms \hspace{0.2cm}  over \hspace{0.2cm} S$ and we can call the set \begin{equation}
	Ker\mathcal{C}:=\mathbb{I}_{G}^{S}.
\end{equation}

Summarizing, we have that 
\begin{equation}
\mathbb{I}_{G}^{S} \leq  KerF\hspace{0.2cm} and \hspace{0.2cm} that \hspace{0.2cm} \mathbb{I}_{G}^{S} \unlhd Aut(G)^S ,\hspace{0.2cm}  Aut(G)^{S}/\mathbb{I}_{G}^{S} \cong (\widehat{{Aut(G)}^{S}}/\prod_{i=1}^{t} N_{G}(Stab_{G}(x_{0}^{i}))/Stab_{G}(x_{0}^{i}))/Sim(I)^{\{{\mathbb{I}_{G}}\}} 
\end{equation}

\qquad

Returning to the last isomorphism of (3.16), the optimal situation would be that $$\widehat{{Aut(G)}^{S}}/\prod_{i=1}^{t} N_{G}(Stab_{G}(x_{0}^{i})/Stab_{G}(x_{0}^{i}) \cong Sim(I)^{\{{\mathbb{I}_{G}}\}} \rtimes({Aut(G)}^{S}/\mathbb{I}_{G}^{S} )$$ i.e. a suitable semi-direct product. But in general things are not always so, and this may depend on the choice of $\gamma^\phi$ elements of $Sim(I)^{Aut(G)^{S}}$, which we have so far set arbitrarily if such $\gamma^\phi \neq \mathbb{I}_{I}$ or if $\phi \notin \mathbb{I}_{G}^{S} $; our intention is to add appropriate assumptions about such choices of elements, but before describing these assumptions it is necessary to index the group $Aut(G)^{S}=\{ \phi_{i}\}$ for $i \in \tilde{I}=\{1,2,...,i,...,\tilde{t}\}$, in particular we pose  $\phi_{0}=\mathbb{I}_{G}$,and in addition, as we will see later, it is necessary to set $\gamma_{i}^{\phi_{i}}=\mathbb{I}_{I}$ and $g_{\gamma_{i}^{\phi_{i}}}=\bar1 \hspace{0.2cm} \forall \phi_{i} \in \mathbb{I}_{G}^{S} $. Now we are ready to describe the following hypothesis.

\qquad

\begin{description}
	\item[Assumption (I)]

It is possible to fix the following elements $\gamma_{i}^{\phi_{i}}, \gamma_{j}^{\phi_{j}}\in Sim(I)^{Aut(G)^{S}}$, such that we get
\begin{itemize}
	\item 
$\gamma_{i}^{\phi_{i}}\gamma_{j}^{\phi_{j}}=\gamma_{i_{j}}^{\phi_{i_{j}}}$
\end{itemize}
for some $i_{j} \in \tilde{I}$ with $\phi_{i_{j}}=\phi_{i}\phi_{j}$.
\begin{itemize}
\item
$(\gamma_{i}^{\phi_{i}})^{-1}=\gamma_{\bar{i}}^{\phi_{\bar{i}}}$
\end{itemize}
for some $\bar{i} \in \tilde{I}$ with $\phi_{\bar{i}}={\phi_{i}}^{-1}$.
\end{description}
\quad

If I adopt hypothesis (I) the set $\{<f_{nS}>f_{g^{\gamma_{i}^{\phi_{i}}}S}| \phi_{i} \in Aut(G)^{S} \}$ is first a subgroup of $Aut(G)^S/<f_{nS}>$ and then is isomorphic to $Aut(G)^S/{\mathbb{I}_G}^S$. In fact

$$<f_{nS}>f_{g_{\gamma_{i}^{\phi_{i}}}S}<f_{nS}>f_{g_{\gamma_{j}^{\phi_{j}}}S}=<f_{nS}>f_{g_{\gamma_{i}^{\phi_{i}}}S}f_{g_{\gamma_{j}^{\phi_{j}}}S}=$$

$$=<f_{nS}>f_{g_{\gamma_{i}^{\phi_{i}}\gamma_{j}^{\phi_{j}}}S}=<f_{nS}>f_{g_{\gamma_{i_{j}}^{\phi_{i_{j}}}S}}$$

and then

$$(<f_{nS}>f_{g_{\gamma_{i}^{\phi_{i}}}S})^{-1}=<f_{nS}>(f_{g_{\gamma_{i}^{\phi_{i}}}S})^{-1}=<f_{nS}>f_{g_{(\gamma_{i}^{\phi_{i}})^{-1}}S}=<f_{nS}>f_{g_{\gamma_{\bar{i}}^{\phi_{\bar{i}}}}S}.$$

And now we show the isomorphism between $Aut(G)^S/{\mathbb{I}_G}^S$  and the group
$<<f_{nS}>f_{g_{\gamma_{i}^{\phi_{i}}}S}>$  via the following map 

\quad

$$	\hspace{0.8cm}Aut(G)^{S}\stackrel{\tilde{\mathcal{C}}}{\longrightarrow}<<f_{nS}>f_{g_{\gamma_{i}^{\phi_{i}}}S}>$$
\quad
$$\hspace{1cm} \phi_{i} \mapsto <f_{nS}>f_{g_{\gamma_{i}^{\phi_{i}}}S} .$$

It is soon said that the map is well defined and surjective to how $\gamma_{i}^{\phi_{i}}$ were chosen in hypothesis (I); furthermore $\tilde{\mathcal{C}}$ is a homomorphism

$$\tilde{\mathcal{C}}(\phi_{i}\phi_{j})=\tilde{\mathcal{C}}(\phi_{i_{j}})=$$

$$=<f_{nS}>f_{g_{\gamma_{i_{j}}^{\phi_{i_{j}}}}S}=<f_{nS}>f_{g_{\gamma_{i}^{\phi_{i}}}S}f_{g_{\gamma_{j}^{\phi_{j}}}S}$$

$$=<f_{nS}>f_{g_{\gamma_{i}^{\phi_{i}}}S}<f_{nS}>f_{g_{\gamma_{j}^{\phi_{j}}}S}=\tilde{\mathcal{C}}(\phi_{i})\tilde{\mathcal{C}}(\phi_{j})$$

To satisfay isomorphism, it's sufficient to demonstrate that $Ker\tilde{\mathcal{C}}=\mathbb{I}_{G}^{S}$; if $\phi_{i} \in Ker\tilde{\mathcal{C}}$ we have that $f_{g_{\gamma_{i}^{\phi_{i}}}S}=f_{nS}$ implies that $g_{\gamma_{j}^{\phi_{j}}} \in N$, for how we chose $g_{\gamma_{i}^{\phi_{i}}}$  at the beginning of section 3, we have necessarily that $g_{\gamma_{i}^{\phi_{i}}}=\bar{1}$. So we have that  $f_{g_{\gamma_{i}^{\phi_{i}}}S}=f_{S}$ but $\phi_{i}\neq \mathbb{I}_{G}$; so we must have $\phi_{i}(g)=gs$ for some $s \in S$ that is $\phi_{i} \in \mathbb{I}_{G}^{S}$.In addition to how we chose  for $\phi_{i} \in \mathbb{I}_{G}^{S}$ we have $f_{g_{\gamma_{i}^{\phi_{i}}}S}=f_{S}$ for which we get the desired result. 

\begin{equation}
	<<f_{nS}>f_{g_{\gamma_{i}^{\phi_{i}}}S}>\cong Aut(G)^S/{\mathbb{I}_G}^S 
\end{equation}

And now we can prove the following assertion. 

\begin{prop}
	If assumption (I) holds, then $\widehat{{Aut(G)}^{S}}/<f_{nS}>\cong <<f_{nS}>f_{g_{\sigma^{\mathbb{I}_{G}}}S}>\rtimes <<f_{nS}>f_{g_{\gamma_{i}^{\phi_{i}}}S}>$ that is, unless isomorphisms $\widehat{{Aut(G)}^{S}}/\prod_{i=1}^{t} N_{G}(Stab_{G}(x_{0}^{i}))/Stab_{G}(x_{0}^{i}) \cong Sim(I)^{\{{\mathbb{I}_{G}}\}} \rtimes({Aut(G)}^{S}/\mathbb{I}_{G}^{S} )$ 
	 for an opportune semi-product
\end{prop}
 \begin{proof}
 	Thanks to Proposition (3.6) and what we have done so far, in order to arrive at the thesis we need only prove that
 	
 	$$<<f_{nS}>f_{g_{\sigma^{\mathbb{I}_{G}}}S}>\bigcap <<f_{nS}>f_{g_{\gamma_{i}^{\phi_{i}}}S}>=<f_{nS}>  $$
 	 
 	 and
 	
 	$$<<f_{nS}>f_{g_{\sigma^{\mathbb{I}_{G}}}S}><<f_{nS}>f_{g_{\gamma_{i}^{\phi_{i}}}S}>=\widehat{{Aut(G)}^{S}}/<f_{nS}>,$$
 	
 	if $<f_{nS}>f_{g_{\gamma_{i}^{\phi_{i}}}S} \in <<f_{nS}>f_{g_{\sigma^{\mathbb{I}_{G}}}S}>\bigcap <<f_{nS}>f_{g_{\gamma_{i}^{\phi_{i}}}S}>$ then $<f_{nS}>f_{g_{\gamma_{i}^{\phi_{i}}}S}=<f_{nS}>f_{g_{\sigma^{\mathbb{I}_{G}}}S}$, so $\phi_{i} \in {\mathbb{I}_G}^S $; it implies that $<f_{nS}>f_{g_{\gamma_{i}^{\phi_{i}}}S}=<f_{nS}>$.
 	And so the first equality is proven.
 	
 	\qquad
 	
 	As for the second equality, it is sufficient to prove
 	
 $\widehat{{Aut(G)}^{S}}/<f_{nS}> \subseteq <<f_{nS}>f_{g_{\sigma^{\mathbb{I}_{G}}}S}><<f_{nS}>f_{g_{\gamma_{i}^{\phi_{i}}}S}>.$
 	
 	Since the inverse inclusion is immediate, but in truth we have already proved this in the beginning of the proof of Proposition (3.5). So we have concluded the demonstration.

 \end{proof}

\quad

We conclude the first part of this section by, further, the factoring  the $\widehat{{Aut(G)}^{S}}$  group by adding additional appropriate assumptions to the previous ones.

\quad

\begin{description}
	\item[Assumption (II)]

	There is a relation 
	
	$$	\hspace{0.8cm}Sim(I)^{Aut(G)}\stackrel{\mathcal{G}}{\longrightarrow}(G)^{t}$$
	\quad
	$$\hspace{2.5cm} \gamma^{\phi}\sigma^{\mathbb{I}_{G}} \mapsto g_{\gamma^{\phi}\sigma^{\mathbb{I}_{G}}} $$
	
	such that

	\begin{itemize}
		\item 
		$g_{(\gamma^{\phi}\sigma^{\mathbb{I}_{G}})^{-1}}=\phi^{-1}((g_{\gamma^{\phi}\sigma^{\mathbb{I}_{G}}})^{-1})s_{(\gamma^{\phi}\sigma^{\mathbb{I}_{G}})^{-1}}$
	\end{itemize}
	for some $s_{(\gamma^{\phi}\sigma^{\mathbb{I}_{G}})^{-1}} \in S_{(\gamma^{\phi}\sigma^{\mathbb{I}_{G}})^{-1}} $.
	\begin{itemize}
		\item
		$\phi(g_{\gamma^{\psi}\tau^{\mathbb{I}_{G}}})g_{\gamma^{\phi}\sigma^{\mathbb{I}_{G}}}=g_{\gamma^{\phi}\sigma^{\mathbb{I}_{G}}\gamma^{\psi}\tau^{\mathbb{I}_{G}}}s_{\gamma^{\phi}\sigma^{\mathbb{I}_{G}}\gamma^{\psi}\tau^{\mathbb{I}_{G}}}$
	\end{itemize}
	for some $s_{\gamma^{\phi}\sigma^{\mathbb{I}_{G}}\gamma^{\psi}\tau^{\mathbb{I}_{G}}} \in S_{\gamma^{\phi}\sigma^{\mathbb{I}_{G}}\gamma^{\psi}\tau^{\mathbb{I}_{G}}} $.
\end{description}

\quad

\begin{os}
	As we saw at the beginning of section 3 we have that $g_{\mathbb{I}_{I}}=\bar{1}$. Moreover, the relation $\mathcal{G}$ relation may not be a well-defined map; in fact, for example, a $\phi \in KerF$ with $\phi\neq \mathbb{I}_{G}$ could have a $\gamma^{\phi}\sigma^{\mathbb{I}_{G}}=\mathbb{I}_{I}$ but a $g_{\gamma^{\phi}\sigma^{\mathbb{I}_{G}}}\neq \bar{1}$ (think of the functions that fix the orbits in a non-trivial way). If there is a need in this case we will specify  $g_{\gamma^{\phi}\sigma^{\mathbb{I}_{G}}}=g_{\phi,\mathbb{I}_{I}}$, to distinguish it from $g_{\mathbb{I}_{I}}$.
	This reasoning can be done in general for every $\phi \in kerF$ that is, if $\gamma^{\phi}\sigma^{\mathbb{I}_{G}}=\tau^{\mathbb{I}_{G}}$ but $g_{\gamma^{\phi}\sigma^{\mathbb{I}_{G}}}\neq g_{\tau^{\mathbb{I}_{G}}}$  in that case  then we will have that $g_{\gamma^{\phi}\sigma^{\mathbb{I}_{G}}}=g_{\phi,\tau^{\mathbb{I}_{G}}}$. 
	Conversely, it also can happen that  $g_{\gamma^{\phi}\sigma^{\mathbb{I}_{G}}}=g_{\gamma^{\psi}\tau^{\mathbb{I}_{G}}}$ but $\gamma^{\phi}\sigma^{\mathbb{I}_{G}}\neq \gamma^{\psi}\tau^{\mathbb{I}_{G}}$.
\end{os}

\quad
By means of Assumption (II) it is verified that the set $\{  f_{g_{\gamma^{\phi}\sigma^{\mathbb{I}_{G}}}S} \}$ is a subgroup of $ \widehat{{Aut(G)}^{S}}$. In fact

$$f_{g_{\gamma^{\phi}\sigma^{\mathbb{I}_{G}}}S}f_{g_{\lambda^{\psi}\tau^{\mathbb{I}_{G}}}S}=f_{\phi(g_{\lambda^{\psi}\tau^{\mathbb{I}_{G}}})g_{\gamma^{\phi}\sigma^{\mathbb{I}_{G}}}S}= $$

$$=f_{g_{\gamma^{\phi}\sigma^{\mathbb{I}_{G}}\lambda^{\psi}\tau^{\mathbb{I}_{G}}}S}=f_{g_{\gamma\lambda^{\phi\psi}\tilde{\sigma}^{\mathbb{I}_{G}} }S},$$

for an opportune $\tilde{\sigma}^{\mathbb{I}_{G}} \in Sim(I)^{\{\mathbb{I}_{G}\}}$ . Furthermore

$$(f_{g_{\gamma^{\phi}\sigma^{\mathbb{I}_{G}}}S})^{-1}=f_{\phi((g_{\gamma^{\phi}\sigma^{\mathbb{I}_{G}}})^{-1})^{-1}S}=f_{g_{(\gamma^{\phi}\sigma^{\mathbb{I}_{G}})^{-1}}S}.$$
 \quad
 
Now we are going to prove the following lemma.

\begin{lem}
If Assumption (II) holds, then	$\widehat{{Aut(G)}^{S}}\cong<f_{nS}>\rtimes <f_{g_{\gamma^{\phi}\sigma^{\mathbb{I}_{G}}}S}>$ that is, unless isomorphisms $\widehat{{Aut(G)}^{S}}\cong <f_{nS}>\rtimes (\widehat{{Aut(G)}^{S}}/<f_{nS}>)$
\end{lem}
\begin{proof}
Thanks to Proposition 3.1 we already know that $<f_{nS}> \unlhd \widehat{{Aut(G)}^{S}}$. We  will now consider the subgroup $<f_{nS}>\bigcap<f_{g_{\gamma^{\phi}\sigma^{\mathbb{I}_{G}}}S}>$ in our choice the elements $g_{\gamma^{\phi}\sigma^{\mathbb{I}_{G}}}$ either belong to $G-N_{\gamma^{\phi}\sigma^{\mathbb{I}_{G}}}$ or are equal to $\bar{1}$,  so that if $f_{g_{\gamma^{\phi}\sigma^{\mathbb{I}_{G}}}S} \in <f_{nS}>$ then necessarily $f_{g_{\gamma^{\phi}\sigma^{\mathbb{I}_{G}}}S}=f_{S}$. Also showing that $<f_{nS}><f_{g_{\gamma^{\phi}\sigma^{\mathbb{I}_{G}}}S}>=\widehat{{Aut(G)}^{S}}$ is quite straightforward, so is showing that $  <f_{g_{\gamma^{\phi}\sigma^{\mathbb{I}_{G}}}S}> \cong \widehat{{Aut(G)}^{S}}/<f_{nS}>$. 	
\end{proof}

\quad

The lemma just shown is propeudetic to the following result.

\begin{te}
If assumption (I) and assumption (II) are both satisfied then 	$\widehat{{Aut(G)}^{S}}=<f_{nS}>\rtimes (<f_{g_{\sigma^{\mathbb{I}_{G}}}S}>\rtimes <f_{g_{\gamma_{i}^{\phi_{i}}}S}> )$ that is, unless isomorphisms  $\widehat{{Aut(G)}^{S}}\cong\prod_{i=1}^{t} N_{G}(Stab_{G}(x_{0}^{i}))/Stab_{G}(x_{0}^{i})\rtimes (Sim(I)^{\{\mathbb{I}_{G}\}}\rtimes Aut(G)^S/\mathbb{I}_{G}^S)$.

\end{te}

\begin{proof}
Since the elements of $Sim(I)^{\{\mathbb{I}_{G}\}} $ do not depend on the elements of $Aut(G)$; adopting the index notation of $Aut(G)$ again, we can rewrite the elements of $Sim(I)^{Aut(G)^S}$ like this $\gamma_{i}^{\phi_{i}} \sigma^{\mathbb{I}_{G}}$, and consequently also changes the relation in assumption (II) is rewritten as follows	

	$$	\hspace{0.8cm}Sim(I)^{Aut(G)}\stackrel{\mathcal{G}}{\longrightarrow}(G)^{t}$$

$$\hspace{2.5cm} \gamma_{i}^{\phi_{i}}\sigma^{\mathbb{I}_{G}} \mapsto g_{\gamma_{i}^{\phi_{i}}\sigma^{\mathbb{I}_{G}}} .$$
\quad

In the proof of the previous lemma we had noted that $ \widehat{{Aut(G)}^{S}}/<f_{nS}> \cong <f_{g_{\gamma^{\phi}\sigma^{\mathbb{I}_{G}}}S}>  $ due to assumption (II); we specify the isomorphism defining this relationship between the two groups 	

	$$	\hspace{0.8cm}\widehat{{Aut(G)}^{S}}/<f_{nS}> \stackrel{\pi_N}{\longrightarrow}<f_{g_{\gamma^{\phi}\sigma^{\mathbb{I}_{G}}}S}> $$

$$\hspace{1cm}<f_{nS}>f_{g_{\gamma_{i}^{\phi_{i}}\sigma^{\mathbb{I}_{G}}}S}  \mapsto f_{g_{\gamma_{i}^{\phi_{i}}\sigma^{\mathbb{I}_{G}}}S}  .$$	

Since assumption (I) also holds for Proposition 3.9 we also know

$$\widehat{{Aut(G)}^{S}}/<f_{nS}>\cong<<f_{nS}>f_{g_{\sigma^{\mathbb{I}_{G}}}S}>\rtimes <<f_{nS}>f_{g_{\gamma_{i}^{\phi_{i}}}S}>.$$

So $\pi_N(<<f_{nS}>f_{g_{\sigma^{\mathbb{I}_{G}}}S}>)=\{ f_{g_{\sigma^{\mathbb{I}_{G}}}S}\}$ and $\pi_N(<<f_{nS}>f_{g_{\gamma_{i}^{\phi_{i}}}S}>)=\{ f_{g_{\gamma_{i}^{\phi_{i}}}S} \}$; this implies that $\{ f_{g_{\sigma^{\mathbb{I}_{G}}}S}\}$ and $\{ f_{g_{\gamma_{i}^{\phi_{i}}}S} \}$ are soubgroups of $\widehat{{Aut(G)}^{S}}$ and we can therefore rewrite them in this way respectively $<f_{g_{\sigma^{\mathbb{I}_{G}}}S} >$, $<f_{g_{\gamma_{i}^{\phi_{i}}}S} >$; in addition, since $ \pi_N(\widehat{{Aut(G)}^{S}}/<f_{nS}>)\cong \pi_N(<<f_{nS}>f_{g_{\sigma^{\mathbb{I}_{G}}}S}>)\rtimes \pi_N(<<f_{nS}>f_{g_{\gamma_{i}^{\phi_{i}}}S}>)$ we have verified the thesis of the theorem regarding the first isomorphism .
The second isomorphism of the theorem statement is shown by exploiting the properties of the $\pi_N$ map plus the 3.9 proposition.

\end{proof}

Assumptions I and II are not only a sufficient condition for the thesis of theorem (3.12) but are also necessary, so that we have the following theorem.

\begin{te}
	$\widehat{{Aut(G)}^{S}}\cong<f_{nS}>\rtimes (<f_{g_{\sigma^{\mathbb{I}_{G}}}S}>\rtimes <f_{g_{\gamma_{i}^{\phi_{i}}}S}> )  \iff $  assumption (I) and assumption (II) are both satisfied.
\end{te}

\qquad

\subsection{Examples of $\widehat{{Aut(G)}^{S}}$}

\quad

 In this last part of the section we will consider the actions of the dihedral group $D_n$ on the ring $\mathbb{Z}_{2n}$ of integers modulo $2n$ 

$$ D_{n}\times\mathbb{Z}_{2n} \longrightarrow \mathbb{Z}_{2n}\hspace{2cm}  $$

$$(\rho^{k},x)\longrightarrow \rho^{k} \cdot x= x+2k $$

$$\hspace{1.5cm}(\sigma,x)\longrightarrow  \sigma \cdot x= -x$$ 

\quad

and divide them into two cases; n even and odd.

\quad

\title{1° Case : n odd } 

\quad

Recalling that $ D_{n}=<\rho, \sigma| \rho^{n}=1, \sigma^{2}=1, \rho^{k}\sigma=\sigma\rho^{-k}>$ , let us first fix the representatives; since the action of $D_{n}$  on $\mathbb{Z}_{2n}$ generates two orbits, we can have that $x_{0}^{0}=0$ and that $x_{0}^{1}=1$.

Let us now consider the stabilizers of $ D_{n}$, which with quick check we can see that

$$Stab_{D_{n}}(0)=<\sigma>, \hspace{0.25cm} Stab_{D_{n}}(1)=<\rho\sigma>$$.

\quad
\quad

Thus $Aut(D_{n})^{S}$ has $S=(<\sigma>, <\rho\sigma>)$.

\quad
\quad

Now we need to know what the normalizers of such stabilizers look like, and it turns out, by elementary algebraic calculations, since n is odd, that

$$N_{D_{n}}(Stab_{D_{n}}(0))=<\sigma>, \hspace{0.25cm} N_{D_{n}}(Stab_{D_{n}}(1))=<\rho\sigma>.$$

\quad

Thus there are no normalizers of the stabilizers taken into account other than the stabilizers themselves, so assumption (II) for n odd holds trivially.

\quad

Now let us verify assumption I.  First we calculate ${Aut(D_{n})}^{(<\sigma>, <\rho\sigma>)}$, starting with $Aut(D_{n})$.

We know that $Aut(D_{n})=\{ \mathcal{T}, \phi_{\nu}     \}$ such that  $\mathcal{T}^{k}(\rho)=\rho$ and $\mathcal{T}^{k}(\sigma)=\rho^{k}\sigma$ , $\phi_{\nu}(\sigma)=\sigma$, $\phi_{\nu}(\rho)=\rho^{\nu}$ with $k \in \mathbb{Z}_{n}$ and $\nu \in U(\mathbb{Z}_{n})$ that is, $\nu$ is an invertible of $\mathbb{Z}_{n} $ (for this see [1] p. 80).

We also recall the rules of composition of these automorphisms and their orders.

\quad

\begin{equation}
 \mathcal{T}^{k}\phi_{\nu}\mathcal{T}^{l}\phi_{\upsilon}=\mathcal{T}^{k+{\nu}l}\phi_{\nu\upsilon}
\end{equation}

\quad 
 
\begin{equation}
\mathcal{T}^{k}\phi_{\nu}=\phi_{\nu}\mathcal{T}^{k\nu^{-1}}
\end{equation}

\begin{equation}
\mathcal{T}^{n}=\mathbb{I}_{D_{n}},\hspace{0.5cm} {\phi_{\nu}}^{o(\nu)}=\mathbb{I}_{D_{n}},\hspace{0.5cm} with \hspace{0.5cm} o(\nu)|\phi(n).
\end{equation}

In this context, the term '$\phi(n)$' refers to the Euler's totient function applied to the congruence rings modulo n, as detailed in the reference texts listed in the bibliography (from [1] to [6]).
\quad
Now let's see how these automorphisms behave with stabilizers $<\sigma> $, $<\rho\sigma> $, recalling that by $2^{-1}$ we mean, $n$ being odd, $2^{-1}\equiv \frac{n+1}{2}$ mod $n$.

\quad

\begin{equation}
\mathcal{T}^{k}(\sigma)=\rho^{k}\sigma=\rho^{k2^{-1}}\sigma\rho^{-k2^{-1}} \Rightarrow \mathcal{T}^{k}(<\sigma>)=<\sigma>^{\rho^{k2^{-1}}}
\end{equation}

\begin{equation}
\mathcal{T}^{k}(\rho\sigma)=\rho^{k+1}\sigma=\rho^{k}\rho\sigma=\rho^{k2^{-1}}(\rho\sigma)\rho^{-k2^{-1}} \Rightarrow \mathcal{T}^{k}(<\rho\sigma>)=<\rho\sigma>^{\rho^{k2^{-1}}}
\end{equation}

\begin{equation}
 \phi_{\nu}(\sigma)=\sigma \Rightarrow  \phi_{\nu}(<\sigma>)=<\sigma>
\end{equation}

\begin{equation}
  \phi_{\nu}(\rho\sigma)=\rho^{\nu}\sigma=\rho^{\nu-1}\rho\sigma=\rho^{(\nu-1)2^{-1}}\rho\sigma\rho^{-(\nu-1)2^{-1}}  \Rightarrow  \phi_{\nu}(<\rho\sigma>)=<\rho\sigma>^{\rho^{(\nu-1)2^{-1}}}
\end{equation}
\quad

To conclude these calculations let us see how the compositions of $\mathcal{T}^{k} $ and $\phi_{\nu}$ behave on the stabilizers. 

\begin{equation}
 \mathcal{T}^{k}\phi_{\nu}(<\sigma>)=\mathcal{T}^{k}(<\sigma>)=<\sigma>^{\rho^{k2^{-1}}}
\end{equation}

\begin{equation}
\mathcal{T}^{k}\phi_{\nu}(<\rho\sigma>)=\mathcal{T}^{k}\Big(<\rho\sigma>^{\rho^{(\nu-1)2^{-1}}}\Big)=<\rho\sigma>^{\rho^{(\nu-1)2^{-1}+k2^{-1}}} =<\rho\sigma>^{\rho^{(\nu-1+k)2^{-1}} }
\end{equation}

All this leads us to the following conclusion.

\begin{prop}
	
For odd n we have that	
$$Aut(D_{n})^{S}=Aut(D_{n})=Aut(D_{n})^{O}$$ 
with $O=(Orb_{D_{n}}(0),Orb_{D_{n}}(1))$ i.e., all automorphisms of $Aut(D_{n})$ fix orbits. 
\end{prop}

\quad

From Proposition 3.14 we have that all the permutations $\phi_{i}$-induced by the $Aut(D_{n})$ automorphisms $\{ \gamma_{i}^{\phi_{i}}  \}=\{   \mathbb{I}_{I}\}$, i.e  are  coincident with the identical permutation $\mathbb{I}_{I} $, with $I=\{0, 1\}$, so the assumption (I) trivially holds.

\quad

Since $n$ is odd we have obtained that assumptions (I) and (II) are valid, thanks to Theorem 3.14 we deduce that

\begin{equation}
\widehat{{Aut(D_{n})}^{S}}\cong <f_{g_{\sigma^{\mathbb{I}_{G}}}S}>\rtimes <f_{g_{\gamma_{i}^{\phi_{i}}}S}>,\hspace{0.12cm} for\hspace{0.12cm} n \hspace{0.12cm} odd.  	
\end{equation}

We now proceed to determine $g_{\gamma_{i}^{\phi_{i}}}$ and $g_{\sigma^{\mathbb{I}_{D_{n}}}}$ .  The $g_{\gamma_{i}^{\phi_{i}}}$ were implicitly obtained in (3.25) and (3.26). We shall now calculate $g_{\sigma^{\mathbb{I}_{D_{n}}}}$ .

\begin{equation}
 \sigma=\rho^{-1}\rho\sigma=\rho^{-2^{-1}}\rho\sigma\rho^{2^{-1}}=(\rho\sigma)^{\rho^{-2^{-1}}}\hspace{0.5cm} \rho\sigma=\rho^{2\cdot2^{-1}}\sigma=\rho^{2^{-1}}\sigma\rho^{-2^{-1}}=(\sigma)^{\rho^{2^{-1}}}.
\end{equation}

Hence $Sim(I)^{\{\mathbb{I}_{D_{n}}\}}=<\tau=(0,1)>$ i.e., the group generated by the $\tau$ transposition that exchanges the representatives $0,1 \in \mathbb{Z}_{2n}$ associated with the $\mathbb{I}_{\mathbb{Z}_{2n}}$-invariant application that exchanges the orbits.

\quad

For the calculations made from (3.21) to (3.27) in addition to (3.28), the choice of $g_{\gamma_{i}^{\phi_{i}}} $ and $g_{\sigma^{\mathbb{I}_{D_{n}}}} $ will be as follows

\begin{equation}
	g_{\mathcal{T}^{k},\mathbb{I}_{I}}=(\rho^{k2^{-1}},\rho^{k2^{-1}}) \hspace{0.3cm} g_{\mathcal{\phi_{\nu}},\mathbb{I}_{I}}=(1,\rho^{(\nu-1)2^{-1}})\hspace{0.3cm} g_{\tau}=(\rho^{-2^{-1}},\rho^{2^{-1}})
\end{equation}

\begin{equation}
 g_{\mathcal{T}^{k}\phi_{\nu},\mathbb{I}_{I}}=g_{\mathcal{T}^{k},\mathbb{I}_{I}}g_{\mathcal{\phi_{\nu}},\mathbb{I}_{I}}\hspace{0.3cm} g_{\mathcal{T}^{k}\phi_{\nu},\tau}=g_{\mathcal{T}^{k}\phi_{\nu},\mathbb{I}_{I}}g_{\tau}.
\end{equation}

Therefore the decomposition, not total for now, of $\widehat{{Aut(D_{n})}^{S}} $  is

$$ \widehat{{Aut(D_{n})}^{S}}\cong <f_{ g_{\tau}S}>\rtimes <f_{g_{\mathcal{T}^{k}\phi_{\nu},\mathbb{I}_{I}}S}>  .$$ 

\quad

Let us try to go more specific by calculating the analytical expression of the functions $f_{ g_{\tau}S}$, $f_{g_{\mathcal{T}^{k}\phi_{\nu},\mathbb{I}_{I}}S}$. We start with the one nontrivial $\mathbb{I}_{\mathbb{Z}_{2n}}$-invariant $f_{ g_{\tau}S}$.

$$ f_{ g_{\tau}S}(\rho^{x}\cdot 0)=f_{ g_{\tau}S}(2x)=\rho^{x}\rho^{-2^{-1}}\cdot 1=\rho^{x-2^{-1}}\cdot 1=1+2x-1-n=2x+n$$

$$f_{ g_{\tau}S}(\rho^{x}\cdot 1)=f_{ g_{\tau}S}(1+2x)=\rho^{x}\rho^{2^{-1}}\cdot 0 =\rho^{x+2^{-1}}\cdot 0=2x+1+n=(1+2x)+n.$$

\quad

From the calculations just done we derive that $f_{ g_{\tau}S}$ is a translation of a "quantity" $n$ on the $\mathbb{Z}_{2n}$ ring and we can denote it by $T_{n}$ such that

$$ T_{n}:\mathbb{Z}_{2n} \longrightarrow \mathbb{Z}_{2n}$$
$$ \hspace{1cm}         x\rightarrow x+n.$$

\quad

We now derive the  $f_{g_{\mathcal{T}^{k}\phi_{\nu},\mathbb{I}_{I}}S}$ . To do this, simply calculate the $f_{g_{\mathcal{T}^{k},\mathbb{I}_{I}}S} $ and then the $f_{g_{\phi_{\nu},\mathbb{I}_{I}}S} $  and compose.

\quad

We begin by calculating the $f_{g_{\mathcal{T}^{k},\mathbb{I}_{I}}S} $ for $k$ even 

$$ f_{ g_{\mathcal{T}^{k},\mathbb{I}_{I}}S}(\rho^{x}\cdot 0)=f_{g_{\mathcal{T}^{k},\mathbb{I}_{I}} S}(2x)=\rho^{x}\rho^{k2^{-1}}\cdot 0=\rho^{x+k2^{-1}}\cdot 0=2x+k(1+n)=2x+k$$

$$ f_{ g_{\mathcal{T}^{k},\mathbb{I}_{I}}S}(\rho^{x}\cdot 1)=f_{g_{\mathcal{T}^{k},\mathbb{I}_{I}} S}(1+2x)=\rho^{x}\rho^{k2^{-1}}\cdot 1=\rho^{x+k2^{-1}}\cdot 1= 1+2x+k(1+n)=(1+2x)+k.$$

Here again we have a translation of $k$ in $\mathbb{Z}_{2n}$.

$$f_{ g_{\mathcal{T}^{k},\mathbb{I}_{I}}S}:=T_{k} \hspace{2cm} T_{k}:\mathbb{Z}_{2n} \longrightarrow \mathbb{Z}_{2n}$$
$$ \hspace{5cm}         x\rightarrow x+k.$$

Now let's calculate $f_{g_{\mathcal{T}^{k},\mathbb{I}_{I}}S} $  for odd $k$.

$$f_{g_{\mathcal{T}^{k},\mathbb{I}_{I}}S}(\rho^{x}\cdot 0)=2x+k(1+n)=2x+k+n$$

$$f_{g_{\mathcal{T}^{k},\mathbb{I}_{I}}S}(\rho^{x}\cdot 1)=1+2x+k(1+n)=(1+2x)+k+n .$$

We obtained a translation in $\mathbb{Z}_{2n}$ of $k+n$ that is always even.

$$\hspace{2cm} T_{k+n}:\mathbb{Z}_{2n} \longrightarrow \mathbb{Z}_{2n} \hspace{1cm}with\hspace{0.2cm}k \hspace{0.2cm} odd $$
$$ \hspace{1cm}         x\rightarrow x+k+n.$$

\quad

Summing up.

\begin{prop}
	For $n$ odd numbers we have $$ \{ f_{ g_{\mathcal{T}^{k},\mathbb{I}_{I}}S}     \}= \{ T_{2k} | 2k \in  \mathbb{Z}_{2n}  \} .$$
\end{prop}

\quad

Now let us calculate the $f_{g_{\phi_{\nu},\mathbb{I}_{I}}S} $ , remembering that for $n$ odd the invertibles $\nu$ of $\mathbb{Z}_{n}$ can be even or odd.
Let us begin with the $f_{g_{\phi_{\nu},\mathbb{I}_{I}}S} $ for $\nu$ even.

$$ f_{g_{\phi_{\nu},\mathbb{I}_{I}}S}(\rho^{x}\cdot 0)=f_{g_{\phi_{\nu},\mathbb{I}_{I}}S}(2x)=\rho^{{\nu}x}\cdot 0=2{\nu}x=2x(\nu+n)        $$

$$ f_{g_{\phi_{\nu},\mathbb{I}_{I}}S}(\rho^{x}\cdot 1)=f_{g_{\phi_{\nu},\mathbb{I}_{I}}S}(1+2x)=\rho^{{\nu}x+(\nu-1)2^{-1}}\cdot 1=1+2{\nu}x+(\nu-1)(1+n)=$$
$$=1+ 2{\nu}x+\nu-1+n={\nu}(1+2x) +n=(\nu+n)(1+2x)                      . $$

\quad

From the calculations we deduce that $f_{g_{\phi_{\nu},\mathbb{I}_{I}}S}$ is an invertible linear application defined by the factor $\nu+n$, that is invertible in $\mathbb{Z}_{2n}$ and therefore we can define it by $f_{\nu+n}$.

$$\hspace{2cm} f_{\nu+n}:\mathbb{Z}_{2n} \longrightarrow \mathbb{Z}_{2n} \hspace{1cm}with\hspace{0.2cm}\nu \hspace{0.2cm} even $$
$$ \hspace{0.2cm}         x\rightarrow (\nu+n)x.$$

We continue our calculations by explicating $f_{g_{\phi_{\nu},\mathbb{I}_{I}}S}$  with odd $\nu$.

$$ f_{g_{\phi_{\nu},\mathbb{I}_{I}}S}(2x)=2{\nu}x        $$

$$ f_{g_{\phi_{\nu},\mathbb{I}_{I}}S}(1+2x)=1+2{\nu}x +(\nu-1)(1+n)=1+ 2{\nu}x+(\nu-1)={\nu}(1+2x)               . $$

Here again we have a linear application of factor $\nu$, which is also invertible in $\mathbb{Z}_{2n}$, and here again we can redefine $f_{g_{\phi_{\nu},\mathbb{I}_{I}}S} $ with $f_{\nu}$.

$$\hspace{2cm} f_{\nu}:\mathbb{Z}_{2n} \longrightarrow \mathbb{Z}_{2n} \hspace{1cm}with\hspace{0.2cm}\nu \hspace{0.2cm} odd $$
$$ \hspace{0.2cm}         x\rightarrow {\nu}x.$$

\quad

The invertibles $\nu$ and $\upsilon+n$ with $\upsilon$ is an even number invertible in $\mathbb{Z}_{n}$ and $\nu$ invertible in both $\mathbb{Z}_{n}$ and $\mathbb{Z}_{2n}$, represent all invertibles in  $\mathbb{Z}_{2n}$ . 
In fact with little logical and algebraic arguments $\nu$ are all invertibles in $\mathbb{Z}_{2n}$ such that $\nu<n$ in $\mathbb{Z}$;
whereas if we consider invertibles $\mathbb{Z}_{2n}$  in  such that $w>n$ in  $\mathbb{Z}$ can be represented as $w=\upsilon+n$ where $\upsilon$ is an even invertible in $\mathbb{Z}_{n}$.

\quad

From this observation we infer that 

$$ \{  f_{\nu}, f_{\upsilon+n} | \nu \hspace{0.2cm} odd,\upsilon \hspace{0.2cm} even\hspace{0.2cm}with\hspace{0.2cm} \nu,\upsilon \in U(\mathbb{Z}_{n})  \}=  \{ f_{w}| w \in U(\mathbb{Z}_{2n})         \}                . $$

\quad

 From (3.30) and the calculations so far, we deduce that

\begin{prop}
	$$ <f_{g_{\mathcal{T}^{k}\phi_{\nu},\mathbb{I}_{I}}S}>=< T_{2k},f_w | k \in \mathbb{Z}_{2n},\hspace{0.1cm} w \in U(\mathbb{Z}_{2n})>$$
	
	In particular, $T_{2k}f_w$ is an affinity in $\mathbb{Z}_{2n} $ i.e.
	
	$$T_{2k}f_{w}:\mathbb{Z}_{2n} \longrightarrow \mathbb{Z}_{2n}  $$
	$$ \hspace{2cm}         x\rightarrow {w}x+2k .$$
\end{prop}

\quad

 It can be easily seen that the operation rules of $<f_{g_{\mathcal{T}^{k}\phi_{\nu},\mathbb{I}_{I}}S}>$ are similar to the rules of $Aut(D_{n})$ . In fact
 
 \begin{equation}
 	T_{2k}f_{\nu}T_{2k'}f_{\nu'}=T_{2k+2k'\nu}f_{\nu\nu'}.
 \end{equation}

From (3.31) it can be deduced

\begin{equation}
	(T_{2k}f_{\nu})^{-1}=T_{-2\nu^{-1}k}f_{\nu^{-1}} \hspace{1cm}f_{\nu}T_{2k}=T_{2k\nu}f_{\nu}.
\end{equation}

\quad

We now prove the following assertion.

\begin{prop}
For n odd
$$ <f_{g_{\mathcal{T}^{k}\phi_{\nu},\mathbb{I}_{I}}S}> \cong <T_{2}>\rtimes<f_\nu | \nu \in U(\mathbb{Z}_{2n})   >            $$	.

\end{prop}
\begin{proof}
	
First we show that 

\begin{equation}
	<T_{2}> \unlhd < T_{2k}f_\nu >
\end{equation}

\quad

From (3.32) it is verified that

$$ f_{\nu}T_{2k}f_{\nu^{-1}}=T_{2k{\nu}}\Rightarrow T_{2k'}f_{\nu}T_{2k}f_{\nu^{-1}}T_{-2k'}=T_{2k'}T_{2k{\nu}}T_{-2k'}=T_{2k{\nu}}$$

Furthermore, it is easy to verify that $<T_{2k}>=<T_{2}>$	and for which (3.33) has been proved. 

\quad

Thereafter we have that $<T_{2}>\cap < f_\nu >=\{ \mathbb{I}_{\mathbb{Z}_{2n}} \}$, in fact if $T_{2k} \in <T_{2}>\cap < f_\nu >$ this implies that there exists a $\nu \in U(\mathbb{Z}_{2n})$ such that $T_{2k}=f_\nu$ i.e., $x+2k=\nu x$ for every $x \in \mathbb{Z}_{2n}$ specifically for $x=0$ we have that $2k=0$ so $T_{2k}=\mathbb{I}_{\mathbb{Z}_{2n}}$ .

\quad

The fact that $<T_2><f_\nu>=<T_{2k}f_\nu>$ is a very simple verification allows us to conclude the proof.

\end{proof}

\quad

From this proposition we derive the following property

\begin{prop}
	For $n$ odd
	
	$$\widehat{{Aut(D_{n})}^{S}}\cong <T_{n}>\times(<T_{2}>\rtimes<f_\nu | \nu \in U(\mathbb{Z}_{2n})   >  )$$ in particular we have that $<T_{n}>=Z(\widehat{{Aut(D_{n})}^{S}})$.
\end{prop}
\begin{proof}
	To prove the proposition, it is sufficient to show that $<T_n>$ is the center of $\widehat{{Aut(D_{n})}^{S}}$.
	
	$$T_{2k}f_\nu T_n f_{{\nu}^{-1}}T_{-2k}(x)=T_{2k}f_\nu T_n(-2k{\nu}^{-1}+{\nu}^{-1}x)=T_{2k}f_\nu(-2k{\nu}^{-1}+{\nu}^{-1}x+n)=   $$
	
	$$=T_{2k}(-2k+n+x)=2k-2k+n+x=n+x=T_n(x).$$

\end{proof}

\quad

As a consequence of proposition (3.18) we have the following corollary,

\begin{coro}
	$\widehat{{Aut(D_{n})}^{S}}\cong \mathbb{Z}_{2}\times(\mathbb{Z}_{n}\rtimes U(\mathbb{Z}_{2n})).$
\end{coro}
	
\quad

A further property concludes example 1.

\begin{prop}
	For $n$ odd
	
	$$\widehat{{Aut(D_{n})}^{S}}\cong <T_1>\rtimes <f_\nu > \cong \mathbb{Z}_{2n}\rtimes GL(\mathbb{Z}_{2n})=Aff(\mathbb{Z}_{2n}).$$
	
	Moreover, we have $\Big|\widehat{{Aut(D_{n})}^{S}}\Big|=2n\phi(2n)=2n\phi(n).$
\end{prop}
\begin{proof}
	
According to Proposition 3.18 we have that $T_{2k}f_{\nu}T_{n}=T_{2k+n}f_{\nu}$. From this it follows that $T_{1+2k} \in \widehat{{Aut(D_{n})}^{S}}$ since $T_1= T_{2n+1}$,  in particular these translations of an odd quantity exchange orbits.

\end{proof}

\quad

\title{2° Case : n even. } 

\quad

Choosing the same representatives as in the previous example, the action of the group $D_n$ on $\mathbb{Z}_{2n}$ and its stabilizers are the same. So there is nothing to say about it. Instead the normalizers are non-trivial, in fact after simple algebraic calculations we have that

\begin{equation}
	N_{D_n}(<\sigma>)=\{ \mathbb{I}_{D_{n}}, \rho\sigma, \rho^{\frac{n}{2}} ,\rho^{\frac{n}{2}}\sigma           \}\cong V
\end{equation}                            

\begin{equation}
	N_{D_n}(<\rho\sigma>)=\{ \mathbb{I}_{D_{n}}, \rho\sigma, \rho^{\frac{n}{2}} ,\rho^{\frac{n}{2}+1}\sigma           \}.	\cong V
\end{equation}

\quad

Unlike the odd-numbered case, as it is impossible for the stabilizer $<\sigma>$ to be conjugated to the stabilizer $<\rho\sigma>$ consequently $Sim(I)^{\{ \mathbb{I}_{D_n}\}}=\mathbb{I}_{D_n}$. This implies that $\{ \gamma_{i}^{\phi_i}\}= Sim(I)^{Aut(D_n)^S}$ so assumption (I) is trivially verified.

\quad

Now we calculate $Aut(D_n)^S$ and consequently $Sim(I)^{Aut(D_n)^S}  $. It is verified that with the same calculation method done in the previous example

$$\mathcal{T}^{2k}(<\sigma>)=<\sigma>^{\rho^{k}} \hspace{0.5cm}\mathcal{T}^{2k}(<\rho\sigma>)=<\rho\sigma>^{\rho^{k}}$$

$$\mathcal{T}^{1+2k}(<\sigma>)=<\rho\sigma>^{\rho^{k}} \hspace{0.5cm}\mathcal{T}^{1+2k}(<\rho\sigma>)=<\sigma>^{\rho^{1+k}}$$

$$\phi_{\nu}(<\sigma>)=<\sigma> \hspace{0.5cm} \phi_{\nu}(<\rho\sigma>)=<\rho\sigma>^{\rho^{\frac{\nu-1}{2}}  }  . $$

These results lead us to the following conclusions

\begin{prop}
	
	\quad
	
	\begin{enumerate}
		 \item $Aut(D_n)^S=Aut(D_n)  \hspace{0.5cm} KerF=<\mathcal{T}^{2k}, \phi_{\nu}| k \in \mathbb{Z}_n, \nu \in U(\mathbb{Z}_n)>$
	
	\quad	
		
		 \item  $|Aut(D_n)/<\mathcal{T}^{2k}, \phi_{\nu}| k \in \mathbb{Z}_n, \nu \in U(\mathbb{Z}_n)>|=2 \hspace{0.5cm} <\mathcal{T}^{2k}, \phi_{\nu}> \lhd Aut(D_n)$ 
		 
	\quad	 
		 
		 \item $ Aut(D_n)^{O}=KerF \Rightarrow Sim(I)^{Aut(D_n)^{S}}=<\tau=(0,1)>=\{ \gamma_{i}^{\phi_i}\}. $

	\end{enumerate}
\end{prop}

\quad

We now test whether assumption (II) holds true.

\quad

If assumption (II) were true, then by Theorem 3.13 we would have

$$ \widehat{{Aut(D_{n})}^{S}}\cong <f_{nS}>\rtimes <f_{g_{\gamma^{ \mathcal{T}^{k}\phi_{\nu}}}S}>,$$

since $\sigma^{\mathbb{I}_{D_{n}}}=\mathbb{I}_{I}$ and thus $<f_{g_{\sigma^{\mathbb{I}_{D_{n}}}}S}>=\{\mathbb{I}_{\mathbb{Z}_{2n}}\}.$

\quad

We know that all $f_{g_{\gamma^{ \mathcal{T}^{k}\phi_{\nu}}}S}   $ where $\gamma^{ \mathcal{T}^{k}\phi_{\nu}} \in <\tau>$ are generated by the $\mathcal{T}^{k}$-invariants and $\phi_{\nu} $-invariants where $g_{\gamma^{ \mathcal{T}^{k}\phi_{\nu}}} \notin N$.
While for $f_{nS}$ we choose representatives of $N/S$. We have that $N/S= \{ (\rho^{\frac{n}{2}}S_{0},\rho^{\frac{n}{2}}S_{1} ), (\rho^{\frac{n}{2}}S_{0},S_{1} ), (S_{0},\rho^{\frac{n}{2}}S_{1} ) ,(S_{0},S_{1})    \}$, where we recall that by $S_{0}$ and $S_{1}$ we mean $<\sigma>$ and $<\rho\sigma>$, respectively  ; we must therefore deduce that $ g_{\gamma^{ \mathcal{T}^{k}\phi_{\nu}}} \notin \{ (\rho^{\frac{n}{2}},\rho^{\frac{n}{2}} ), (\rho^{\frac{n}{2}}, 1), (1,\rho^{\frac{n}{2}} )     \}$.

\quad

Again with calculations similar to the odd case, it is verified that the $\phi_{\nu}$-invariant functions are invertible linear applications in $\mathbb{Z}_{2n}$ i.e., $f_{ g_{\gamma^{ \phi_{\nu}}}} =f_\nu$ with $\nu \in U(\mathbb{Z}_{n})$ with $f_\nu(x)={\nu}x$.

While $\mathcal{T}^{2k}$-invariants and $\mathcal{T}^{1+2k}$-invariants are respectively the translations in $\mathbb{Z}_{2n}$ of quantities $2k$ and $2k+1$, with $k \in \mathbb{Z}_{n}$ i.e. $f_{g_{\gamma^{ \mathcal{T}^{2k}}}}=T_{2k}$ with $T_{2k}(x)=x+2k$ and $T_{1+2k}(x)=x+1+2k$.
Hence we have that

$$ <f_{g_{\gamma^{ \mathcal{T}^{k}\phi_{\nu}}}S}>=<T_{k},f_{\nu}| k \in \mathbb{Z}_{n}, \frac{k}{2}\neq \frac{n}{2}\hspace{0.1cm}for\hspace{0.1cm} k\hspace{0.1cm} even,\hspace{0.1cm} \frac{k-1}{2}\neq \frac{n}{2} for\hspace{0.1cm} k\hspace{0.1cm} odd,\hspace{0.1cm}\nu \in U(\mathbb{Z}_{n})\hspace{0.1cm}with\hspace{0.1cm} \frac{\nu-1}{2}\neq \frac{n}{2}>  .$$

The last conditions on the second member of equality comes from the assumption that $g_{\gamma^{ \mathcal{T}^{k}\phi_{\nu}}} \notin N$.

 Now I consider the element $f_{{\bigl (\rho^{\frac{n}{2}},\rho^{\frac{n}{2}} \bigr)}S} \in <f_nS>$; it is verified that this is equal to the translation of a quantity $n$ i.e.,  $f_{{\bigl (\rho^{\frac{n}{2}},\rho^{\frac{n}{2}} \bigr)}S}=T_n$ and due to the fact that assumption (II) holds, $T_n \notin <T_k, f_\nu>$; from here we get the absurd since $T_1 \in <T_k, f_\nu>$ we get that $(T_{1})^{n}=T_{n}$.

\quad

We then showed that 

\begin{prop}
For even n we have that in $\widehat{{Aut(D_{n})}^{S}}  $ the assumption (II) is not satisfied.	
	
\end{prop}

Thus the set $<T_{k},f_{\nu}| k \in \mathbb{Z}_{n}, \frac{k}{2}\neq \frac{n}{2}\hspace{0.1cm}for\hspace{0.1cm} k\hspace{0.1cm} even,\hspace{0.1cm} \frac{k-1}{2}\neq \frac{n}{2} for\hspace{0.1cm} k\hspace{0.1cm} odd,\hspace{0.1cm}\nu \in U(\mathbb{Z}_{n})\hspace{0.1cm}with\hspace{0.1cm} \frac{\nu-1}{2}\neq \frac{n}{2}>  $   is not a subgroup of $\widehat{{Aut(D_{n})}^{S}}  $ but representatives of the quotient group $\widehat{{Aut(D_{n})}^{S}}/<f_{nS}>$, so every element of $\widehat{{Aut(D_{n})}^{S}} $ can be written as $f_{g_{\gamma^{ \mathcal{T}^{k}\phi_{\nu}}}S}f_{nS}=f_{g_{\gamma^{ \mathcal{T}^{k}}}S}f_{n_{1}S}f_{g_{\gamma^{\phi_{\nu}}}S}f_{n_{2}S}$.

\quad

From this consideration, let us try to delve into the algebraic structure of  $\widehat{{Aut(D_{n})}^{S}} $  by beginning to describe in detail the elements of $<f_{nS}>$.

\quad

$$ f_{{\bigl ( 1,\rho^{\frac{n}{2}} \bigr)}S}:\mathbb{Z}_{2n} \longrightarrow \mathbb{Z}_{2n} $$
$$ \hspace{1.5cm}         2x\rightarrow 2x$$
$$ \hspace{3cm}        1+2x\rightarrow 1+2x+n$$

\quad

$$ f_{{\bigl (\rho^{\frac{n}{2}}, 1 \bigr)}S}:\mathbb{Z}_{2n} \longrightarrow \mathbb{Z}_{2n} $$
$$ \hspace{2cm}         2x\rightarrow 2x+n$$
$$ \hspace{3cm}        1+2x\rightarrow 1+2x.$$

\quad

Observing the fact that all invertibles in $\mathbb{Z}_{2n}$, similar to the $n$ odd case, can be rewritten as $\nu$ if $\nu$ is invertible in $\mathbb{Z}_{n}$ and $w=\nu+n$ if $w>n$ in $\mathbb{Z}$; then we have that $T_n f_{\nu}=f_{{\bigl (\rho^{\frac{n}{2}}, 1 \bigr)}S}f_{\nu+n}$. That is, all maps  $ f_{\nu+n}$ are invertible linear applications in $\mathbb{Z}_{2n}$ .

Further from this fact we can rewrite the elements of $f_{nS}$ in such a way: $ f_{{\bigl ( 1,\rho^{\frac{n}{2}} \bigr)}S}=f_{1+n}$, $ f_{{\bigl (\rho^{\frac{n}{2}}, 1 \bigr)}S}=T_{n}f_{1+n}$ and we'd already seen that $f_{{\bigl (\rho^{\frac{n}{2}},\rho^{\frac{n}{2}} \bigr)}S}=T_{n} $. Therefore, according to Proposition 3.2 

\begin{equation}
	<f_{nS}>=<T_{n}, f_{1+n}> \cong V.
\end{equation}

\quad

We also have that $< T_{n}, T_{k} | k \in \mathbb{Z}_{n}>=<T_{k} | k \in \mathbb{Z}_{2n}>$ and that $<T_{n}f_{\nu}, f_{\nu}, f_{1+n} | \nu \in U(\mathbb{Z}_{n})>=<f_{\nu}| \nu \in U(\mathbb{Z}_{2n})>$.

Moreover, we already know that the subgroup of translations in $\mathbb{Z}_{2n}$,  $<T_n>$ is normal in $\widehat{{Aut(D_{n})}^{S}} $. As we have seen in the $n$ odd case, it is verified that $<T_n>\cap<f_\nu>=\{\mathbb{I}_{\mathbb{Z}_{2n}}\}$ and $\widehat{{Aut(D_{n})}^{S}}=<T_n><f_\nu>$.

Then we conclude that

\begin{prop}
	For $n$ even
	
	$$\widehat{{Aut(D_{n})}^{S}}\cong <T_1>\rtimes <f_\nu > \cong \mathbb{Z}_{2n}\rtimes GL(\mathbb{Z}_{2n})=Aff(\mathbb{Z}_{2n}).$$
	
	Moreover, we have $\Big|\widehat{{Aut(D_{n})}^{S}}\Big|=2n\phi(2n).$
\end{prop}

\quad

So in general the following are valid

\begin{coro}

	$$\widehat{{Aut(D_{n})}^{S}}\cong <T_1>\rtimes <f_\nu > \cong \mathbb{Z}_{2n}\rtimes GL(\mathbb{Z}_{2n})=Aff(\mathbb{Z}_{2n})$$
	
	Moreover, we have $\Big|\widehat{{Aut(D_{n})}^{S}}\Big|=2n\phi(2n).$
\end{coro}   

\section{Algebraic sieves}

Consider a finite ordered list of finite groups $G:=(G_{i})_{i=0}^{l}$ and consider the actions of $G_i$ on a finite nonempty set $X$.

$$ G_i \times X \longrightarrow X$$

$$ (g_{i},x) \rightarrow g_{i}\cdot x .$$

\quad

Now we fix the representatives of the orbits induced by $G_i$ on $X$ i.e $x_{i}^{j} \in X$ with $j \in L_{i} :=\{0,..., j,...,l_{i} -1\}$ with $l_{i}=\{$Number of orbits of $X$ induced by $G_{i}\}$ as $i$ varies in $L:=\{0,...,i,...,l\}$.

\quad

In this context we can give the following definitions

\begin{definition}
	A $multiple$ action on X, with $X \neq \emptyset$, is the data of a finite list of actions of groups $G=(G_{i})_{i=0}^{l}$, which we can define $(G\times X, \rho):=(G_{i} \times X, \rho_{i})_{i=0}^{l}$.
\end{definition}

\title{Notations:}

We can denote any ordered $l$-uple of representatives, previously fixed, $(x_{i}^{j_{i}})_{i=0}^l$ by $j_i \in L_i$ in such a way $x^J$ with $J=(j_{i})_{i=0}^l$ and $J \in \prod_{i=0}^l L_{i}$.

\quad

\begin{definition}
	Given a multiple action of a list of $G=(G_{i})_{i=0}^{l}$ groups on $X$ and given an $l$-uple of $x^J$ representatives of $G_i$ induced orbits for some $J \in \prod_{i=0}^l L_{i} $.
	A $G$-induced algebraic sieve of $x^J$ basis points consists of the partition of $X$ into the following sets $\cup_{i=0}^l  Orb_{G_{i}}(x_{i}^{j_{i}})$ , $\overline{{(\cup_{i=0}^l  Orb_{G_{i}}(x_{i}^{j_{i}}))}}$ where the latter is the complementary with respect to the set $X$. That is, in the formulas $X=\cup_{i=0}^l Orb_{G_{i}}(x_{i}^{j_{i}})\bigsqcup \overline{\cup_{i=0}^l  Orb_{G_{i}}(x_{i}^{j_{i}})}$. Furthermore, we can denote such a partition as $\mathcal{C}(G; x^J)=\bigl\{ \cup_{i=0}^l  Orb_{G_{i}}(x_{i}^{j_{i}}), \overline{\cup_{i=0}^l  Orb_{G_{i}}(x_{i}^{j_{i}})}\bigr\} $ . 
	
\end{definition}

\quad

\begin{definition}
	Given an algebraic sieve $\mathcal{C}(G; x^J) $, if $G=(G_{i})_{i=0}^{l}$ the set $\cup_{i=0}^l Orb_{G_{i}}(x_{i}^{j_{i}})$  is called a $G$-covering $\mathcal{A}(G; x^J)$. Such a set is said to be a complete $G$-covering if $\mathcal{A}(G; x^J)=X$, that is, a covering of $X$.

\end{definition}
   
   	So we can write an algebraic sieve in such a way        
   	
   	\begin{equation}
   		\mathcal{C}(G; x^J)=\bigl\{ \mathcal{A}(G; x^J) , \overline{\mathcal{A}(G; x^J)}         \bigr\}
   	\end{equation}  

\quad
   
   \subsection{Examples of algebraic sieves}

   \quad

 \quad
 
Before describing such an example we need the following definitions

\begin{definition}
	An algebraic sieve $\mathcal{C}(G; x^J) $ is said to be dihedral if $G=(D_{n_{i}})_{i=0}^{l}$ with dihedral group $D_{n_{i}}$.
\end{definition} 

\begin{definition}
	An algebraic sieve $\mathcal{C}(G; x^J) $ is said to be  a Goldbach's sieve if the complement of its $G$-covering $\overline{\mathcal{A}(G; x^J)}$  is in biunivocal correspondence with  all pairs of primes satisfying Goldbach's conjecture for some even number $N$ and if necessary with the pair 1 and N-1 if N-1 is a prime number. 
\end{definition}

\quad

Using the definitions just explained , we will now show an example of an algebraic dihedral sieve, which would be Goldbach's if the conjecture were satisfied for a certain number $2n$.

\quad

Let us consider, in $\mathbb{Z}_{2n}$, the set that has as class representatives all the prime numbers in $\mathbb{Z}$ that are smaller than  $\sqrt{2n}$ thus indexed .

\begin{equation}
	2=p_{0}<p_{1}<p_{2}<...<p_{j}<...<p_{t} \hspace{0.5cm} 2<q_{t+1}<...<q_{k}<...<q_{l}\le \bigl[\sqrt{2n}\bigr] \hspace{0.2cm} with \hspace{0.2cm} p_{j}|2n\hspace{0.2cm} q_{k} \nmid 2n 
\end{equation}
Then we consider the following list of groups $D_{p,q}^{2n}:=\{ D_{\frac{2n}{p_{j}}},D_{\bigl[\frac{2n}{q_{k}}\bigr]-1}| 0\le j \le t$, $  t+1\le k \le l \}$ accompanied by the following list of group actions on $\mathbb{Z}_{2n}$

\begin{equation}
	D_{\frac{2n}{p_{j}}}\times \mathbb{Z}_{2n} \longrightarrow \mathbb{Z}_{2n} 
\end{equation}    

$$\hspace{2.5cm}(\rho^{k},x)\rightarrow \rho^{k}\cdot x=x+kp_{j}  $$

$$\hspace{2cm}(\sigma, x)\rightarrow \sigma\cdot x=-x $$

$$\hspace{3cm}(\rho^{k}\sigma,x)\rightarrow \rho^{k}\sigma \cdot x=-x+kp_{j} $$

\quad

with $0 \le j \le t$.

\quad

While in the case $q_{k} \nmid 2n$ we need a more laborious description. We consider the following subsets of $\mathbb{Z}_{2n}$ \hspace{0.2cm}   $\mathcal{C}_{k}=\{ (2+m)q_k \in \mathbb{Z}_{2n} | 0 \le m \le \bigl[\frac{2n}{q_{k}}\bigr]-2 \hspace{0.2cm } in \hspace{0.2cm }  \mathbb{Z} \}$, $\mathcal{C}_{-k}=\{ -(2+m)q_k \in \mathbb{Z}_{2n} | 0 \le m \le \bigl[\frac{2n}{q_{k}}\bigr]-2 \hspace{0.2cm } in \hspace{0.2cm }  \mathbb{Z} \}$. The sets $\mathcal{C}_{k}$ and $\mathcal{C}_{-k}$, as they are constructed, are disjointed; by joining them we can consider them as the set 

$\mathcal{Q}_{k}=\mathcal{C}_{k} \bigsqcup \mathcal{C}_{-k}       =\{ \pm(2+m)q_k \in \mathbb{Z}_{2n} | 0 \le m \le \bigl[\frac{2n}{q_{k}}\bigr]-2 \hspace{0.2cm } in \hspace{0.2cm }  \mathbb{Z} \}$.

We immediately see that $|\mathcal{Q}_{k}|=\Big| D_{\bigl[\frac{2n}{q_{k}}\bigr]-1}\Big|$ so surely there is a biunivocal correspondence $f:D_{\bigl[\frac{2n}{q_{k}}\bigr]-1}\longrightarrow \mathcal{Q}_{k}$ in particular we choose this

\begin{equation}
	f(\rho^{m})=(2+m)q_{k} \hspace{0.2cm} in \hspace{0.1cm} particular \hspace{0.1cm} f(\rho^{0}=1)=2q_{k}
	\quad
	f(\sigma \rho^{m})=-(2+m)q_{k} \hspace{0.2cm} in \hspace{0.1cm} particular \hspace{0.1cm} f(\sigma \rho^{0}=\sigma)=-2q_{k}.
\end{equation}

The action of group $D_{\bigl[\frac{2n}{q_{k}}\bigr]-1}$ on $\mathbb{Z}_{2n}$ will be induced by the $f$ in (4.4) in following way

\begin{equation}
	D_{\bigl[\frac{2n}{q_{k}}\bigr]-1}\times \mathbb{Z}_{2n} \longrightarrow \mathbb{Z}_{2n}  
\end{equation}

$$\hspace{7cm}(\sigma^{h}\rho^{k},x)\rightarrow \sigma^{h}\rho^{k}\cdot x=\begin{cases}
	f(\sigma^{h}\rho^{k}f^{-1}(x)) \hspace{0.1cm } if \hspace{0.1cm } x \in \mathcal{Q}_{k}\hspace{0.5cm} h \in \{  0,1\} \\ (-1)^{h}x\hspace{0.1cm } otherwise.
\end{cases}  $$

\quad

The next step is to fix the representatives for each orbit that generates the list group $D_{p,q}^{2n}$.

We begin with groups of the type 	$D_{\frac{2n}{p_{j}}}$ with $ 0\le j \le t$ the choice of representatives is as follows

\begin{equation}
	x_{j}^{n_{j}}=n_{j} \hspace{0.2cm}with \hspace{0.2cm}0 \le n_{j} \le p_{j}-1 \hspace{0.2cm}n_{j}\in L_{j} .
\end{equation}

While the representatives associated with the groups $D_{\bigl[\frac{2n}{q_{k}}\bigr]-1}$   with $ t+1\le k \le l$ are  ,by indexing the following orbits in this way,

\begin{equation}
	Orb_{D_{\bigl[\frac{2n}{q_{k}}\bigr]-1}}(y_{k}^{n_{k}})=\{y_{k}^{n_{k}},-y_{k}^{n_{k}} \} \hspace{0.2cm}with \hspace{0.2cm} y_{k}^{n_{k}} \notin \mathcal{Q}_{k} \hspace{0.2cm} 1\le n_{k} \le n-\Big(\Bigl[\frac{2n}{q_{k}}\Bigr]-1\Big) \hspace{0.2cm} and\hspace{0.2cm}  n_{k}\in L_{k}
\end{equation} 

the following

\begin{equation}
	x_{k}^{0}=2q_{k} \hspace{0.2cm}	x_{k}^{n_{k}}=z_{k}^{n_{k}} \hspace{0.2cm} with \hspace{0.2cm} z_{k}^{n_{k}} \in  \{y_{k}^{n_{k}},-y_{k}^{n_{k}} \} \hspace{0.2cm} fixed.
\end{equation}

\quad

And now we can set the basic points

$$ x^{J}=(x_{0}^{0},x_{1}^{0},...,x_{j}^{0},...,x_{t}^{0},x_{t+1}^{0},...,x_{k}^{0},...,x_{l}^{0})$$

$$x^{J}=(\underbrace{0, 0,..., 0,..., 0}_{(t+1)-times}, 2q_{t+1},..., 2q_{k},..., 2q_{l})$$

\begin{equation}
	x^J=((0)_{i=0}^{t},(2q_{k})_{k=t+1}^{l})=:x_{0,q} \hspace{0.5cm} J=\{ 0\}_{i=0}^{t}=\overline{0} \in \prod_{i=0}^l L_{i}.
\end{equation}

\quad

We then have the necessary elements to construct the following algebraic dihedral sieve with the actions defined in (4.3) and (4.5); these actions generate the following $D_{p,q}^{2n}$-covering

\begin{equation}
	\mathcal{A}(D_{p,q}^{2n}; x_{0,q})=(\cup_{i=0}^t Orb_{D_{\frac{2n}{p_{j}}}}(0))\bigcup (\cup_{i=t+1}^l Orb_{D_{\bigl[\frac{2n}{q_{k}}\bigr]-1}}(2q_{k})).  
\end{equation}  

\quad

And we get the following dihedral sieve, which is Goldbach's if the conjecture holds for $2n$

\begin{equation}
	\mathcal{C}(D_{p,q}^{2n}; x_{0,q} )=\bigl\{ \mathcal{A}(D_{p,q}^{2n}; x_{0,q}), \overline{\mathcal{A}(D_{p,q}^{2n}; x_{0,q})}         \bigr\}.
\end{equation}

 \quad
 
 \section{Multi-invariant functions}
 
 \quad
 \begin{definition}
 	Consider a list of groups $G=(G_{i})_{i=0}^{l}$ and an associated list of actions on a set $X$,  $(G\times X, \rho)=(G_{i}\times X, \rho_{i})_{i=0}^{l}$.
 	We also consider a list of automorphisms $\phi:=(\phi^{(i)})_{i=0}^{l}$ with $\phi^{(i)}\in Aut(G_{i})$.
 	Then a function  $f  \in Sym(X)$ is said to be $\phi$-invariant or multi-invariant with respect to $\phi$ if
 	
 	\begin{equation}
 		f( \rho_{i}(g_{i},x))=\rho_{i}(\phi^{(i)}(g_{i}),f(x)) \hspace{0.2cm}\forall i=0,1,...,l.
 	\end{equation}
 	
 \end{definition}

\quad

If it is not necessary to specify the name of an action $\rho_i$ with each element $x \in X$. We will use the usual right-hand action symbol ' $\cdot$ '(as, moreover, we have used so far). So (5.1) becomes thus

\begin{equation}
	f( g_{i}\cdot x)=\phi^{(i)}(g_{i})\cdot f(x) \hspace{0.2cm}\forall i=0,1,...,l.
\end{equation}

\begin{os}
	If f is $\phi$-invariant for some $\phi \in \prod_{i=0}^l Aut(G_i)^{S_i}$,  recall that by $S_i$ we want to mean the list do stabilizers induced by the action of $G_i$ i.e. $S_i=(Stab_{G_i}(x_{i}^{j_i}))_{{j_i}=0}^{l_{i}-1}$ with $l_i$ is the number of orbits of the action of $G_i$ , then $f \in \bigcap_{i=0}^{l} \widehat{Aut(G_i)^{S_i}}$. Such a set makes sense indeed is a subgroup of $Sym(X)$, since every group of $Aut(G_i)^{S_i}$ , as it is defined is a subgroup of $Sym(X)$.
\end{os}
          
\quad

Now we see how an $\phi$-invariant $f$ acts on the algebraic sieves $ \mathcal{C}(G; x^J)=\bigl\{ \mathcal{A}(G; x^J) , \overline{\mathcal{A}(G; x^J)}         \bigr\}$;  we know that $ \mathcal{A}(G; x^J) = \cup_{i=0}^l  Orb_{G_{i}}(x_{i}^{j_{i}})$ with  $J \in \prod_{i=0}^l L_{i} $  by definition of $f$ (see (1.2)) we have that $f(Orb_{G_{i}}(x_{i}^{j_{i}}))=Orb_{G_{i}}(g_{\gamma^{\phi^{(i)}}}\cdot x_{i}^{\gamma^{\phi^{(i)}}(j_{i})})=Orb_{G_{i}}(x_{i}^{\gamma^{\phi^{(i)}}(j_{i})})$ as $i$ varies in $L$ and $j_i$ varies in $L_i$. Therefore 

\begin{equation}
	f(\mathcal{A}(G; x^J))=f(\cup_{i=0}^l  Orb_{G_{i}}(x_{i}^{j_{i}}))=\cup_{i=0}^l  Orb_{G_{i}}(x_{i}^{\gamma^{\phi^{(i)}}(j_{i})}).
\end{equation}
If we denote by $\gamma^\phi$ the list of permutations induced by the list of $\phi=(\phi^{(i)})_{i=0}^{l}$ automorphisms i.e. $\gamma^{\phi}=(\gamma^{\phi^{(i)}})_{i=0}^{l} \in \prod_{i=0}^l Sym(I_i)^{Aut(G_i)^{S_i}} $ where $I_i$ are the indices of the $x_{i}^{j_i}$ representatives and we can write

\begin{equation}
	 x^{\gamma^{\phi}(J)}=(x_{i}^{\gamma^{\phi^{(i)}}(j_i)} )_{i=0}^{l} \hspace{0.2cm} with \hspace{0.2cm} \gamma^{\phi}(J)=(\gamma^{\phi^{(i)}}(j_i) )_{i=0}^{l},  
\end{equation}
  then (5.3) is rewritten more compactly as follows
  
\begin{equation}
	f(\mathcal{A}(G; x^J))=f(\cup_{i=0}^l  Orb_{G_{i}}(x_{i}^{j_{i}}))=\cup_{i=0}^l  Orb_{G_{i}}(x_{i}^{\gamma^{\phi^{(i)}}(j_{i})})=\mathcal{A}(G; x^{\gamma^{\phi}(J)}).
\end{equation}  
And then we can define the following action of the multi-invariant function $f$ by $\phi$ on the algebraic sieves

\begin{equation}
	f(\mathcal{C}(G; x^J)):=\bigl\{ f(\mathcal{A}(G; x^J)) , f(\overline{\mathcal{A}(G; x^J)})   \bigr\}=\bigl\{  \mathcal{A}(G; x^{\gamma^{\phi}(J)}), \overline{\mathcal{A}(G; x^{\gamma^{\phi}(J)})}  \bigr\}=\mathcal{C}(G;  x^{\gamma^{\phi}(J)} ).
\end{equation}

\section{Some symmetry groups of an algebraic sieve}

\begin{definition}
	A symmetry group on an algebraic sieve $\mathcal{C}(G; x^J)$ of a given set $X$ is a group of transformations $T=\{ \mathcal{T}  \}$, acting as group actions on all bipartitions  $X$ (i.e. $\{A, B \}$ with $A \subseteq X$, $ B\subseteq X$ such that at least one of the two subsets is non-empty and that $A \cap B =\emptyset$ and $A\cup B=X$) that fixes the sieve i.e. $\mathcal{T}(C(G;x^J))=C(G;x^J)$.
\end{definition}

\quad

\title{ 1.Group symmetry through multi-invariant applications.
}

\quad

I consider a $f$ $\phi$-invariant with $f  \in \bigcap_{i=0}^{l} \widehat{Aut(G_i)^{S_i}}$ such that $f(\mathcal{C}(G;x^J))= \mathcal{C}(G;x^J)$ in which case the set $\{ f \in \bigcap_{i=0}^{l} \widehat{Aut(G_i)^{S_i}} | f(\mathcal{C}(G;x^J))= \mathcal{C}(G;x^J)  \}$ is a subgroup of $\bigcap_{i=0}^{l} \widehat{Aut(G_i)^{S_i}}$.

In fact from (5.6) we know that if $f(\mathcal{A}(G;x^J))=\mathcal{A}(G;x^J)$ then we have $f(\overline{\mathcal{A}(G; x^J)} )=\overline{f(\mathcal{A}(G; x^J))}=\overline{\mathcal{A}(G; x^J)} $. Also, since $f$ is invertible, if $f(\mathcal{A}(G;x^J))=\mathcal{A}(G;x^J)$ then $f^{-1}(\mathcal{A}(G;x^J))=\mathcal{A}(G;x^J)$ and for similar reasoning we can show that $f^{-1}(\overline{\mathcal{A}(G; x^J)})=\overline{\mathcal{A}(G; x^J)}$. Furthermore, if there exists another multi-invariant $g$ such that $g(\mathcal{A}(G;x^J))=\mathcal{A}(G;x^J)$ then $gf(\mathcal{A}(G;x^J))=g(\mathcal{A}(G;x^J))=\mathcal{A}(G;x^J)$.
\quad
Then

 \begin{equation}
 	{\left(\bigcap_{i=0}^{l} \widehat{Aut(G_i)^{S_i}}\right)}^{\mathcal{C}(G;x^J)}:=\{ f \in \bigcap_{i=0}^{l} \widehat{Aut(G_i)^{S_i}} | f(\mathcal{C}(G;x^J))= \mathcal{C}(G;x^J)  \}
 \end{equation}
 
  is a symmetry group of the algebraic sieve that takes into account the multiple action of the list of groups $G=(G_{i})_{i=0}^{l}$ on the set $X$ of the algebraic sieve $\mathcal{C}(G;x^J)$.
  
  \quad
  However, there may be problems in obtaining information about  $\mathcal{C}(G;x^J)$ from such a symmetry group. Just from the fact that ${\left(\bigcap_{i=0}^{l} \widehat{Aut(G_i)^{S_i}}\right)}^{\mathcal{C}(G;x^J)}\leq \bigcap_{i=0}^{l} \widehat{Aut(G_i)^{S_i}}$. For $l$ very large it might happen that $\bigcap_{i=0}^{l} \widehat{Aut(G_i)^{S_i}}=\mathbb{I}_{X}$ and in this case it is difficult to get information from the algebraic sieve  $\mathcal{C}(G;x^J)$ and then classify it; for example if  $\bigcap_{i=0}^{l} \widehat{Aut(G_i)^{S_i}}=\mathbb{I}_{X}$  it would be difficult to know whether $\mathcal{A}(G;x^J)$ is $G$-covering complete or not. In this case we would have to choose for choosing other symmetry groups appropriate for our research.
  
  \quad
  
  \title{2.Symmetry groups of an algebraic sieve with respect to $\phi$-invariant functions with $\phi \in Aut(G_i)^{S_i}$}.
  
  \quad
  
  I consider the following subgroup
  \begin{equation}
   \widehat{Aut(G_i)^{S_i}}^{\mathcal{C}(G;x^J)}:= \{ f \in \widehat{Aut(G_i)^{S_i}} | f(\mathcal{C}(G;x^J))= \mathcal{C}(G;x^J)  \}
  \end{equation} 
  for a $G_i$;  such a group, as it is defined, does not totally take into account all the actions in the list of groups of $G$, but also prefers the action of $G_i$. In fact, it is easily seen that  $\widehat{Aut(G_i)^{S_i}}^{\mathcal{C}(G;x^J)}\geq\left(\bigcap_{i=0}^{l} \widehat{Aut(G_i)^{S_i}}\right)^{\mathcal{C}(G;x^J)} $. The choice of such a symmetry group may have an advantage in knowing whether $\mathcal{A}(G;x^J)$ can be a complete $G$-covering or not. In fact if $\widehat{Aut(G_i)^{S_i}}^{\mathcal{C}(G;x^J)}< \widehat{Aut(G_i)^{S_i}}$ surely  $\mathcal{A}(G;x^J)$ is incomplete.

\section{The group $\widehat{Aut(D_n)}^{\mathcal{C}(D_{p,q}^{2n}; x_{0,q} )}$  }

\quad

In this section we will treat only the dihedral sieve $\mathcal{C}(D_{p,q}^{2n}; x_{0,q} ) $; so to lighten the notation,  from now on until the end of the article, we will denote by $\mathcal{G}_{2n}$ the group $\widehat{Aut(D_n)}^{\mathcal{C}(D_{p,q}^{2n}; x_{0,q} )}  $ i.e. $\mathcal{G}_{2n}=\widehat{Aut(D_n)}^{\mathcal{C}(D_{p,q}^{2n}; x_{0,q} )}$. And by $\mathcal{A}_{2n}$ and $\overline{\mathcal{A}_{2n}}$ we will denote $\mathcal{A}(D_{p,q}^{2n}; x_{0,q}) $ and $\overline{\mathcal{A}(D_{p,q}^{2n}; x_{0,q})}  $, respectively, and by $ \mathcal{C}_{2n}$ the sieve $\mathcal{C}_{2n}=	\mathcal{C}(D_{p,q}^{2n}; x_{0,q} )$. 
We now begin to study the first properties of $\mathcal{G}_{2n}$.

\begin{lem}
	If $T_{d} \notin \mathcal{G}_{2n}$ and $f_{\nu} \notin \mathcal{G}_{2n}$ but $T_{d} f_{\nu} \in \mathcal{G}_{2n}$ then $T_{2d} \in \mathcal{G}_{2n}$, $f_{\nu^2} \in \mathcal{G}_{2n}$
\end{lem}
\begin{proof}
	Let us denote by $\overline{\mathcal{A}_{2n}}=\{ r_{1},..., r_{i},...,r_{m}\}$ . Such a set is nonempty by hypothesis of the lemma. Moreover, it can be easily verified that if $r_{i} \in \overline{\mathcal{A}_{2n}}$ then $-r_{i} \in \overline{\mathcal{A}_{2n}}$ in $\mathbb{Z}_{2n}$.
	Since $T_{d} \notin \mathcal{G}_{2n}$ , there exists at least one $r_{i}$   such that $T_{d}(r_{i})=w_{i}$ with $w_{i} \in \mathcal{A}_{2n}$. It is shown that $w_{i} \in U(\mathbb{Z}_{2n})$. In fact, if this were not so i.e. $T_{d}(r_{i})=mu$ with $(m,2n)>1$ and $(u,2n)=1$ then one would have that $T_{d}(-mu)=-r_{i}$, in particular there would exist $r_{j} \in \overline{\mathcal{A}_{2n}}$ such that $f_{\nu}(r_{j})=-mu$ since $T_{d}f_{\nu} \in \mathcal{G}_{2n}$ but this is absurd if $\frac{n}{2}$ is not prime, since $(\nu r_{j}, 2n)=1$.  If, on the other hand, $\frac{n}{2}=p$ were prime one would have as the only plausible possibility that $u=1$ and $r_{j}=p$ but since $f_{\nu}(p)=p$ then $f_{\nu}(r_{j})=p \in\overline{\mathcal{A}_{2n}} $ therefore $w_{i}=p$ but this is absurd since $w_{i} \notin \overline{\mathcal{A}_{2n}}$.
	\quad
	Then $w_{i}=T_{d}(r_{i}) \in U(\mathbb{Z}_{2n})-\overline{\mathcal{A}_{2n}}$  and furthermore $T_{d}(T_{d}(r_{i}))=T_{2d}(r_{i})=T_{d}(w_{i})$. We know that $T_{d}(r_{i})=w_{i}$ then implies that $T_{d}(-w_{i})=-r_{i}$, so by hypothesis there exists $r_{j} \in \overline{\mathcal{A}_{2n}}$ such that $f_{\nu}(r_{j})=-w_{i} \Rightarrow \nu r_{j}=-w_{i} \Rightarrow w_{i}=-\nu r_{j}=\nu(- r_{j}) \Rightarrow f_{\nu}(-r_{j})=w_{i} $.
	Hence $T_{2d}(r_{i})=T_{d}f_{\nu}(-r_{j})=r_{k}$ for some  $r_{k}\in \overline{\mathcal{A}_{2n}}$.
	So since $r_{i}$ is generic we get that $T_{2d} \in  \mathcal{G}_{2n}$, and since $(T_{d}f_{\nu})^{2} \in \mathcal{G}_{2n}$, we have that $(T_{d}f_{\nu})^{2}=T_{d}f_{\nu} T_{d}f_{\nu}=T_{d+d\nu} f_{\nu^{2}} \Rightarrow T_{d(1+\nu)}f_{\nu^{2}}=T_{2d(\frac{1+\nu}{2})}f_{\nu^{2}} =(T_{2d})^{\frac{1+\nu}{2}} f_{\nu^{2}}$ then $(T_{2d})^{\frac{1+\nu}{2}} \in \mathcal{G}_{2n} $ this implies that $f_{\nu^{2}} \in \mathcal{G}_{2n} $ from which the thesis.
\end{proof}

\quad

\begin{os}
	Under the assumptions of Lemma 7.1 if we add that $\mathcal{G}_{2n}^{(1)}:=<T>\bigcap \mathcal{G}_{2n}=<T_m> \neq <\mathbb{I}_{\mathbb{Z}_{2n}}>$ with $m \neq 1$( this implies that $\mathcal{C}$ is Goldbach's so $m$ is even). So if there exists $T_{d}$ and $f_{\nu}$ that verify the assumptions of lemma 7.1. In particular this implies that $T_{2d} \in <T_m>$ and that $2d= \alpha m$ in $\mathbb{Z}$ with $\alpha$ an appropriate odd number so $d = \alpha \frac{m}{2}$. Hence we have that $T_{{\alpha} \frac{m}{2}} f_{\nu}\in \mathcal{G}_{2n}$. But this is true for every $\alpha$; because, given an odd $\beta$ different from $\alpha \in \mathbb{Z}$, then $T_{\alpha \frac{m}{2}} f_{\nu} =T_{(\alpha-\beta)\frac{m}{2}} T_{\beta\frac{m}{2}} f_{\nu} \in \mathcal{G}_{2n}$, and since this is $\alpha-\beta$ even we have that $T_{(\alpha-\beta)\frac{m}{2}} \in  \mathcal{G}_{2n}$, and hence we have that $T_{\beta\frac{m}{2}} f_{\nu} \in \mathcal{G}_{2n} $ .
     If, on the other hand,$\mathcal{G}_{2n}^{(1)}=<\mathbb{I}_{\mathbb{Z}_{2n}}>$ again under the assumptions of lemma 7.1 that there exists a $T_{d}$ satisfying this assumption we have that $T_{2d} \in <\mathbb{I}_{\mathbb{Z}_{2n}}>$ so  $T_{2d}=\mathbb{I}_{\mathbb{Z}_{2n}} \Rightarrow d=n$ so $T_{n} f_{\nu} \in \mathcal{G}_{2n}$.
	
\end{os}

\quad

From this observation, we now illustrate an additional lemma functional to demonstrate a fundamental property of the structure of the group $\mathcal{G}_{2n}$ in terms of factorization into products and/or semiproducts of groups.

\quad

\begin{lem}
		If the assumptions of Lemma 7.1 hold, then  $T_{\frac{m}{2}} f_{\nu} \in \mathcal{G}_{2n} $ (with $m$ referring to $T_{m}$, generator of $\mathcal{G}_{2n}^{(1)}$) and is unique unless multiplication ( left or right) by an appropriate element $T_{\beta m}f_{u}$ with $f_{u} \in \mathcal{G}_{2n}.$
\end{lem}

\begin{proof}
	The fact that $T_{\frac{m}{2} }f_{\nu} \in \mathcal{G}_{2n}$ is deduced from observation 7.2. We move from uniqueness unless products with other elements of $\mathcal{G}_{2n}$; if there existed another $T_{\frac{m}{2} }f_{w} \in \mathcal{G}_{2n}$  with $f_{w} \neq f_{\nu}$ and $f_{w}$ satisfies the hypothesis of Lemma 7.1, then we would have that $T_{\frac{m}{2} }f_{w}, T_{\frac{m}{2} }f_{\nu} \in \mathcal{G}_{2n}$ in particular we would have that $T_{\frac{m}{2} }f_{w}(T_{\frac{m}{2} }f_{\nu})^{-1} \in \mathcal{G}_{2n}\Rightarrow T_{\frac{m}{2} } f_{w}(T_{-\nu^{-1}\frac{m}{2}} f_{{\nu}^{-1}}) \in \mathcal{G}_{2n}\Rightarrow   T_{\frac{m}{2}-w{\nu}^{-1}\frac{m}{2}}  f_{{w}{\nu}^{-1}}=T_{\frac{m}{2}(1-w{\nu}^{-1})} f_{{w}{\nu}^{-1}}$. Given that $T_{\frac{m}{2}(1-w{\nu}^{-1})} \in \mathcal{G}_{2n}^{(1)}$  then $f_{{w}{\nu}^{-1}} \in \mathcal{G}_{2n}$. Then there exists $u \in U(\mathbb{Z}_{2n})$ such that $f_{u} \in \mathcal{G}_{2n}$ and $f_{u}=f_{{w}{\nu}^{-1}}$ . This implies that $f_{w}=f_{u}f_{\nu}\Rightarrow T_{\frac{m}{2} }f_{w}(T_{\frac{m}{2} }f_{\nu})^{-1}=T_{\frac{m}{2}(1-w{\nu}^{-1})}f_{u}$. Thus $T_{\frac{m}{2} }f_{w}=T_{\frac{m}{2}(1-w{\nu}^{-1})}f_{u}T_{\frac{m}{2} }f_{\nu}$. Thus Lemma 7.3 is verified for $\beta=\frac{1-w{\nu}^{-1}}{2}$ and furthermore $ T_{\frac{m}{2} }f_{w}=T_{\frac{m}{2} }f_{u}f_{\nu}=T_{\frac{m}{2} }f_{\nu}(f_{u}).$
\end{proof}
\quad

Now let us begin to see how these lemmas just proved can be exploited on the following proposition

\begin{prop}
	Let $N=2^{k}\overline{ N }$ with $k>1$and $\overline{ N }$ odd and $T_{2^{k-1}\overline{ N }}$ and $ f_{1+2^{k-1}\overline{ N } } \notin  \mathcal{ G }_ {N}$ then $\mathcal{ G }_ {N}=<T_{2^{k-1}\overline{ N }}f_{1+2^{k-1}\overline{ N }}>\times( <T_{ 2m }>\rtimes H ) $ with $H < U(\mathbb {Z }_{ N })$ and $<T_{ 2m } > < <T_{ 2 }>.$
\end{prop}
\begin{proof}
First, it is easy to see that $T_{2^{k-1}\overline{ N }}f_{1+2^{k-1}\overline{ N }} \in \mathcal{ G }_ {N}$; in fact given an $r \in \overline{\mathcal{A}_{N}}$ we have that $T_{2^{k-1}\overline{ N }}f_{1+2^{k-1}\overline{ N }}(r)=T_{2^{k- 1}\overline{ N }}(r(1+2^{k-1}\overline{ N }) ) =T_{2^{k-1}\overline{ N }}(r+2^{k-1}\overline{ N })=r+2^{k-1}\overline{ N }+2^{k-1}\overline{ N }=r$. We are in the condition of Lemma 7.1 since ${(T_{2^{k-1}\overline{ N }})}^{2}=T_{2^{k}\overline{ N }}= \mathbb{I}_{\mathbb{Z}_{N}}$ and ${(f_{1+2^{k-1}\overline{ N }})}^{2}=f_{1+{(2^{k-1}\overline{ N })}^{2}}=f_{1+2^{k-1}\overline{ N }2^{k-1}\overline{ N }}=f_{ 1 }=\mathbb{I}_{\mathbb{Z}_{N}}$, due to the fact that $k>1$ .

Then by Lemma 7.2 and observation 7.3, it is verified that $T_{2^{k-1}\overline{ N }}f_{1+2^{k-1}\overline{ N }}$ is unique unless left and right products for $T_{2\alpha m }f_{ u }$ with $f_{ u } \in H$ .
We note that $T_{2^{k-1}\overline{ N }}f_{1+2^{k-1}\overline{ N }} T_{2\alpha m }f_{ u }=T_{2^{k-1}\overline{ N }+2\alpha m}f_{u+2^{k-1}\overline{ N }}$ and that $T_{2\alpha m }f_{ u }T_{2^{k- 1}\overline{ N }}f_{1+2^{k-1}\overline{ N }}=T_{2\alpha m+u2^{k-1}\overline{ N }}f_{u+u2^{k- 1}\overline{ N }}=T_{2^{k-1}\overline{ N }+2\alpha m}f_{u+2^{k-1}\overline{ N }}$ so that between $< T_{2^{k-1}\overline{ N }}f_{1+2^{k-1}\overline{ N }}>$ and $<T_{ 2m }>\rtimes H $ there is a direct product.

\quad

Moreover, $H< U(\mathbb{Z}_{N})$ and $<T_{2m}> < <T_{2}>$ if it were not so we would have that at least one of the elements $T_{2^{k-1}\overline{ N }}$ and $f_{1+2^{k-1}\overline{ N }}$ belong to $\mathcal{ G }_ {N}$ against the assumptions of proposition 7.4.

\end{proof}

\begin{os}
	We can be seen from the proof of Proposition 7.4 that $<T_{2^{k- 1}\overline{ N }}f_{1+2^{k-1}\overline{ N }}>=Z( \widehat{Aut(D_n)} ) $ that is, the center of $ \widehat{Aut(D_n)} $ so we have that $\mathcal{G}_{N} =Z( \widehat{Aut(D_n)} ) \times (<T_{ 2m }>\rtimes H).$
\end{os}		
\quad

We now turn to the remaining cases to analyze how the group $\mathcal{G}_{N}$ factorizes. We anticipate that in these cases the proof is more elaborate than in proposition 7.4.

\quad

\begin{prop}
	If either $T_{2^{k- 1}\overline{ N }}$, $f_{1+2^{k-1}\overline{ N }} \in \mathcal{G}_{N}$ and $N=2^{k}\overline{ N }$ with $k>1$and $\overline{ N }$ odd or $k=1$ ,that is, $N=2\overline{ N }$, then there is no element $T_{m}f_{\nu} \in \widehat{Aut(D_n)}$ such that $T_{m} \notin \mathcal{G}_{N}$ and $f_{\nu} \notin \mathcal{G}_{N}$ and $T_{m}f_{\nu} \in \mathcal{G}_{N}.$
\end{prop}
			
\begin{proof}
	Let us start with the case $N=2\overline{ N }$, if there existed the element $T_{m}f_{\nu} \in \mathcal{G}_{N}$ such that $T_{m} \notin \mathcal{G}_{N}$ and $f_{\nu} \notin \mathcal{G}_{N}$ based on observation 7.2 then we can assume that $\mathcal{G}_{N}^{(1)}=<T_ {2m}> \geq < \mathbb{I}_{\mathbb{Z}_{N}}>$, in which case we would have that $m$ is odd, since $2m | 2\overline{ N }$ and $\overline{ N }$ is odd. Since $Orb_{ D_{ \overline{ N }}}(0) \in \mathcal{A}$, then we have that $T_{m}f_{\nu}( Orb_{ D_{ \overline{ N }}}(0) ) =T _{m}(Orb_{ D_{ \overline{ N }}}(0) )= Orb_{ D_{ \overline{ N }}}(1) $ due to the fact that $m$ is odd. From that fact we obtain the absurdity in that $Orb_{ D_{ \overline{ N }}}(1) \notin \mathcal{A} $ and instead $T_{m}f_{\nu} \in \mathcal{G}_{N}$. 
	
	\quad
	
	We now turn to the case $N=2^{k}\overline{ N }$ with $k>1$ which we will prove as the case before by absurdity. So we put ourselves under the assumptions that $T_{m} \notin \mathcal{G}_{N}$ and $f_{\nu} \notin \mathcal{G}_{N}$ and $T_{m}f_{\nu} \in \mathcal{G}_{N}$. 
	
	\quad
	By virtue of observation 7.2 and Lemma 7.3 we can state that all elements of the type $T_{n}f_{w} \in \mathcal{G}_{N}$ such that $T_{n} \notin {N}$ and $f_{w} \notin \mathcal{G}_{N}$ can be seen as $T_{m(2j+1)}f_{\nu u}$, assuming that $T_{2m}$ is a generator of ${\mathcal{G}_{N}}^{(1)}$ with $j \in \{ 0,...,i,...\frac{N}{2m}-1 \}$ and $f_{u} \in H=\mathcal{G}_{N}\cap U(\mathbb{Z}_{N})$.
	From this it can be deduced that $|\{T_{m(2j+1)}f_{\nu u}\} _{j \in \{ 0,...,i,...\frac{N}{2m}-1 \} , f_{u} \in H}|=|<T_{2m}>\rtimes H|$ and also it is not difficult to see that $\mathcal{G}_{N}=\{T_{m(2j+1)}f_{\nu u} \}\sqcup <T_{2m}>\rtimes H$. Hence $|\mathcal{G}_{N} |=| \{T_{m(2j+1)}f_{\nu u}\} |+|<T_{2m}>\rtimes H |= 2| <T_{2m}>\rtimes H |$ in particular $<T_{2m}>\rtimes H \unlhd \mathcal{G}_{N} $ since $\bigl[ \mathcal{G}_{N} :  <T_{2m}>\rtimes H \bigr]=2$.
	\quad
	
	I consider the subgroup $<T_{2m}, T_{m}f_{\nu}, f_{{\nu}^{2}} >$ , such a group can be seen as $$\{ T_{m \alpha}f_{{\nu}^{j}} \hspace{0.1cm}with\hspace{0.1cm} \alpha\hspace{0.1cm} and\hspace{0.1cm}  j\hspace{0.1cm} both\hspace{0.1cm} even\hspace{0.1cm} or\hspace{0.1cm} both\hspace{0.1cm} odd \}.$$
			In fact $<T_{2m}, T_{m}f_{\nu}, f_{{\nu}^{2}} >\subseteq \{ T_{m \alpha}f_{{\nu}^{j}} \}$
		if this were not so, one would have that $T_{m(2j+1)}f_{{\nu}^{2n}}$ or $T_{m(2j)}f_{{\nu}^{2n+1}} \in <T_{2m}, T_{m}f_{\nu}, f_{{\nu}^{2}} > $; in either case we would have absurdity since we would have that $T_{m(2j+1)}$ or $f_{{\nu}^{2n+1}} \in \mathcal{G}_{N}$ against the assumed hypothesis.
			We now show that $\{ T_{m \alpha}f_{{\nu}^{j}}  \} \subseteq <T_{2m}, T_{m}f_{\nu}, f_{{\nu}^{2}} >$.
				If $ T_{m\alpha}f_{{\nu}^{j}}$ with $\alpha$ and $j$ both odd then $$ T_{m\alpha}f_{{\nu}^{j}}=T_{m(\alpha -{ \sum^{j-1}_{i=0}}{\nu}^{i}+{ \sum^{j-1}_{i=0}}{\nu}^{i})}f_{{\nu}^{j}}=$$
				
				 $$=T_{m(\alpha-\sum^{j-1}_{i=0}{{\nu}^{i}})} T_{m(\sum^{j-1}_{i=0}{{\nu}^{i}})}f_{{\nu}^{j}}=T_{m(\alpha-\sum^{j-1}_{i=0}{{\nu}^{i}})}{(T_{m}f_{\nu})}^{j}=T_{2m\frac{(\alpha-\sum^{j-1}_{i=0}{{\nu}^{i}})}{2}}{(T_{m}f_{\nu})}^{j} .$$ 
				 
				 The latter equality leads us to say, due to the fact that $\alpha$ is odd $j$ is odd and $\nu$ is odd , that $T_{2m\frac{(\alpha-\sum^{j-1}_{i=0}{{\nu}^{i}})}{2}}{(T_{m}f_{\nu})}^{j} \in <T_{2m}, T_{m}f_{\nu}, f_{{\nu}^{2}} >$ and thus $ T_{m\alpha}f_{{\nu}^{j}} \in <T_{2m}, T_{m}f_{\nu}, f_{{\nu}^{2}} >$ . The case $ T_{m\alpha}f_{{\nu}^{j}} $ with $\alpha$ and $j$ even is proved analogously to the odd case since as before $ T_{m\alpha}f_{{\nu}^{j}} =T_{2m\frac{(\alpha- \sum^{j-1}_{i=0}{{\nu}^{i}})}{2}}{(T_{m}f_{\nu})}^{j} $ so that the inclusion $\{ T_{m \alpha}f_{{\nu}^{j}}  \} \subseteq <T_{2m}, T_{m}f_{\nu}, f_{{\nu}^{2}} >$ is proved.
				 Thus $ |<T_{2m}, T_{m}f_{\nu}, f_{{\nu}^{2}} >|=|\{ T_{m \alpha}f_{{\nu}^{j}} \} |$ and furthermore we can see that $\{ T_{m \alpha}f_{{\nu}^{j}} \} $ can be distributed in such a way $$ \{ T_{m \alpha}f_{{\nu}^{j}} \}= \bigsqcup^{\frac{o(\nu)}{2}-1} _{k=0}{\{ T_{m \beta} f_{\nu^{P_{2k}(\beta)}}   \}_{\beta \in \{ 0,...,i,...,\frac{N}{m}-1 \} }} \hspace{0.2cm} with \hspace{0.2cm} P_{2k}(\beta)=\begin{cases} 2k  \hspace{0.2cm} for\hspace{0.2cm} \beta \hspace{0.2cm} even \\
				 2k+1 \hspace{0.2cm} otherwise
				 	
				 \end{cases}  .$$
				 
				 \quad
				 
				 Hence
				 
				 $$ |\{ T_{m \alpha}f_{{\nu}^{j}} \}|= \sum_{k=0}^{\frac{o(\nu)}{2}-1}\Big| {\{ T_{m \beta}    \}_{\beta \in \{ 0,...,i,...,\frac{N}{m}-1 \} }}   \Big|=\sum_{k=0}^{\frac{o(\nu)}{2}-1}\frac{N}{m}=\frac{o(\nu)}{2} \frac{N}{m}=\frac{o(\nu)}{2} \frac{2N}{2m}=o(\nu)\frac{N}{2m}.$$
				 
				 And so
				 
				 $$ |<T_{2m}, T_{m}f_{\nu}, f_{{\nu}^{2}} >|=o(\nu)\frac{N}{2m}.$$
				 
				 Given that $\mathcal{G}_{N}=<T_{2m}, T_{m}f_{\nu}, f_{{\nu}^{2}} >\sqcup (H - <f_{{\nu}^{2}}>)$, and furthermore, we observe that $o({\nu}^{2})=\frac{o(\nu)}{2}$ (so $o(\nu)$ must be even otherwise we would have that $(f_{\nu})^{-1} \in \mathcal{G}_{N}$ against the assumptions); therefore, we have that 
			 
		 $$ | \mathcal{G}_{N}|=|<T_{2m}, T_{m}f_{\nu}, f_{{\nu}^{2}} >|+|H|-| <f_{{\nu}^{2}} >|=o(\nu)\frac{N}{2m}+|H|-\frac{o(\nu)}{2}.$$

But we know that $  |\mathcal{G}_{N}|=2| <T_{2m}>\rtimes H |=2| <T_{2m}> |\cdot |H|=2\frac{N}{2m}\cdot |H|$ .

$$o(\nu)\frac{N}{2m}+|H|-\frac{o(\nu)}{2}= 2\frac{N}{2m}\cdot |H| \Rightarrow \frac{o(\nu)}{2}( 2\frac{N}{2m}- 1)=|H|( 2\frac{N}{2m}-1 ) \Rightarrow o({\nu}^{2}) ( 2\frac{N}{2m}-1)=|H|( 2\frac{N}{2m}-1 ). $$

Given that $ f_{{\nu}^{2}} \in \mathcal{G}_{N} \Rightarrow o({{\nu}^{2}})  | |H|$ then $|H|=a o({{\nu}^{2}})$ for some $a \geq 1$.
	
	Then $o({{\nu}^{2}})( 2\frac{N}{2m}-1)=a o({{\nu}^{2}})( 2\frac{N}{2m}-1 ) \Rightarrow ( 2\frac{N}{2m}-1)=a( 2\frac{N}{2m}-1 ) $. This implies that $a$ must necessarily be $a=1$, that is, $H=<f_{{\nu}^{2}}>$, in particular $H$ is cyclic, but this is absurd since by hypothesis $T_{2^{k- 1}\overline{ N }}$, $f_{1+2^{k-1}\overline{ N }} \in \mathcal{G}_{N}$ and consequently, the group H contains two elements of order 2. Hence the thesis.

\end{proof}	
	
	\section{Methods of calculation for the group $\mathcal{G}_{N}$ .  }					
	
\quad

In this section we will discuss methods of calculating the group $\mathcal{G}_{N}$ that do not have general validity for all even numbers $N$ . We will begin with relatively simple special cases such as $N=2p$ with $p$ prime number.

\quad
Given an $N=2p$ for Proposition 7.6 we have that $\mathcal{G}_{N}=\mathcal{G}^{(1)}_{N}\rtimes H$ it is verified that $\mathcal{G}^{(1)}_{2p}=< \mathbb{I}_{\mathbb{Z}_{N}}>$.

In fact since $\overline{\mathcal{A}_{2p}}\neq \emptyset$ by hypothesis, then $\mathcal{G}^{(1)}_{2p} \leq <T_{2}>$ and furthermore $|<T_{2}>|=p$, so $\mathcal{G}^{(1)}_{2p}=\begin{cases}< \mathbb{I}_{\mathbb{Z}_{N}}>  \\ <T_{2}> otherwise \end{cases}$. If it were $\mathcal{G}^{(1)}_{2p}=<T_{2}>$ then $T_{2}(p) \in \overline{\mathcal{A}_{2p}}$ but $2p=0$ and that is absurd.

\quad

Hence $  \mathcal{G}_{2p} \leq U(\mathbb{Z}_{2p})$ unless isomorphisms, moreover the action on $\mathbb{Z}_{2p}$ of $\mathcal{G}_{2p}$ is of the type $f_{\nu}(x)=\nu x$. It is easy to check that the action of $ \mathcal{G}_{2p}$ on the partition of $U(\mathbb{Z}_{2p})$ given by the sets $\mathcal{A}_{2p} \cap U(\mathbb{Z}_{2p})$ and $\overline{\mathcal{A}_{2p}}-\{ p \}$ is free. It can be deduced from what has been said so far that $|\mathcal{G}_{2p}|  |  |\mathcal{A}_{2p} \cap U(\mathbb{Z}_{2p})|$ and $|\mathcal{G}_{2p}| |  |\overline{\mathcal{A}_{2p}}-\{ p \}|$ further we have that $|\mathcal{G}_{2p}| | ( |\mathcal{A}_{2p} \cap U(\mathbb{Z}_{2p})|, |\overline{\mathcal{A}_{2p}}-\{ p \}| )$.

\quad

Since $ |U({Z}_{2p})|=|\mathcal{A}_{2p} \cap U(\mathbb{Z}_{2p})|+|\overline{\mathcal{A}_{2p}}-\{ p \}|$, then we have that $|\mathcal{A}_{2p} \cap U(\mathbb{Z}_{2p})|$ and $|\overline{\mathcal{A}_{2p}}-\{ p \}|$ are even.

We can reach the conclusion of these arguments by introducing a criterion that can calculate the possible cardinality of the group $\mathcal{G}_{2p}$.

\quad

\begin{prop}
	Given $N=2p$ with $p$ prime then there exist two even integers $\alpha$ and $\beta$ such that $p-1=\alpha + \beta$ and $|\mathcal{G}_{2p}| | (\alpha, \beta).$
\end{prop}

\begin{proof}
For what we have seen so far, it suffices to pose $\alpha=|\mathcal{A}_{2p} \cap U(\mathbb{Z}_{2p})|$ and $\beta=|\overline{\mathcal{A}_{2p}}-\{ p \}| $

\end{proof}
		
\begin{os}
	This criterion just described in Proposition 8.1 in some cases may not be very helpful. For example in the case $N=26$ we have that $|\mathcal{G}_{26}|=\{ 2, 4,6 \}$; in such a case then it is convenient to explicitly compute the two sets $\mathcal{A}_{26} \cap U(\mathbb{Z}_{26}) $, and $ \overline{\mathcal{A}_{2p}}-\{ p \}$ , which in the case $N=26$ can of course can be done but when the even number $N$ becomes very large it gets complicated. This is due to the fact that for $p=13$ and that $p-1=12$ in addition to the factor $2$ contains a composite number $6$. On the other hand, when a prime number $p$ is of the type $p=2q+1$ with $q$ prime then the only common factor between $|\mathcal{A}_{26} \cap U(\mathbb{Z}_{2p})|$ and $|\overline{\mathcal{A}_{2p}}-\{ p \}|$ is $2$. So make criterion 8.1 more binding by adding an extra assumption about $N$.
\end{os}	

\quad

	\begin{prop}
		If there exists an $N=2p$, with $p$ prime, such $p=2q+1$ with $ q$ prime, then $\mathcal{G}_{2p} \cong \mathbb{Z}_{2}$.
	\end{prop}

\quad 

\begin{example}
	Thanks to criterion 8.3 we can easily calculate the symmetry group $ \mathcal{G}_{2p}$ of the algebraic dihedral goldbach sieve $\mathcal{C}_{2p}$ for example for $p=5, 7, 11, 23, 47$ which will be isomorphic to $ \mathbb{Z}_{2}.$
\end{example}

\quad

Now we will calculate symmetry groups of the type $ \mathcal{G}_{N}$ of particular even numbers $N$. We begin by introducing them through the following definition.

\quad

\begin{definition}
	An even number $N$ can be called $ \mathcal{G}_{N}$-cyclotomic if it is of the type $N=2^{k}\prod_{i=1}^{s} p^{h_{i}}_{i}$ where $p_{i}$ are all consecutive odd primes less than $[ \sqrt{N}]$ with $h_{i}>0, k \geq 0$ and there is no prime $q$ such that $q \nmid N$ and $q \le [\sqrt{N}].$
\end{definition}

We will now prove the following proposition describing the properties of the number just defined.

\begin{prop}

	 	Let $N$ be an even number $ \mathcal{G}_{N}$-cyclotomic then
	 	
 \begin{enumerate}  
	\item	The dihedral sieve $\mathcal{C}_{N}$ is Goldbach's ( i.e. $\overline{\mathcal{A}_{N}} \neq \emptyset$) 
	
	\item   $ \mathcal{G}_{N} \cong <T_{2m}> \rtimes U(\mathbb{Z}_{N}) $  for a suitable integer $m$ such that if $m>0$ then $m |N$
	
\end{enumerate}

\end{prop}

\begin{proof}

If $\overline{\mathcal{A}_{N}}=\emptyset$ we will have that $-r$ is a multiple number in $ \mathbb{Z}$ for every prime number $r$ in particular by hypothesis $-r \in \mathcal{A}_{N}$ so $(N-r, N) >1$ and $(r, N)>1$ for which there are no primes $N >q> [\sqrt{N}]$ such that $ q \nmid N$. From this we deduce that $U(\mathbb{Z}_{N})= \{ 1 \}$; in fact, if there were a $ \nu \in U(\mathbb{Z}_{N})$ with $ \nu \neq 1$. Then $\nu$ would necessarily be a composite number i.e. $\nu$ is a multiple of some prime number $q$ such that $q \nmid N$ but we have just seen that if $\overline{\mathcal{A}_{N}}=\emptyset$ then such primes do not exist. From there the absurdity, for given that $ U(\mathbb{Z}_{N})=\{ 1 \}$, we have that $\phi(N)=1$ but $\phi(N)=\phi(2^{k}\prod_{i=1}^{s} p^{h_{i}}_{i})=2^{k-1}\phi(p^{h_{i}}_{i})$. And so for some $i$ we get that $\phi(N) \ge \phi(p^{h_{i}}_{i})=p^{h_{i}-1}_{i} (p_{i}-1) \ge 2$ that is $\phi(N) >1$ and (1) is proved.

\quad
So we are  in the situation that $\overline{\mathcal{A}_{N}} \neq \emptyset$ but this implies that $\overline{\mathcal{A}_{N}}=U(\mathbb{Z}_{N})$. For if $q \neq 1$ and is a prime number such that $ q \nmid N$, then obviously it cannot belong to $\mathcal{A}_{N}$ . Conversely, if $\nu \in U(\mathbb{Z}_{N})$ and $ \nu \neq 1$ cannot be composite. In fact, if $\nu$ were so then $\nu \in \mathcal{A}_{N}$, but by hypothesis all $ x \in \mathcal{A}_{N}$ have factors in common with $(x,N)>1$, which is impossible.
\quad
So that $U(\mathbb{Z}_{N})\le \mathcal{G}_{N}$ then by Proposition (7.6) $\mathcal{G}_{N}\cong <T_{2m}>\rtimes U(\mathbb{Z}_{N})$, which by (1) we have that $<T_{2m}> < <T_{2}>$ and thus we have also shown (2).

\end{proof}

\begin{os}
Given a number $N$ that is $\mathcal{G}_{N}$-cyclotomic then it can easily be seen that $N=q+1$ where $q$ is a prime. This condition is only necessary in general. In fact it is obvious that in general the inverse is not true.	
\end{os}
\quad
Now let's look at some examples of these numbers $\mathcal{G}_{N}-cyclotomic$.

\begin{example}
	The number $N=12$ is $\mathcal{G}_{12}-cyclotomic$, the verification in this case is quite elementary. In particular, we see how the group $\mathcal{G}_{12}$ is; it is verified that $U(\mathbb{Z}_{12})=\{ 1, 5, 7, 11 \}$ while it is verified that $\mathcal{G}^{(1)}_{12}=<T_{6}>$. Then by (2) of Proposition 8.5 we have that $\mathcal{G}_{12}=<T_{6}>\rtimes U(\mathbb{Z}_{12}) \cong \mathbb{Z}_{2}\rtimes V \cong ( \mathbb{Z}_{2})^{3}$.
	
	\quad
	For $N=18$ we have a $\mathcal{G}_{18}-cyclotomic$ number with $\overline{\mathcal{A}_{18}}=U (\mathbb{Z}_{18})=\{ 1, 5, 7, 11, 13, 17 \}$. We always have by applying Proposition 8.5 that $\mathcal{G}_{18}=<T_{6}>\rtimes U (\mathbb{Z}_{18})\cong \mathbb{Z}_{3}\rtimes \mathbb{Z}_{6}$.
	
	\quad
	
	For $N=24$ it is obtained that $U(\mathbb{Z}_{24})=\{ 1,5, 7, 11, 13, 17, 19,23 \} \cong (\mathbb{Z}_{2})^{3}$ and thus $\mathcal{G}_{24}=<T_{6}>\rtimes U(\mathbb{Z}_{24})\cong \mathbb{Z}_{4}\rtimes (\mathbb{Z}_{2})^{3}$.
	
	\quad
	
	For $N=30$ we get that $U(\mathbb{Z}_{30})=\{ 1, 7, 11, 13, 17, 19,23, 29 \}$ and thus $\mathcal{G}_{30}=<\mathbb{I}_{\mathbb{Z}_{30}}>\rtimes U(\mathbb{Z}_{30})\cong \mathbb{Z}_{2}\times \mathbb{Z}_{4}$.
	
\end{example}

\quad

The (1) and (2) of Proposition 8.5 is a necessary condition for an even number $N$ to be $\mathcal{G}_{N}-cyclotomic$, in truth the condition is sufficient and we will see this by proving the following proposition.

\begin{prop}
Given an even number $N > 6$ such that the dihedral sieve $\mathcal{C}_{N}$ is Goldbach's and that $\mathcal{G}_{N}\cap U(\mathbb{Z}_{N})=U(\mathbb{Z}_{N})$ unless isomorphisms, then $N$ is $\mathcal{G}_{N}-cyclotomic$.
\end{prop}

\begin{proof}
	First, we have that $1 \in \overline{\mathcal{A}_{N}}$. If this were not so, we would have that $f_{\nu}(1)\notin \overline{\mathcal{A}_{N}} \hspace{0.2cm} \forall \nu \in U(\mathbb{Z}_{N})$ by hypothesis. Hence  we are in the case that $\overline{\mathcal{A}_{N}}=\{ p\}$ and thus $N=2p$ with $p\ge 3$. Then all prime $q$ ,as $\mathcal{A}_{N}$ is defined, such that $q \nmid N$ and $q \le [\sqrt{N}]$ belong to $\mathcal{A}_{N}$, and there are none larger than $[\sqrt{N}]$. And from here we get the absurdity since by Bertrand's postulate(see cap 22 in [8]) applied to $p$ there exists at least one prime $q$ such that $2p-2>q>p>\sqrt{2p}$ the last inequality is justified by the fact that $p>3$. This contradicts the fact that we had deduced that there were no primes $q$ that did not divide $N$ greater than $[\sqrt{N}]$. Then we have that $f_{\nu}(1) \in \overline{\mathcal{A}_{N}}$ i.e., $\nu \in \overline{\mathcal{A}_{N}}$ and since $\nu$ is generic we get that $\overline{\mathcal{A}_{N}}=U(\mathbb{Z}_{N})$ if $N>6$. From here we deduce that $\mathcal{A}_{N}=\bigcup_{j}Orb_{D_{\frac{N}{p_{j}}}}(0)$ i.e. $N$ is $\mathcal{G}_{N}$-cyclotomic.
\end{proof}

\quad

We now see a criterion for determining the possible generators of $\mathcal{G}^{(1)}_{N}$. This criterion is based on the fact that the action of any element $T_{d}$ is free. In fact it can be easily seen that $$T_{d}(x)=x \Rightarrow x+d=x \Rightarrow d=0,$$ from which it follows that $|<T_{d}> |  | |\mathcal{A}_{N}| $ in particular if $d$ is even it preserves the parity of $x$ and thus we have that $|<T_{d}>| || \mathcal{A}_{N}\cap Orb_{D_{\frac{N}{2}}}(1)|$ and we also do not forget that $|<T_{d}>|  | |\overline{ \mathcal{A}_{N} }|$. Hence $|<T_{d}> | | (|\mathcal{A}_{N}\cap Orb_{D_{\frac{N}{2}}}(1) |, |\overline{ \mathcal{A}_{N} } |)$.
				
				From here one can give a kind of estimate on $\mathcal{G}^{(1)}_{N}$ on a group of a given generator $<T_{2m}>$ in the following  way
				
\begin{equation}
\mathcal{G}^{(1)}_{N} \le < T_{\frac{N}{(|\mathcal{A}_{N}\cap Orb_{D_{\frac{N}{2}}}(1) |, |\overline{ \mathcal{A}_{N} } |)}} 	>.
\end{equation}

\quad

Thus, we have the elements to prove the following criterion

\begin{prop}
	There exist two integers $\alpha$ and $\beta$ such that $\alpha + \beta=N$ with $N$ even so that $|\mathcal{G}^{(1)}_{N}| \mid (\alpha, \beta)$. Moreover, we will have in that case that $\mathcal{G}^{(1)}_{N} \le <T_\frac{N}{(\alpha, \beta)}>$
\end{prop}
\begin{proof}
	It is enough to put $\alpha=|\mathcal{A}_{N}\cap Orb_{D_{\frac{N}{2}}}(1) |$ and $\beta=|\overline{ \mathcal{A}_{N} } |$.
\end{proof}	
\quad

It is appropriate to make a few remarks about this criterion

\begin{os}
As with Criterion  8.2, there are disadvantages and advantages depending on the nature of the even number $N$.

Certainly, if the number of solutions of the pair of integers $\alpha, \beta$ of (8.8) is small compared to the generators of $<T_{2}>$, and generators of the whole group $<T_{2}>$ are excluded,  we have a relatively good advantage. For example, if the solutions include instead include those for which $(\alpha, \beta)= \frac{N}{2}$, we get exactly the element $T_{2}$ as a possible generator of $\mathcal{G}^{(1)}_{N}$ ; which tells us nothing about such a group unless we compute directly $|\mathcal{A}_{N}\cap Orb_{D_{\frac{N}{2}}}(1) |$ and $|\overline{ \mathcal{A}_{N} }|$ which, can as you can imagine, can get very complicated. But, thanks to this criterion it is easy to prove that $\mathcal{G}^{(1)}_{2p} =<\mathbb{I}_{\mathbb{Z}_{2p} }>$ with $p>3$ prime number.	 In fact, if $\alpha+\beta=p$ you will  necessarily have that $(\alpha,\beta)=1$ if $\alpha \neq 0$ and $\beta \neq 0$ and also the cases $\alpha=0$ or $\beta=0$ are excluded due to the fact that $N=2p$ and with $p>3$.
\end{os}	

\quad

As in Criterion 8.3 we determine which possible groups can estimate the group $ \mathcal{G}_{N}\cap U(\mathbb{Z}_{N})$,  but unlike the group $\mathcal{G}^{(1)}_{N}$, which is cyclic, the group $ \mathcal{G}_{N}\cap U(\mathbb{Z}_{N})$ is a subgroup of $U(\mathbb{Z}_{N})$ ( again, unless isomorphisms), which in turn is not always a cyclic group, rather it is a direct product of appropriate cyclic groups. So beyond estimating the cardinality of $ \mathcal{G}_{N}\cap U(\mathbb{Z}_{N})$ to know which group contains it, we have to do calculations other than the criteria used here. In particular, we will determine such groups in the case of $N$ composite number.

\quad

Given $H= \mathcal{G}_{N}\cap U(\mathbb{Z}_{N})$ we have that the action of $H$ on $\mathcal{A}_{N}\cap U(\mathbb{Z}_{N})$ and on $\overline{\mathcal{A}_{N}}$ is free; for which $|H| | (|\mathcal{A}_{N}\cap U(\mathbb{Z}_{N})|, |\overline{\mathcal{A}_{N}}|)$ with $ \phi(N)=|\mathcal{A}_{N}\cap U(\mathbb{Z}_{N})|+|\overline{\mathcal{A}_{N}}|$, and this is possible since we remember that $N$ is a composite number.

\quad

Combining propositions (7.4) and  (7.6) with Criterion (8.8) and with what we have seen so far, we obtain for an even number $N$ of the type $N=2^{k}\overline{N}$ with $k \ge 1$ and $\overline{N}\ge 1$ with $\overline{N}$ odd the following equalities.

\begin{equation}
	For \hspace{0.1cm} k>1 \hspace{0.2cm}\mathcal{G}_{2^{k}\overline{N}} \leq <T_{2^{k-1}\overline{N}}f_{1+2^{k-1}\overline{N}}> \times (< T_{\frac{2^{k}\overline{N}}{(|\mathcal{A}_{2^{k}\overline{N}}\cap Orb_{D_{2^{k-1}\overline{N}}}(1) |, |\overline{ \mathcal{A}_{2^{k}\overline{N}} } |)}}  	>\rtimes H)  \hspace{0.2cm} with \hspace{0.2cm} H= \mathcal{G}_{2^{k}\overline{N}}\cap U(\mathbb{Z}_{2^{k}\overline{N}}) 
\end{equation}

$$ if \hspace{0.1cm}	T_{2^{k-1}\overline{N}},\hspace{0.1cm} f_{1+2^{k-1}\overline{N}}\hspace{0.1cm} \notin \mathcal{G}_{2^{k}\overline{N}}$$

\begin{equation}
	For \hspace{0.1cm} k>1 \hspace{0.2cm}\mathcal{G}_{2^{k}\overline{N}} \leq  (< T_{\frac{2^{k}\overline{N}}{(|\mathcal{A}_{2^{k}\overline{N}}\cap Orb_{D_{2^{k-1}\overline{N}}}(1) |, |\overline{ \mathcal{A}_{2^{k}\overline{N}} } |)}}  	>\rtimes H)  \hspace{0.2cm} with \hspace{0.2cm} H= \mathcal{G}_{2^{k}\overline{N}}\cap U(\mathbb{Z}_{2^{k}\overline{N}}) 	
\end{equation}
$$ if \hspace{0.1cm}	T_{2^{k-1}\overline{N}},\hspace{0.1cm} f_{1+2^{k-1}\overline{N}}\hspace{0.1cm} \in \mathcal{G}_{2^{k}\overline{N}}$$

\begin{equation}
	for \hspace{0.1cm} k=1 \hspace{0.2cm}\mathcal{G}_{2\overline{N}} \leq < T_{\frac{2\overline{N}}{(|\mathcal{A}_{2\overline{N}}\cap Orb_{D_{\frac{2\overline{N}}{2}}}(1) |, |\overline{ \mathcal{A}_{2\overline{N}} } |)}} 	>\rtimes H  \hspace{0.2cm} with \hspace{0.2cm} H= \mathcal{G}_{2\overline{N}}\cap U(\mathbb{Z}_{2\overline{N}}),
\end{equation}	

with $|H| \mid (|\mathcal{A}_{2^{k}\overline{N}}\cap U(\mathbb{Z}_{2^{k}\overline{N}}) |,|\overline{ \mathcal{A}_{2^{k}\overline{N}} } | ) \hspace{0.1cm} \forall k >0$ and $\overline{N}$ a composite number.

\quad
We now introduce other types of even numbers $N$ for which surely the associated dihedral sieve $\mathcal{C}_{N}$ is Goldbach's.

\begin{definition}
	An even number $N$ that admits in its dihedral sieve $\mathcal{C}_{N}$ a single orbit of the type $Orb_{D_{\bigr[\frac{N}{q}\bigl]} }(2q)$ with $q \nmid N$ and $q<[\sqrt{N}]$ is called $Mono- Orbital$ or in short $m. o$.
\end{definition}
For such even numbers $N$ the following property applies.

\begin{prop}
For an even $m.o$ number $N$ the sieve $\mathcal{C}_{N}$ is Goldbach's.	
\end{prop}
\begin{proof}
	Assume absurdly that this is not the case i.e., that $\overline{ \mathcal{A}_{N}}=\emptyset$ from this it follows that $Orb_{D_{\bigr[\frac{N}{q}\bigl]} }(2q) \supset U(\mathbb{Z}_{N})$ then $U(\mathbb{Z}_{N})=<q>$ ; in fact if $\nu \in U(\mathbb{Z}_{N})$ this implies that $\nu=wq^{t}$ for an appropriate $t>0$ and $w \in U(\mathbb{Z}_{N})$. If $w=1$ we have done otherwise if $w \neq 1$ we can assume , unless of rearranging the factors of $\nu$, that $t$ is the maximal integer for which $\nu$ can be decomposed as a power of $q$ but this is absurd, since $w$ can also be factorized as a power of $q$ and an appropriate invertible $u$ . For which $U(\mathbb{Z}_{N})$ is cyclic this implies that $N=2,4, 2p^{k}$ for $k>0$ with $p$ odd prime. Excluding $2,4$ where trivially it is verified that the dihedral sieve is Goldbach's, we are in the case $N=2p^{k}$. Then there exist only two primes $p$ and $q$ less than $\sqrt{N}$ in addition to $2$ and from here we get the ultimate absurdity that proves the thesis, since the only cases would be in the sequence $1,2,3,4,5,6$ where the primes in question are $3,5$ and for a $k>\log_{3}(18)$ or $k>\log_{5}(18)$ to get $[\sqrt{N}]>6$. In the case $k\le \log_{3}(18)$ and $k \le \log_{5}(18)$ it is obvious that the sieve is Goldbach's.

\end{proof}

\quad

Now let's look at some examples of mono-orbital numbers abbreviated as $m.o$.

\begin{example}
It can be seen in an elementary way that the numbers $90, 120$ are $m.o$. Even if we had not verified this condition, however, it should be said that it is obvious anyway that the respective sieves are dihedral; they are very small numbers.
It is probably interesting to note is that since $120$ is divisible by $4$ by Proposition (7.4), the group $\mathcal{G}_{120}\cong V$ and instead since $90$ is only divisible by $2$, then $\mathcal{G}_{90}\cong \mathbb{Z}_{2}$. After a series of relatively simple smaller calculations on numbers smaller than 120 using the criteria or methods of calculation used so far, it turns out that even $N$ numbers that are not $ \mathcal{G}_{N}$-cyclotomic all enjoy this property. This  will be discussed later.

\end{example}

\quad

Thus we have seen two types of even compound numbers that surely have a dihedral sieve of goldbach. Having similar characteristics i.e., few orbits of the type $Orb_{D_{\bigr[\frac{N}{q}\bigl]} }(2q)$ specifically for a $ \mathcal{G}_{N}$-cyclotomic has no orbit while for an $m.o$ number we have exactly one orbit. From this observation it may be useful to combine these two types of numbers into one definition.
	
\quad

\begin{definition}
An even number that has at most one orbit of the type $Orb_{D_{\bigr[\frac{N}{q}\bigl]} }(2q)$ is called $Quasi-Mono-Orbital$ for short $q.m.o$.
\end{definition}	

\quad

The following property of such numbers can be easily derived
\begin{prop}
	Given an even number $q.m.o$ of the type $2^{k}\overline{N}$ with $k>1$ and $\overline{N}$ odd then $2^{j}\overline{N}$ is $q.m.o$ with $1<j<k$.
\end{prop}
\quad
We see now that there is a relationship between the elements of $T_{d}$ and $f_{\nu}$ belonging to $\mathcal{G}_{N}$.

\begin{prop}
If $T_{2d} \in \mathcal{G}_{N}$ then $f_{1+2dt} \in \mathcal{G}_{N}\cap U(\mathbb{Z}_{N})\hspace{0.2cm} \forall t \in \{ 1, ...,t,..., \frac{N}{2d}-1 \}$.
\quad
If $f_{1+2dt} \notin \mathcal{G}_{N}\cap U(\mathbb{Z}_{N})$ then $T_{2d}\notin \mathcal{G}_{N}$ for an opportune $d \mid N$.
\end{prop}

\begin{proof}
	If there exists $T_{2d} \in \mathcal{G}_{N}$ then for each $q_{i} \in \overline{\mathcal{A}_{N}}$ we have that there exists a $q_{j}  \in \overline{\mathcal{A}_{N}}$ such that $q_{i}(1+2d)=q_{j}$. In fact, $q_{i}(1+2d)= q_{i}+2 q_{i}d=q_{j}$ since by assumption $T_{2d} \in \mathcal{G}_{N}$. Hence $(1+2d)=q_{j}q_{i}^{-1}$ and thus $(1+2d) \in U(\mathbb{Z}_{N})$ and it is easy to deduce this for $(1+2td)$ with $t \in  \{ 1, ...,t,..., \frac{N}{2d}-1 \}$.
		This implication just proved is a necessary condition of the fact that there exists a $T_{2d} \in \mathcal{G}_{N}$ but not sufficient. However, we can observe that if $f_{1+2dt} \notin \mathcal{G}_{N}$ surely $T_{2dt} \notin \mathcal{G}_{N}$, which implies that $T_{2d} \notin \mathcal{G}_{N}$.

\end{proof}	

\quad

By means of Proposition 8.14, the following property can be shown.

\begin{prop}
	If $\mathcal{G}_{2^{k}\overline{N}}\cap U(\mathbb{Z}_{2^{k}\overline{N}})=<f_{-1}>$ with $\overline{N}$ odd, then$$\mathcal{G}_{2^{k}\overline{N}}=\begin{cases} <f_{-1}> \hspace{0.2cm} if \hspace{0.2cm} k=1 \\<T_{2^{k-1}\overline{N}}f_{1+2^{k-1}\overline{N}}>\times <f_{-1}> \hspace{0.2cm} otherwise \end{cases}.$$
\end{prop}
\begin{proof}

If $k=1$ it suffices to show by Proposition 7.6 that $ \mathcal{G}_{2^{k}\overline{N}}^{(1)}=<\mathbb{I}_{\mathbb{Z}_{2^{k}\overline{N}}}>$. If not then there exists a $T_{2d} \in  \mathcal{G}_{2\overline{N}}^{(1)}$ where $d \neq \pm 1$ with $d\geq 2 \hspace{0.2cm} d \mid 2\overline{N}$. Then by 8.14 we have that $f_{1+2d} \neq f_{-1} \in \mathcal{G}_{2\overline{N}}$ hence the absurdity against the assumptions . The case $k>1$ is similar in proof and in addition 7.4 applies.
\end{proof}

\quad

Now let us see some consequences of the fact that the group $ \mathcal{G}_{N}^{(1)}$ is nontrivial. First we begin by defining the following set $$\overline{\mathcal{A}_{m,N}}:=\{ q_{l} \in \mathbb{Z}_{N} \hspace{0.2cm} prime\hspace{0.2cm} or \hspace{0.2cm}1\hspace{0.2cm}  and \hspace{0.2cm}  q_{l}<m \hspace{0.2cm}in \hspace{0.2cm}\mathbb{Z}  \hspace{0.2cm}|\hspace{0.2cm}  m-q_{l} \hspace{0.2cm} prime \hspace{0.2cm}or \hspace{0.2cm}1\hspace{0.2cm} and \hspace{0.2cm}  m-q_{l}<m \hspace{0.2cm}in \hspace{0.2cm}\mathbb{Z} \},$$ 
 with $m$ even and $m<N$.
 
 It can be easily seen that there is a biunivocal correspondence between $\overline{\mathcal{A}_{m,N}}$ and $\overline{\mathcal{A}_{m}}$ in such a way
 
 $$\overline{\mathcal{A}_{m,N}} \longrightarrow \overline{\mathcal{A}_{m}}  $$
 $$ \hspace{0.4cm} q_{l} \mapsto q_{l}$$
 $$  m- q_{l} \mapsto -q_{l}.$$
	
\quad

We now have the necessary elements to see what consequences we need to deal with.

\begin{prop}
If there exists a $T_{2d\alpha}  \in \mathcal{G}_{N}^{(1)}$ with $2d \mid N$ then we have that $\overline{\mathcal{A}_{2d\alpha,N}}\cap \overline{\mathcal{A}_{N}} \neq \emptyset$ and in particular such a set is symmetrical for $T_{2d\alpha}f_{-1}.$	
\end{prop}	
\begin{proof}
	I consider an element $q_{i} \in \overline{\mathcal{A}_{N}}$ and $q_{i}>2d\alpha$ surely we have by hypothesis that $q_{i} =2d\alpha s + r_{i}$ where $r_{i}$ is a prime or 1 and $r_{i}<2d\alpha $ in $\mathbb {Z}$ for an appropriate integer $s$. If, on the other hand, $q_{i}<2d\alpha$ there would surely have existed another $q_{j}>2d\alpha$ and the role of $r_{i}$ would have been assumed by $q_{i}$. So we can safely assume without losing the generality of the proof that $q_{i}>2d\alpha$ due to the hypothesis. 
	So we have that 
	$$2d\alpha-r_{i}=2d\alpha-(q_{i}-2d\alpha s)< 2d\alpha\hspace{0.1cm} in \hspace{0.1cm} \mathbb {Z},$$
	
	$$ 2d\alpha-r_{i}=  2d\alpha- q_{i}+2d\alpha s=2d\alpha(1+s)- q_{i}= 2d\alpha(1+s)+q_{j}.                                                     $$
	Where $q_{j}=-q_{i}$ and $q_{j} \in ´\overline{\mathcal{A}_{N}}$ being $ q_{i} \in ´\overline{\mathcal{A}_{N}} $. Then $2d\alpha-r_{i}$ is a prime number or 1 and is less than $2d\alpha$ in $\mathbb{Z}$ for which we obtained the thesis. 
	The symmetry of the set$\overline{\mathcal{A}_{2d\alpha,N}}\cap \overline{\mathcal{A}_{N}}$ for the element $T_{2d\alpha}f_{-1}$ is immediate .
	
\end{proof}

\quad

As for Proposition 8.16, the reverse is not always true. In fact, if we have that $ \overline{\mathcal{A}_{2}}=\{ 1\}$ and $1 \in \overline{\mathcal{A}_{12}}$ we then have that $\overline{\mathcal{A}_{2, 12}}\cap\overline{\mathcal{A}_{12}}\neq \emptyset$ and is symmetric with respect to $T_{2}f_{-1}$, but $T_{2}(1)=3 \notin \overline{\mathcal{A}_{12}}.$

\quad
Let us see what sufficient conditions lead to the fact that a given element of $T_{2d\alpha} \notin \mathcal{G}_{N}^{(1)}.$

\begin{prop}
	Given a $2d \mid N$, we have that if 
	$\overline{\mathcal{A}_{2d\alpha,N}}\cap \overline{\mathcal{A}_{N}} = \emptyset$ then $T_ {2d\alpha} \notin \mathcal{G}_{N}^{(1)}$. If $\overline{\mathcal{A}_{2d\alpha,N}}\cap \overline{\mathcal{A}_{N}}\neq \emptyset$ the same consequence comes if that set is not symmetrical to $T_{2d\alpha}f_{-1}.$
\end{prop}

\quad

We have just seen that if $T_{2d\alpha} \in \mathcal{G}_{N}^{(1)}$ then to each $q_{i} \in \overline{\mathcal{A}_{N}}$ we can associate with it some $r_{j}, 2d\alpha -r_{j} \in \overline{\mathcal{A}_{2d\alpha,N}} \cap \overline{\mathcal{A}_{N}}$ and this implies that $\overline{\mathcal{A}_{N}} \subset \bigcup_{j=0}^{\frac{N}{2d}-1}T_{2d\alpha j} (\overline{\mathcal{A}_{2d\alpha,N}}\cap \overline{\mathcal{A}_{N}}).$
		
 And this is easy to verify because of what we have just seen; the opposite inclusion is trivial, so the following proposition holds
 
 \begin{prop}
 Let $\alpha$ be such that $(\alpha,N)=1$. 	
 If $T_{2d\alpha } \in \mathcal{G}_{N} \rightarrow \overline{\mathcal{A}_{N}} = \bigcup_{j=0}^{\frac{N}{2d}-1}T_{2d\alpha j} (\overline{\mathcal{A}_{2d\alpha,N}}\cap \overline{\mathcal{A}_{N}}).$
 \end{prop}		

\quad
In truth, the reverse is true.

\begin{prop}
Given a $\alpha$ such that $(\alpha,N)=1$ and $2d \mid N$. Then  $$T_{2d\alpha } \in \mathcal{G}_{N} \iff \overline{\mathcal{A}_{N}} = \bigcup_{j=0}^{\frac{N}{2d}-1}T_{2d\alpha j} (\overline{\mathcal{A}_{2d\alpha,N}}\cap \overline{\mathcal{A}_{N}})$$
\end{prop}
\begin{proof}
	Given  $ \overline{\mathcal{A}_{N}} = \bigcup_{j=0}^{\frac{N}{2d}-1}T_{2d\alpha j} (\overline{\mathcal{A}_{2d\alpha,N}}\cap \overline{\mathcal{A}_{N}})$ we have that, for some $j^{'}$ , $$T_{2d\alpha j^{'}}(\overline{\mathcal{A}_{N}})=\bigcup_{j=0}^{\frac{N}{2d}-1}T_{2d\alpha (j+j^{'})} (\overline{\mathcal{A}_{2d\alpha,N}}\cap \overline{\mathcal{A}_{N}})=\bigcup_{j=0}^{\frac{N}{2d}-1}T_{2d\alpha j} (\overline{\mathcal{A}_{2d\alpha,N}}\cap \overline{\mathcal{A}_{N}})$$
	
	i.e.
	
	$$T_{2d\alpha j^{'}}(\overline{\mathcal{A}_{N}})=\overline{\mathcal{A}_{N}}.$$

\end{proof}

\quad

We have from Proposition 8.19 the following corollaries

\begin{coro}
Given a $\alpha$ such that $(\alpha,N)=1$ and $2d \mid N$.	$$T_{2d\alpha } \in \mathcal{G}_{N}\rightarrow \overline{\mathcal{A}_{N}}=\bigcup_{l=0}^{\frac{N}{(2d\alpha j, N)}-1}T_{2d\alpha jl} (\overline{\mathcal{A}_{2d\alpha j,N}}\cap \overline{\mathcal{A}_{N}}) \hspace{0.2cm} \forall j \in \{ 0,1,...,\frac{N}{2d}-1 \}.$$
\end{coro}

\quad

\begin{coro}
Given a $\alpha$ such that $(\alpha,N)=1$ and $2d \mid N$.
$$<T_{2d\alpha}	> <  \mathcal{G}_{N} \iff \overline{\mathcal{A}_{N}}=\bigcup_{l=0}^{\frac{N}{(2d\alpha j, N)}-1}T_{2d\alpha jl} (\overline{\mathcal{A}_{2d\alpha j,N}}\cap \overline{\mathcal{A}_{N}}) \hspace{0.2cm} \forall j \in \{ 0,1,...,\frac{N}{2d}-1 \}.$$
\end{coro}

\quad

Returning to Proposition 8.17, this should be understood as a criterion for determining whether there exist $T_{2d\alpha} \in \mathcal{G}_{N}$ and by negating the thesis of the last two corollaries, we can enrich criterion with other possibilities. For example, if $\overline{\mathcal{A}_{N}} \neq \bigcup_{j=0}^{\frac{N}{2d}-1}T_{2d\alpha j} (\overline{\mathcal{A}_{2d\alpha, N}}\cap \overline{\mathcal{A}_{N}})$ this may imply that $\bigcup_{j=0}^{\frac{N}{2d}-1}T_{2d\alpha j} (\overline{\mathcal{A}_{2d\alpha, N}}\cap \overline{\mathcal{A}_{N}}) \subset \overline{\mathcal{A}_{N}} $ or it may imply that there exists some $q_{i} \in \overline{\mathcal{A}_{2d\alpha,N}} \cap \overline{\mathcal{A}_{N}}$ such that $T_{2d\alpha j}(q_{i} )=\pm m$ with $m$ a multiple in $\mathbb{Z}$ for an appropriate $j$.

\quad
So we have the following criterion

\begin{prop}
	
Given a $\alpha$ such that $(\alpha,N)=1$ and $2d \mid N$, 	$T_{2d\alpha } \notin \mathcal{G}_{N}$ if
\begin{enumerate}
	\item $\overline{\mathcal{A}_{2d\alpha,N}}\cap \overline{\mathcal{A}_{N}} = \emptyset$ or
	
	\item  $\overline{\mathcal{A}_{2d\alpha,N}}\cap \overline{\mathcal{A}_{N}} \neq \emptyset$ but it is not symmetrical with respect to  $T_{2d\alpha }f_{-1 }$ or
	
	\item If  (2) holds and $\bigcup_{j=0}^{\frac{N}{2d}-1}T_{2d\alpha j} (\overline{\mathcal{A}_{2d\alpha, N}}\cap \overline{\mathcal{A}_{N}}) \subset \overline{\mathcal{A}_{N}} $ or
	
	\item  $\overline{\mathcal{A}_{N}}\neq \bigcup_{j=0}^{\frac{N}{2d}-1}T_{2d\alpha j} (\overline{\mathcal{A}_{2d\alpha, N}}\cap \overline{\mathcal{A}_{N}}).$
\end{enumerate}
\end{prop}

\quad

We begin by applying this criterion on a property already demonstrated several times i.e. $\mathcal{G}_{2p}^{(1)}=<\mathbb{I}_{\mathbb{Z}_{2p}}>$ for prime numbers $p>3$. We know that if $\mathcal{ G }_{2p}^{(1)}\neq<\mathbb{I}_{\mathbb{Z}_{2p}}>$ then $\mathcal{ G }_{2p}^{(1)}=<T_{2}>$.
Then we apply the criterion on $T_{2}$; if $1 \notin \overline{\mathcal{A}_{2p}} $ we're done otherwise we have that  $\overline{\mathcal{A}_{2,2p}}\cap \overline{\mathcal{A}_{2p}} \neq \emptyset$ 
and it is quite simple to verify that such a set is symmetric with respect to $T_{2}f_{-1 }$, but we will have that $T_{8}(1)=9$. Thus, we are left with (3) and (4) in Criterion (8.22), but since $2p \geq 10$ it is very easy to verify that $\overline{\mathcal{A}_{2p}}\neq \bigcup_{j=0}^{p-1}T_{2 j} (\overline{\mathcal{A}_{2, 2p}}\cap \overline{\mathcal{A}_{2p}}).$

\quad

Now let us apply Criterion (8.22) to the last proposition of this section

\begin{prop}
	For Goldbach sieves $\mathcal{C}_{N}$ with $N \geq 10$, it is verified that $T_{2 } \notin \mathcal{G}_{N}.$
\end{prop}
	\begin{proof}
		In the case $N=10$ we are in the equals of the type $N=2p$ with $p$ prime and in this case $p=10$ already proved above. So we see this for $N>10$; if $1 \notin \overline{\mathcal{A}_{N}}$ then it is obvious otherwise we are in the case $\overline{\mathcal{A}_{2, N}}\cap \overline{\mathcal{A}_{N}}\neq \emptyset$ which in this case is symmetric with respect to $T_{2 }f_{-1 }$ but we have that $T_{8}(1)=9$ and so we have the thesis.
	\end{proof}
\quad

Now let's look at examples where we can apply criterion (8.22). In the first example we will analyze the cases $N<10$ not covered in Proposition 8.23.

\begin{example}
	For $N=2$ we have that $\overline{\mathcal{A}_{2}}=\{ 1 \}$, $\mathcal{A}_{2}=\{ 0 \}$, and furthermore we have that $<T_{2}>=<\mathbb{I}_{Z_{2}}>$. Therefore trivially $T_{2} \in \mathcal{G}_{2}$.
	\quad
	
	For $N=4$ we have that $\overline{\mathcal{A}_{4}}=\{ 1, 3 \}$, furthermore $T_{2}(1)=3$, $T_{2}(3)=5=1$, and furthermore $<T_{2}>=\{\mathbb{I}_{Z_{4}}, T_{2} \}$. Therefore $T_{2} \in \mathcal{G}_{4}$.
	
	\quad 
	For $N=6$ we have that $\overline{\mathcal{A}_{6}}=\{ 1,5, 3 \}$, $T_{2}(1)=3$, $T_{4}(1)=5$ so $T_{2} \in \mathcal{G}_{6}$.
	\quad
	And we have arrived at the limiting case $N=8$. Here we have that $\overline{\mathcal{A}_{8}}=\{ 1, 3, 5,7 \}$, $T_{2}(1)=3$, $T_{2}(3)=5$, $T_{2}(5)=7$, $T_{2}(7)=1$ so $T_{2} \in \mathcal{G}_{8}.$
	
\end{example}

\begin{example}
	In this last example we exploit a fact that can be easily deduced for even numbers that are powers of $2$ and that is if $T_{2^{k-1} } \notin \mathcal{G}_{2^{k} } \rightarrow \mathcal{G}_{2^{k} }^{(1)}=<\mathbb{I}_{\mathbb{Z}_{2^{k}} }>$ for $k>0$.
	In particular we apply it for $k=6$ i.e. $N=128$. It is verified that $\overline{\mathcal{A}_{128}}=\{ 19, 31, 61,109,97,67 \}$ while $\overline{\mathcal{A}_{64,128}}=\{ 1, 3, 5,11,17,23,63,61,59,53,47, 41 \}$ , then $\overline{\mathcal{A}_{128}}\cap\overline{\mathcal{A}_{64,128}}=\{ 61\}$ the latter set is not symmetrical with respect to $T_{64 }f_{-1 }$ . Hence by Criterion 8.22 we have that $T_{64} \notin \mathcal{G}_{128 } \rightarrow \mathcal{G}_{128}^{(1)}=<\mathbb{I}_{\mathbb{Z}_{ 128} }>$.
\end{example}

\quad
\section{Final conclusions and conjectures}
\quad
Calculating let's say by hand without the help of calculators, it can be seen always for small numbers that for even $N$ numbers, it can be seen that if $N$ is not $\mathcal{G}_{N}$-cyclotomic then it will be that $$\mathcal{G}_{N}\cong \begin{cases} V \hspace{0.2cm}if \hspace{0.2cm}4 \mid N \\ \mathbb{Z}_{2} \hspace{0.2cm} otherwise \end{cases}$$
\quad
We can conjecture that this is true for every even $N$ number that is not $\mathcal{G}_{N}$-cyclotomic i.e.

\begin{prop}
\quad\quad
	
\title{STRONG CONJECTURE}.	

\quad\quad
For every even number $N$ that is not $\mathcal{G}_{N}$-cyclotomic we have that
 $$\mathcal{G}_{N}\cong \begin{cases} V \hspace{0.2cm}if \hspace{0.2cm}4 \mid N \\ \mathbb{Z}_{2} \hspace{0.2cm} otherwise \end{cases}.$$
\end{prop}
\quad
It would immediately come to say that if the Strong Conjecture applies then Goldbach's Strong Conjecture applies. But care must be taken; Goldbach's sieve, as we have defined it (see def (4.5)), also contemplates the $1$, $N-1$ case where $N-1$ is a prime number, and should it happen that $\overline{\mathcal{A}_{N}}=\{1,N-1\}$ Goldbach's strong conjecture would not be satisfied. Therefore it is more correct to state that

\begin{prop}
	If the $STRONG\hspace{0.2cm} CONJECTURE$ holds for some $N$ such that $ \{1,N-1\}\subset \overline{\mathcal{A}_{N}}$ or $ 1 \notin \overline{\mathcal{A}_{N}}$ then for $N$ Goldbach's strong conjecture holds .
\end{prop}
\quad
The strong conjecture hypothesis perhaps is very stringent, we could conjecture weaken the claims even though the conjectured claims will turn out to be of a remarkable magnitude namely

\begin{prop}
	\quad\quad
	
	\title{WEAK CONJECTURE}

	\quad\quad
	For every even $N>2$ we have that $\mathcal{G}_{N}\neq Aff(\mathbb{Z}_{N})$(see Corollary (3.24)).
\end{prop}	

 The weak conjecture is a necessary condition for Goldbach's strong conjecture i.e.

\quad

\begin{prop}
	If Goldbach's strong conjecture is true for an even number $N>2$ then the $WEAK\hspace{0.2cm} CONJECTURE$ is true.
\end{prop}

\quad
While for the sufficient condition we need to make the necessary clarifications as we did for the strong conjecture.

\begin{prop}
If the $WEAK\hspace{0.2cm} CONJECTURE$ holds for some $N$ such that $ \{1,N-1\}\subset \overline{\mathcal{A}_{N}}$ or $ 1 \notin \overline{\mathcal{A}_{N}}$ then for $N$ Goldbach's strong conjecture holds.
\end{prop}

\end{document}